\documentclass[3p]{elsarticle}

\usepackage{amsfonts, amsthm,amsmath,amssymb}  

\usepackage{tikz-cd}
\usepackage{hyperref} 
\hypersetup{breaklinks=true,colorlinks,allcolors=black} 
\numberwithin{equation}{section} 
\usepackage{appendix}

\newtheorem{theorem}{Theorem}[section]
\newtheorem{assumption}{Assumption}
\newtheorem{lemma}[theorem]{Lemma}
\newtheorem{proposition}[theorem]{Proposition}
\newtheorem{corollary}[theorem]{Corollary}
\newtheorem{definition}[theorem]{Definition}
\theoremstyle{remark} 
\newtheorem{remark}[theorem]{Remark}

\journal{ArXiv}

\begin{document}

\begin{frontmatter}
	
	\title{Stability Results for Bounded Stationary Solutions of Reaction-Diffusion-ODE Systems} 
	
	\author[1]{Chris Kowall} 
	\ead{kowall@math.uni-heidelberg.de}
	\author[1]{Anna Marciniak-Czochra} 
	\ead{anna.marciniak@iwr.uni-heidelberg.de} 
	\author[3]{Finn M\"unnich}
	\ead{finn.muennich@stud.uni-heidelberg.de}


\address[1]{Institute of Mathematics and IWR, Heidelberg University,
	Im Neuenheimer Feld 205, 69120 Heidelberg, Germany}
\address[3]{Institute of Mathematics, Heidelberg University,
	Im Neuenheimer Feld 205, 69120 Heidelberg, Germany}

\begin{abstract}
	Reaction-diffusion equations coupled to ordinary differential equations (ODEs) may exhibit spatially low-regular stationary solutions. This work provides a comprehensive theory of asymptotic stability of bounded, discontinuous or continuous, stationary solutions of reaction-diffusion-ODE systems. We characterize the spectrum of the linearized operator and relate its spectral properties to the corresponding semigroup properties. Considering the function spaces $L^\infty(\Omega)^{m+k}, L^\infty(\Omega)^m \times C(\overline{\Omega})^k$ and $C(\overline{\Omega})^{m+k}$, we establish a sign condition on the spectral bound of the linearized operator, which implies nonlinear stability or instability of the stationary pattern. 
\end{abstract}

\begin{keyword}
	Reaction-diffusion equation \sep  stationary solution \sep stability \sep instability \sep asymptotic behavior \sep semigroup theory 
\end{keyword}

\end{frontmatter}


\section{Introduction} \label{sec:intro}

Reaction-diffusion-type equations form a class of pattern formation models arising from ecological and biological applications. Interaction of diffusive and non-diffusive processes leads to reaction-diffusion equations coupled with ordinary differential equations (reaction-diffusion-ODE models)  \cite{He, Klika, Marasco, Takagi21, White}. The lack of diffusion may cause emergence of bounded model solutions with low spatial regularity and, in particular, patterns with jump-discontinuities in space \cite{HMCT, He, KMCT20}. Analysis of the pattern formation phenomena usually focuses on existence and asymptotic stability of  spatially heterogeneous stationary solutions.  
While stability of linearized systems is well-established in the case of reaction-diffusion-ODE systems, the passage to nonlinear stability is accomplished only for specific models or under restriction of the regularity of the stationary solution \cite{CMCKSregular, KMCT20, MKS17, Takagi21, Wang}. To close this gap, this paper concerns nonlinear stability and nonlinear instability of bounded stationary solutions of reaction-diffusion-ODE systems, accounting for both, discontinuous and continuous solutions. Among others, the established general approach allows applications to systems with spatially heterogeneous parameters arising from modeling of processes in heterogeneous habitats  \cite{Cantrell, He, Takagi21}.  

\subsection*{Mathematical challenges}

Linearization at a stationary solution of a reaction-diffusion-ODE system leads us to an unbounded linear (degenerated differential) operator and a nonlinear remainder on an infinite-dimensional function space. Usually, the linear operator is well-studied on a function space such as $L^p(\Omega)^{m+k}, 1 < p < \infty,$ which is less regular than the domain of the nonlinear part, where the latter may even neither be well-defined nor continuous in the less regular space \cite{Friedlander, Henry}. In such infinite-dimensional cases, linear stability does not necessarily reflect the properties of the nonlinear model. There exist examples of  linear exponentially stable stationary solutions that become nonlinearly unstable \cite{Zwart}, and conversely, of linear exponentially unstable stationary solutions stabilized by the nonlinearity \cite{Rodrigues}.

In case of bounded stationary solutions of a reaction-diffusion-ODE system, there exists an instability result on $L^\infty(\Omega)^{m+k}$ using a spectral gap of the linear operator studied on $L^p(\Omega)^{m+k}$, see \cite{MKS17} based on results of \cite{Friedlander}. Since in the case of reaction-diffusion-ODE systems the spectrum of the linear operator might not solely consist of discrete eigenvalues, such a spectral gap condition is not always fulfilled. Moreover, classical nonlinear stability results as in \cite{Shatah, Webb} only apply in the case where the linear and the nonlinear operator is studied on the same function space. Consequently, in the case of a continuous steady state, the choice of the space $C(\overline{\Omega})^{m+k}$ is expedient for stability and instability results.
In the case of jump-discontinuous stationary solutions, analysis of the problem takes place on $L^\infty(\Omega)^{m+k}$.
While the nonlinearity is Fr{\'e}chet differentiable on that space, the study of the linear operator is more delicate. The linear operator does not generate a strongly continuous semigroup on $L^\infty(\Omega)^{m+k}$, however, this semigroup still features analytic properties for $t>0$. A priori, it is not clear how to link the spectral properties of the linear operator to the growth properties of the corresponding semigroup. As shown in this paper,  the growth of the semigroup can still  be controlled by the spectral bound of the linear operator.
We adapt the works \cite{Shatah, Webb}, using continuity of solutions to the nonlinear problem for $t>0$, to finally accomplish the transfer from linear to nonlinear stability and instability.

\subsection*{State of the art}

Spectral analysis of a linear reaction-diffusion-ODE operator is presented in \cite{CMCKSregular, MKS17} in the case of a system of ODEs coupled to one reaction-diffusion equation. It is well-known that the spectrum consists of a (possibly non-discrete) part determined by a multiplication operator corresponding to the ODE subsystem and a remainder of eigenvalues of the full linear operator. In the presence of a spectral gap and a positive spectral bound of the linearized operator, \cite{MKS17} 
derives nonlinear instability of discontinuous, bounded steady states. 
Instability of all sufficiently regular non-constant stationary solutions of certain reaction-diffusion-ODE systems is shown in \cite{CMCKSregular, Rendall, KMCT20, MKS17, JWang19}, caused by a positive spectral bound of a linearized operator. This instability result seems to be an effect of the coupling of the ODEs to a single reaction-diffusion equation in a  homogeneous environment. It is well-known that a scalar reaction-diffusion equation on a convex domain features a similar instability result for homogeneous environments, see references in \cite{CMCKSregular}. However, in heterogeneous environments, there exist continuous non-constant stationary solutions of reaction-diffusion-ODE systems which are stable \cite{He, Takagi21}. In accordance with classical reaction-diffusion systems, a similar stabilizing effect is possible when ODEs are coupled to a system of reaction-diffusion equations. 

Asymptotic stability of constant stationary solutions is considered in \cite{Wang, Xu}, based on a linearization, and in \cite{Rendall, Hattaf, JWang21}, based on the knowledge of a specific Lyapunov function.  Asymptotic stability is also studied in \cite{He} for constant and non-constant stationary solutions of a reaction-diffusion-ODE model with heterogeneous model parameters. Furthermore, concerning exponential stability, explicit error estimates are obtained for certain reaction-diffusion-ODE models in \cite{HMCT, KMCT20, Takagi21}. The most recent nonlinear stability result for a jump-discontinuous bounded steady state of a reaction-diffusion-ODE system using linearization has been established using explicit estimates on the space $L^\infty(\Omega)^m \times C(\overline{\Omega})^k$ for $k=1$, under an additional stability assumption on the scalar diffusive model component, \cite{CMCKSirregular}.

\subsection*{Problem formulation}

The aim of this paper is a comprehensive nonlinear stability theory for bounded, continuous and jump-discontinuous,  solutions of coupled reaction-diffusion-ODE systems with spatially heterogeneous parameters,
\begin{align} \label{fullsys}
	\begin{split}
		\frac{\partial \mathbf{u}}{\partial t} & = \mathbf{f}( \mathbf{u} , \mathbf{v}, x) \quad \mbox{in} \quad \Omega_T, \qquad \mathbf{u}(\cdot, 0) =\textbf{u}^0 \quad \text{in} \quad \Omega,  \\
		\frac{\partial \mathbf{v}}{\partial t} - \mathbf{D}^v \Delta \mathbf{v} & =  \mathbf{g} ( \mathbf{u} , \mathbf{v} ,x) \quad \mbox{in} \quad \Omega_T, \qquad \mathbf{v}(\cdot, 0) =\textbf{v}^0 \quad \text{in} \quad \Omega,\\
		\frac{\partial \mathbf{v}}{\partial \mathbf{n}} & =\mathbf{0} \quad \mbox{on} \quad \partial \Omega \times (0,T),
	\end{split}
\end{align}
where $\mathbf{u}: \Omega_T \to \mathbb{R}^m$, $\mathbf{v}: \Omega_T \to \mathbb{R}^k, m,k \in \mathbb{N},$ are vector-valued functions on the domain $\Omega_T:=\Omega \times (0,T)$ endowed with zero Neumann boundary conditions for each diffusive component. The vector $\mathbf{n}$ denotes the outward unit normal on $\partial \Omega$. We consider a bounded domain $\Omega \subset \mathbb{R}^n$, i.e., an open and connected set, with a Lipschitz boundary $\partial \Omega \in C^{0,1}$ and $n \in \mathbb{N}$. The Laplace operator $\Delta$ is defined componentwise and is multiplied by a diagonal matrix $\mathbf{D}^v \in  \mathbb{R}_{> 0}^{k \times k}$ with positive entries. The model nonlinearities 
\begin{align*}
	\mathbf{f}: \mathbb{R}^{m+k} \times \overline{\Omega}  \to \mathbb{R}^m \quad \text{and} \quad \mathbf{g}: \mathbb{R}^{m+k} \times \overline{\Omega} \to \mathbb{R}^k
\end{align*}
are given functions which may depend on the unknown solution and the space variable, further details are given in Assumption \ref{ass:N} below. Recall that solutions to reaction-diffusion-ODE system \eqref{fullsys} may feature low regularity in space, even in homogeneous environments \cite{Steffenthesis, dynspike}. Hence, mild solutions are considered in this work, similar to \cite{Rothe}. Accordingly,  we assume the following local Lipschitz condition

\begin{assumption}[Local Lipschitz condition] 
	\label{ass:Exist} Let the function $\mathbf{h} = (\mathbf{f}, \mathbf{g})(\mathbf{u}, \mathbf{v}, \cdot)$ be measurable in $x \in {\overline \Omega}$ for every fixed $(\mathbf{u},\mathbf{v}) \in \mathbb{R}^{m+k}$. Moreover, for every bounded set $B \subset \mathbb{R}^{m+k} \times \overline{\Omega}$, let there exist a constant $L(B)>0$ such that for all $(\mathbf{u}, \mathbf{v},x), (\mathbf{y}, \mathbf{z},x) \in B$ there holds
	\begin{align*} 
		|\mathbf{h}(\mathbf{u}, \mathbf{v}, x)| & \le L(B),\\
		|\mathbf{h}(\mathbf{u}, \mathbf{v}, x)-\mathbf{h}(\mathbf{y}, \mathbf{z}, x)| & \le L(B)\left( |\mathbf{u}- \mathbf{y}| + |\mathbf{v} - \mathbf{z}| \right).
	\end{align*} 
\end{assumption}

\noindent If Assumption \ref{ass:Exist} is satisfied by $(\mathbf{f}, \mathbf{g})$, existence and uniqueness of a mild solution is guaranteed by \cite[Part II, Theorem 1]{Rothe}. In view of the linearization procedure, we make the following stronger regularity assumption 

\begin{assumption}[Nonlinearity] \label{ass:N} Let the nonlinearities $\mathbf{f}, \mathbf{g}$ of system \eqref{fullsys} be once continuously differentiable with respect to the unknown variables $\mathbf{u}, \mathbf{v}$. Furthermore, we assume that the local Lipschitz condition (Assumption \ref{ass:Exist}) holds for $(\mathbf{f}, \mathbf{g})$ and its derivatives $\nabla_\mathbf{u} (\mathbf{f}, \mathbf{g}), \nabla_\mathbf{v} (\mathbf{f},\mathbf{g})$.
\end{assumption}

\noindent Assumption \ref{ass:N} is, for instance, satisfied by functions $\mathbf{f}, \mathbf{g}$ which are twice continuously differentiable with respect to the unknown variables $\mathbf{u}, \mathbf{v}$ and which do not depend explicitly on the space variable $x$.\\

\noindent In the case of a continuous steady state of system \eqref{fullsys}, we assume the following analog of Assumption \ref{ass:N}.

\begin{assumption}[Continuous case] \label{ass:Nc} Let Assumption \ref{ass:N} hold, based on the local Lipschitz condition (Assumption \ref{ass:Exist}) in which we replace the term \emph{measurable in $x \in \overline{\Omega}$} by the term \emph{continuous in $x \in \overline{\Omega}$}.
\end{assumption}

\subsection*{Main objectives of the work}

In this work, we consider a bounded stationary solution $(\overline{\mathbf{u}}, \overline{\mathbf{v}}) \in L^\infty(\Omega)^{m+k}$ of system \eqref{fullsys}, i.e., a bounded solution of the problem 
\begin{align} \label{steady}
	\begin{split}
		\mathbf{0} & = \mathbf{f}( \mathbf{u} , \mathbf{v},x) \quad \mbox{in} \quad \Omega, \\
		- \mathbf{D}^v \Delta \mathbf{v} & =  \mathbf{g} ( \mathbf{u} , \mathbf{v},x) \quad \mbox{in} \quad \Omega, \qquad  \frac{\partial \mathbf{v}}{\partial \mathbf{n}} =\mathbf{0} \quad \mbox{on} \quad \partial \Omega.
	\end{split}
\end{align}
Note that elliptic regularity from \cite{Nittka} implies that the component $\overline{\mathbf{v}}$ is at least H\"older continuous in $\overline{\Omega}$. We compare solutions of system \eqref{fullsys} with a solution of system \eqref{steady} with respect to the uniform topology induced by $L^\infty(\Omega)^{m+k}$ or $L^\infty(\Omega)^m \times C(\overline{\Omega})^k$ and, in the case of a continuous stationary solution $(\overline{\mathbf{u}}, \overline{\mathbf{v}}) \in C(\overline{\Omega})^{m+k}$, additionally with respect to the norm on $ C(\overline{\Omega})^{m+k}$. This work provides a unified approach to study nonlinear stability and instability of bounded steady states of system \eqref{fullsys} using linearization. 

While the spectrum is purely discrete in the case of classical reaction-diffusion systems, we may lose the spectral gap in the case of a reaction-diffusion-ODE system. Hence, results from \cite{MKS17} are not applicable to discontinuous stationary solutions in the absence of a spectral gap. 
Motivated by this observation, we show nonlinear instability in the case of a positive spectral bound of the linearized operator 
but regardless of a spectral gap condition. This result is achieved for systems of reaction-diffusion equations coupled to a system of ODEs, including discontinuous and continuous patterns.

Since the patterns emerging in reaction-diffusion-ODE models may be both, spatially jump-discontinuous and continuous stationary solutions, \cite{HMCT, He, KMCT20, Takagi21}, the second objective of the current study is to establish a rigorous nonlinear stability result that can be applied to the different classes of patterns. More precisely,  we show in different functional settings that a negative spectral bound of the linearized operator implies a local exponential stability result. 
The robust notion of stability allows considering discontinuous and continuous perturbations of the steady state. Hence, 
the results are applicable to study stability of continuous patterns of reaction-diffusion-type systems such as \cite{Anma, He, Klika, Marasco, Smith}, including classical reaction-diffusion systems, and jump-discontinuous patterns of reaction-diffusion-ODE systems \cite{HMCT, He, KMCT20, Takagi21, White}. We indicate possible generalizations of the results in Remark \ref{rem:general}.

\subsection*{Structure of the paper}

In Section \ref{sec:Lin}, we linearize the reaction-diffusion-ODE system \eqref{fullsys} at a bounded stationary solution and study the linear case, including spectral properties and semigroup properties on the spaces $L^\infty(\Omega)^{m+k}$, $L^\infty(\Omega)^m \times C(\overline{\Omega})^k$ and $C(\overline{\Omega})^{m+k}$. Within this study, we focus on the challenging case of a discontinuous stationary solution. Section \ref{sec:nonlinearstab} is devoted to nonlinear stability and instability results, while we apply these results to two model examples in Section \ref{sec:appl}. Basic results for the heat semigroup on $L^p(\Omega)$ and $C(\overline{\Omega})$, existence of mild solutions and spectral analysis of a matrix multiplication operator on uniformly continuous functions is deferred to the Appendix.

\section{Linearization} \label{sec:Lin}

Existing linear stability and instability principles, e.g., \cite[Ch. VII, Theorem 2.1]{Daleckii}, \cite[Theorem 5.1.1]{Henry}, \cite[Proposition 4.17]{Webb}, and \cite[Theorem 1]{Shatah}, strongly depend on the choice of the underlying function spaces, e.g., $L^p(\Omega)$ or $W^{1,p}(\Omega)$ as in \cite{
MKS17} or, in more abstract words, the domain of the fractional operator $(a I - \mathbf{D} \Delta)^\alpha$ for some $a, \alpha>0$ \cite[Definition 1.5.4]{Henry}.\\

By \cite{Steffenthesis, dynspike}, the ODE component of a stationary solution of problem \eqref{fullsys} may exhibit jump-discontinuities in space. Thus, we consider a steady state $(\overline{\mathbf{u}}, \overline{\mathbf{v}}) \in L^\infty(\Omega)^{m} \times C(\overline{\Omega})^{k}$. We consider a linearization of the reaction-diffusion-ODE system at a steady state $(\overline{\mathbf{u}}, \overline{\mathbf{v}})$,
\begin{align}
	\frac{\partial \xi}{\partial t} = \mathbf{L} \xi + \mathbf{N}(\xi), \qquad \xi(0) = \xi^0 \in L^\infty(\Omega)^{m+k} \label{LNequation}
\end{align}
for $\xi= (\mathbf{u}, \mathbf{v}) - (\overline{\mathbf{u}}, \overline{\mathbf{v}})$. Nonlinear stability of the stationary solution depends on the properties of the linear part $\mathbf{L}$ which induces a semigroup and of the nonlinear part $\mathbf{N}$. For a general discontinuous steady state, the nonlinear operator is not well-defined as an operator from $L^\infty(\Omega)^m \times C(\overline{\Omega})^k$ to itself. 
Hence, admitting general bounded steady states, we consider the nonlinear remainder $\mathbf{N}$ on $L^\infty(\Omega)^{m+k}$. It is well-defined and satisfies a certain estimate for small arguments, see Lemma \ref{nonlinearity}. Accordingly, to deduce nonlinear stability from the linearization, we are interested in growth estimates of the corresponding semigroup restricted to $L^\infty(\Omega)^{m+k}$. Since the semigroup is neither analytic nor strongly continuous on $L^\infty(\Omega)^{m+k}$, validity of the spectral mapping theorem is a priori not guaranteed. Nevertheless, we will show in Proposition \ref{SBeGB} that the growth bound of this semigroup equals the spectral bound of its generator. As the growth of the semigroup determines linear and nonlinear stability, we additionally characterize the spectrum of the linearized operator in Proposition \ref{RDODEspec}.\\

Whereas \cite{CMCKSregular} uses a linearization in the space $L^p(\Omega)^{m} \times L^p(\Omega)$, we depart from this approach and start from a linearization in $L^\infty(\Omega)^m \times L^p(\Omega)^{k}$ in the case of a discontinuous steady state. The generated semigroups are both analytic, however, the latter space allows considering a restriction of the generators and its semigroups to $L^\infty(\Omega)^{m+k}$ and $L^\infty(\Omega)^{m} \times C(\overline{\Omega})^{k}$ on which the spectrum remains the same. While the semigroup restricted to $L^\infty(\Omega)^{m+k}$ is not strongly continuous, a further restriction leads to an analytic semigroup on $L^\infty(\Omega)^{m} \times C(\overline{\Omega})^{k}$. The fact that the semigroup operators restricted to $L^\infty(\Omega)^{m+k}$ are the limit of regularized semigroup operators acting on $L^\infty(\Omega)^m \times C(\overline{\Omega})^k$ allows deriving semigroup estimates on $L^\infty(\Omega)^{m+k}$. These results enable us studying nonlinear stability of bounded, discontinuous steady states in Section \ref{sec:nonlinearstab} from results obtained for its linearization. In the case of a continuous steady state, the spectral analysis on the function space $C(\overline{\Omega})^{m+k}$ cannot easily be deduced from a restriction of the linear operator considered on $L^\infty(\Omega)^m \times L^p(\Omega)^{k}$. The analysis is performed similar to the case $L^p(\Omega)^{m+k}$ studied in \cite[Proposition 5.13]{Kowall}, however, on the function space $C(\overline{\Omega})^{m+k}$. Details are deferred to Proposition \ref{RDODEspecC}.\\

The linearization of system \eqref{fullsys} at a bounded steady state $(\overline{\mathbf{u}}, \overline{\mathbf{v}}) \in L^\infty(\Omega)^{m+k}$ is induced by the operator
\begin{align}
	(\mathbf{L} 
	\xi)(x) & = \mathbf{D} \Delta \xi(x)
	+ \mathbf{J}(x) \xi(x), \quad x \in \Omega.  \label{Linear1}
\end{align}
Here, we use the Jacobian of $(\mathbf{f}, \mathbf{g})$, 
\begin{align}
	\mathbf{J}(x) = \begin{pmatrix}
		\nabla_\mathbf{u} \mathbf{f}(\overline{\mathbf{u}}(x), \overline{\mathbf{v}}(x), x) &  \nabla_\mathbf{v} \mathbf{f}(\overline{\mathbf{u}}(x), \overline{\mathbf{v}}(x), x)\\
		\nabla_\mathbf{u} \mathbf{g}(\overline{\mathbf{u}}(x), \overline{\mathbf{v}}(x), x) &  \nabla_\mathbf{v} \mathbf{g}(\overline{\mathbf{u}}(x), \overline{\mathbf{v}}(x), x)
	\end{pmatrix} =: \begin{pmatrix}
		\mathbf{A}_{\ast}(x)&  \mathbf{B}_{\ast}(x) \\
		\mathbf{C}_{\ast}(x) &\mathbf{D}_{\ast}(x) 
	\end{pmatrix}, \label{Jacobian}
\end{align}
evaluated at the steady state $(\overline{\mathbf{u}}, \overline{\mathbf{v}})$. Moreover, the diagonal diffusion matrix $\mathbf{D}= \mathrm{diag}(\mathbf{0}, \mathbf{D}^v)$ consists of $m$ zeroes and the diagonal matrix $\mathbf{D}^v$. Boundedness of the steady state $(\overline{\mathbf{u}}, \overline{\mathbf{v}})$ results in a multiplication operator $\mathbf{J}$ induced by the Jacobian in formula \eqref{Jacobian}. In the case of $L^p(\Omega)^{m+k}$, the linearized operator $\mathbf{L}$ is a bounded perturbation of the sectorial operator $\mathbf{D} \Delta$ on $L^p(\Omega)^{m+k}$, hence, generates an analytic semigroup. 
However, to solve nonlinear problems, we need boundedness of the ODE component. For this reason, we consider the product space 
\[
Z_p:= L^\infty(\Omega)^m \times L^p(\Omega)^k. 
\]
On the space $Z_p$, unfortunately, the multiplication operator induced by the Jacobian $\mathbf{J}$ is unbounded. Hence, it is not trivial that the linearized operator $\mathbf{L}$ still generates a strongly continuous semigroup on $Z_p$ for a certain domain induced by the degenerated differential operator $\mathbf{D}\Delta$. In order to define the domain of the unbounded operator $\mathbf{D} \Delta$ on the space $L^\infty(\Omega)^m \times L^p(\Omega)^k$, we use the fact that $\mathbf{D} \Delta$ is the generator of a strongly continuous semigroup on $L^\infty(\Omega)^m \times L^p(\Omega)^k$. Indeed, using the heat semigroup $(S_\Delta(t))_{t \in \mathbb{R}_{\ge 0}}$ from Proposition \ref{heathom}, we can define the semigroup $\mathbf{S}(t)=(\mathbf{S}^u(t), \mathbf{S}^v(t))$ by its components
\begin{equation}
S^u_i(t) = I \qquad \text{and} \qquad S_j^v(t) = S_\Delta(D_j^v t)  \label{semigroup}
\end{equation} 
for $ i = 1, \dots, m$ and $j=1, \dots, k$, provided $I$ is the identity operator on $L^\infty(\Omega)$ and 
\[
\mathbf{D}^v:= \mathrm{diag}(D_1^v, \dots, D_k^v) \in \mathbb{R}_{> 0}^{k \times k}.
\] 
The semigroup $(\mathbf{S}^u(t))_{t \in \mathbb{R}_{\ge 0}}$ is time-independent and analytic on $L^p(\Omega)^m$ for $1 \le p \le \infty$. Clearly, its generator $\mathbf{0}$ is bounded. According to Proposition \ref{heathom}, the semigroup $(\mathbf{S}^v(t))_{t \in \mathbb{R}_{\ge 0}}$ is strongly continuous on $L^p(\Omega)^{k}$ for $1 \le p < \infty$ and analytic for each $1 <p <\infty$. It is well-known that its generator 
\[
\mathbf{D}^v \Delta:  \mathcal{D}(A_p)^k \subset L^p(\Omega)^k \to L^p(\Omega)^k
\]
is a densely defined, closed, linear operator \cite[Ch. II, Theorem 1.4]{Engel}. All in all, $(\mathbf{S}(t))_{t \in \mathbb{R}_{\ge 0}}$ is strongly continuous on $L^\infty(\Omega)^m \times L^p(\Omega)^{k}$ for $1 \le p < \infty$ and even analytic for each $1 <p <\infty$ \cite[Theorem 2.16]{Yagi}. Let us define $(\mathbf{D}\Delta, \mathcal{D}(\mathbf{D}\Delta))$ as the generator of the semigroup $(\mathbf{S}(t))_{t \in \mathbb{R}_{\ge 0}}$, i.e.,
\[
\mathbf{D}\Delta :=\mathrm{diag}(\mathbf{0}, \mathbf{D}^v \Delta): L^\infty(\Omega)^m \times \mathcal{D}(A_p)^k \subset L^\infty(\Omega)^m \times L^p(\Omega)^k \to L^\infty(\Omega)^m \times L^p(\Omega)^k.
\]
Properties of the full linearized operator $\mathbf{L}$ defined by \eqref{Linear1} are summarized in the next lemma.

\begin{lemma} \label{Lclosed}
Let the linear operator $\mathbf{L}: \mathcal{D}(\mathbf{L}) \subset Z_p =  L^\infty(\Omega)^m \times L^p(\Omega)^k \to Z_p$ be defined by
\begin{align}
	(\mathbf{L} 
	\xi) (x) & =  \begin{pmatrix}
		\mathbf{0} \\
		\mathbf{D}^v \Delta \xi_2(x)
	\end{pmatrix}
	+  \begin{pmatrix}
		\mathbf{A}_{\ast}(x) \xi_1(x) +  \mathbf{B}_{\ast}(x) \xi_2(x)\\
		\mathbf{C}_{\ast}(x) \xi_1(x) + \mathbf{D}_{\ast}(x) \xi_2(x)
	\end{pmatrix}, \quad x \in \Omega, \label{Linear2}
\end{align}
where $\xi= (\xi_1, \xi_2) \in \mathcal{D}(\mathbf{L}):= L^\infty(\Omega)^{m} \times \mathcal{D}(A_p)^{k}$ (see Proposition \ref{heathom} for definition of the domain $\mathcal{D}(A_p)$). Herein, the diffusion matrix $\mathbf{D}^v \in \mathbb{R}_{> 0}^{k \times k}$ is diagonal and $\mathbf{A}_{\ast}, \mathbf{B}_{\ast}, \mathbf{C}_{\ast},\mathbf{D}_{\ast}$ are matrices with entries in $L^\infty(\Omega)$. Let $n^\ast:= \max\{n/2,2\}$, then for each $n^\ast<p< \infty$ the operator $\mathbf{L}$ is closed and densely defined on $Z_p$ and generates a strongly continuous semigroup $(\mathbf{T}(t))_{t \in \mathbb{R}_{\ge 0}}$ on $Z_p$ which is even analytic.
\end{lemma}

\begin{proof}
In order to show that the operator $\mathbf{L}$ generates an analytic semigroup on $Z_p$, we apply perturbation theory \cite[Ch. III, Theorem 2.10]{Engel}. Although $\mathbf{A}_{\ast}, \mathbf{B}_{\ast}, \mathbf{C}_{\ast}, \mathbf{D}_{\ast}$ are assumed to be bounded matrices which compose the Jacobian $\mathbf{J}$, the operator $\mathbf{D}\Delta$ is perturbed by a an unbounded multiplication operator $\mathbf{J}$ on $Z_p$. Nevertheless, we will show that $\mathbf{J}$ is bounded with respect to the operator $\mathbf{D}\Delta$ in the sense of \cite[Ch. III, Definition 2.1]{Engel}. Finally, this implies that the perturbation $\mathbf{L}= \mathbf{D}\Delta + \mathbf{J}$ generates an analytic semigroup, and is closed and densely defined on $Z_p$ by \cite[Ch. II, Theorem 1.4]{Engel}.

The Jacobian $\mathbf{J}$ consists of the bounded multiplication operator $\mathbf{A}_\ast$ on $L^\infty(\Omega)^m$, while the multiplication operator
\begin{align}
	\mathbf{B}_\ast : \mathcal{D}(\mathbf{B}_\ast) \subset L^p(\Omega)^k \to L^\infty(\Omega)^m \label{Bast}
\end{align}
is unbounded, densely defined and closed for the domain 
\begin{align*}
	\mathcal{D}(\mathbf{B}_\ast) = \{ \xi_2 \in L^p(\Omega)^k \mid \mathbf{B}_\ast \xi_2 \in L^\infty(\Omega)^m\}. 
\end{align*}
Indeed, the subset $C_c^\infty(\Omega)^k \subset \mathcal{D}(\mathbf{B}_\ast)$ is dense in $L^p(\Omega)^k$. 
To show closedness, take a sequence $(\psi_j)_{j \in \mathbb{N}}$ in $\mathcal{D}(\mathbf{B}_\ast)$ which converges in $L^p(\Omega)^k$ to some $\psi \in L^p(\Omega)^k$ such that $\mathbf{B}_\ast \psi_j \to \Psi$ in $L^\infty(\Omega)^m$ for some $\Psi \in L^\infty(\Omega)^m$. Due to boundedness of the multiplication operator in $L^p(\Omega)^k$ \cite[Appendix B.2]{KMCMnonlinear}, 
we have
\[
\|\mathbf{B}_\ast \psi - \Psi\|_p \le \|\mathbf{B}_\ast(\psi-\psi_j)\|_p + \|\mathbf{B}_\ast \psi_j - \Psi \|_p  \le \|\mathbf{B}_\ast\|_\infty \|\psi-\psi_j\|_p + \|\mathbf{B}_\ast \psi_j - \Psi \|_p.
\]
Since the right-hand side converges to $0$ as $j \to \infty$, we obtain $\mathbf{B}_\ast \psi = \Psi$ a.e. in $\Omega$. Thus, $\mathbf{B}_\ast \psi =\Psi \in L^\infty(\Omega)^m$, $\psi \in \mathcal{D}(\mathbf{B}_\ast)$ and $(\mathbf{B}_\ast, \mathcal{D}(\mathbf{B}_\ast))$ is closed. The remaining multiplication operators
\begin{align}
	\mathbf{C}_\ast : L^\infty(\Omega)^m \to L^p(\Omega)^k \qquad \text{and} \qquad \mathbf{D}_\ast : L^p(\Omega)^k \to L^p(\Omega)^k  \label{Cast}
\end{align}
are bounded by the continuous embedding $L^\infty(\Omega) \hookrightarrow L^p(\Omega)$ and
\[
\|\mathbf{C}_\ast \xi_1\|_{L^p(\Omega)^k} \le \|\mathbf{C}_\ast\|_\infty \|\xi_1\|_{L^p(\Omega)^m} \le C \|\xi_1\|_{L^\infty(\Omega)^m}.
\]
To show $\mathbf{D} \Delta$-boundedness of the operator 
$
\mathbf{J}: \mathcal{D}(\mathbf{J}):= L^\infty(\Omega)^m \times \mathcal{D}(\mathbf{B}_{\ast}) \subset Z_p \to Z_p
$
in the sense of \cite[Ch. III, Definition 2.1]{Engel}, we use the continuous embedding $\mathcal{D}(A_p) \hookrightarrow C^{0,\gamma}(\overline{\Omega})$ from Lemma \ref{domchar} and the compact embedding $C^{0,\gamma}(\overline{\Omega}) \hookrightarrow C(\overline{\Omega})$ for $\gamma>0$ from \cite[Theorem 1.34]{Adams}. Ehrling's compactness lemma in \cite[Lemma III.1.1]{Showalter} implies that for each $\varepsilon>0$ there exists a constant $C_\varepsilon>0$ such that
\[
\|\xi_2\|_{C(\overline{\Omega})^k} \le \varepsilon \|\xi_2\|_{C^{0, \gamma}(\overline{\Omega})^k} + C_\varepsilon \|\xi_2\|_{L^p(\Omega)^k} \qquad \forall \; \xi_2 \in C^{0, \gamma}(\overline{\Omega})^k.
\]
Since $\mathbf{J}$ is also a bounded multiplication operator on $Z_\infty :=L^\infty(\Omega)^{m+k}$ by \cite[Appendix B.2]{KMCMnonlinear}
, the embedding $L^\infty(\Omega) \hookrightarrow L^p(\Omega)$ implies
$
\|\mathbf{J} \xi \|_p \le C\|\mathbf{J} \xi \|_\infty \le C \|\mathbf{J}\|_\infty \|\xi \|_\infty. 
$
Herein, we denoted by $\| \cdot \|_p$ and $\| \cdot \|_\infty$ the corresponding norm on $Z_p$ and $Z_\infty$, respectively. By Lemma \ref{domchar} and the above compactness estimate, we obtain for each $\varepsilon>0$ a constant $b_\varepsilon>0$ such that
\[
\|\mathbf{J} \xi \|_p \le \varepsilon \|\mathbf{D}\Delta \xi\|_p + b_\varepsilon \|\xi\|_p \qquad \forall \; \xi \in \mathcal{D}(\mathbf{D}\Delta).
\]
This inequality yields that $\mathbf{J}$ is $\mathbf{D}\Delta$-bounded with bound $0$, since $\varepsilon>0$ was arbitrary. Hence, \cite[Ch. III, Theorem 2.10]{Engel} applies to $\mathcal{D}(\mathbf{D} \Delta) \subset \mathcal{D}(\mathbf{J})$ and the operator $\mathbf{L}$ generates an analytic semigroup on $Z_p$.
\end{proof}

{Before starting with the analysis of the linear system, let us summarize properties of the remaining nonlinear operator $\mathbf{N}$.} 

\begin{lemma} \label{nonlinearity}
Let Assumption \ref{ass:N} hold. For a steady state $(\overline{\mathbf{u}}, \overline{\mathbf{v}}) \in L^\infty(\Omega)^{m+k}$ of system \eqref{fullsys}, define the nonlinear operator 
\begin{align*}
	\mathbf{N}:  L^\infty(\Omega)^{m+k} & \to L^\infty(\Omega)^{m+k}, \quad \mathbf{N}(\xi)(x)  = \mathbf{h} (\mathbf{u},  \mathbf{v}, x)- \mathbf{h} (\overline{\mathbf{u}}, \overline{\mathbf{v}}, x)
	- \nabla \mathbf{h} (\overline{\mathbf{u}}, \overline{\mathbf{v}}, x) \xi  
\end{align*}
for $\xi= (\mathbf{u}, \mathbf{v}) - (\overline{\mathbf{u}}, \overline{\mathbf{v}}), x \in \Omega$. Then, the operator $\mathbf{N}$ is locally Lipschitz continuous and satisfies $\mathbf{N}(\mathbf{0})= \mathbf{0}$. Moreover, there exist constants $\rho, C >0$ such that the superlinear decay estimate 
\begin{align}
	\|\mathbf{N}(\xi)\|_\infty \le C \| \xi\|_\infty^{2} \qquad \forall \; \xi \in L^\infty(\Omega)^{m+k}, \|\xi\|_\infty < \rho \label{estimateN}
\end{align}
holds as $\xi \to 0$.
\end{lemma}

\begin{proof}
Using Taylor's expansion, we write for $\mathbf{h}=(\mathbf{f}, \mathbf{g})$ and $\xi=(\mathbf{u}-\overline{\mathbf{u}}, \mathbf{v}-\overline{\mathbf{v}})$
\begin{align*}
	\mathbf{h} (\mathbf{u},  \mathbf{v}, x)- \mathbf{h} (\overline{\mathbf{u}}, \overline{\mathbf{v}}, x)
	& =  \nabla \mathbf{h} (\overline{\mathbf{u}}, \overline{\mathbf{v}}, x) \xi  + \mathbf{H} (\xi,x),
\end{align*}
where the remainder $\mathbf{H}=(\mathbf{F}, \mathbf{G})$ is given due to the mean value theorem by
\begin{align}
	\mathbf{H}(\xi,x) =  \int_0^1 \nabla \mathbf{h} (\overline{\mathbf{u}} + t \xi_1, \overline{\mathbf{v}} + t \xi_2,x)- \nabla \mathbf{h} (\overline{\mathbf{u}}, \overline{\mathbf{v}},x) \; \mathrm{d}t \cdot \xi. \label{Taylor}
\end{align}
The nonlinear operator $\mathbf{N}$ in equation \eqref{LNequation} then reads
\begin{align*}
	\mathbf{N} \begin{pmatrix} \xi_1\\ \xi_2 \end{pmatrix}  = 
	\begin{pmatrix}
		\mathbf{f}(\overline{u}+\xi_1,\overline{v}+\xi_2,\cdot) - \mathbf{f}(\overline{u},\overline{v}.\cdot)\\
		\mathbf{g}(\overline{u}+\xi_1,\overline{v}+\xi_2,\cdot) - \mathbf{g}(\overline{u},\overline{v},\cdot)
	\end{pmatrix}
	-  \begin{pmatrix}
		\mathbf{A}_{\ast}(\cdot) \xi_1 + \mathbf{B}_{\ast}(\cdot) \xi_2\\
		\mathbf{C}_{\ast}(\cdot) \xi_1 + \mathbf{D}_{\ast}(\cdot) \xi_2
	\end{pmatrix} =\mathbf{H} (\xi, \cdot)
\end{align*}
for $\xi =(\xi_1,\xi_2) \in L^\infty(\Omega)^{m+k}$. To prove estimate \eqref{estimateN} for some $\rho>0$ and all $\xi \in L^\infty(\Omega)^{m+k}$ with $\|\xi\|_\infty < \rho$, we use the local Lipschitz continuity of the gradients of $\mathbf{h}$ from Assumption \ref{ass:Exist}. This assumption allows to estimate the integral in \eqref{Taylor} and we obtain some constant $C=C(\rho)>0$ such that 
\begin{align*}
	\|\mathbf{N}(\xi)\|_\infty \le C \|\xi\|_\infty^2 \qquad \forall \; \xi \in L^\infty(\Omega)^{m+k}, \|\xi\|_\infty < \rho.
\end{align*}
Note that this constant $C$ also depends on the bounded steady state $(\overline{\mathbf{u}}, \overline{\mathbf{v}})$. Additionally, local Lipschitz continuity transfers from $\mathbf{h}$ to $\mathbf{N}$.
\end{proof}

Results on existence of mild solutions to equation \eqref{LNequation} can be obtained in a similar way to \cite[Part II, Theorem 1]{Rothe} by a Picard iteration on $L^\infty(\Omega)^{m+k}$, see Proposition \ref{Rothesol}. However, the notion of a mild solution deserves special attention since solutions are not continuous up to $t=0$ with respect to $L^\infty(\Omega)^{m+k}$, see Definition \ref{mildsol}.


\subsection{Spectral properties} \label{sec:specanalysis} 

In this section, we first characterize the spectrum of the linear operator $\mathbf{L}$ defined in Lemma \ref{Lclosed}. Unfortunately, the proof in \cite{MKS17} seems not adaptable to systems with more than one reaction-diffusion equation. The current proof is adapted from \cite[Proposition B.2]{KMCMnonlinear}, using the close relation of the reaction-diffusion-ODE system \eqref{fullsys} to its shadow limit approximation. Note that the methods used to analyze the spectrum of the linearized operator on $L^\infty(\Omega)^m \times L^p(\Omega)^{k}$ are also applicable to different function spaces such as $L^p(\Omega)^{m+k}$ in \cite{Kowall} or $C(\overline{\Omega})^{m+k}$ in Proposition \ref{RDODEspecC}. In addition to the spectral characterization, we study restrictions of the operator $\mathbf{L}$ to the function spaces $L^\infty(\Omega)^{m+k}$ and $L^\infty(\Omega)^m \times C(\overline{\Omega})^{k}$. Due to elliptic regularity theory and $p$-independence of the spectrum of multiplication operators, the spectrum of the operator $\mathbf{L}$ equals on all these function spaces for all $p \ge 2$ big enough. The basic idea of restricting the operator $\mathbf{L}$ to $L^\infty(\Omega)^{m+k}$ and $L^\infty(\Omega)^m \times C(\overline{\Omega})^{k}$ is the possibility to determine the growth bound of the semigroup $(\mathbf{T}(t))_{t \in \mathbb{R}_{\ge 0}}$ restricted to $L^\infty(\Omega)^{m+k}$ by means of the spectral bound $s(\mathbf{L})$ of the linear operator $\mathbf{L}$. This issue is addressed in Section \ref{sec:semigroupest} which is devoted to semigroup properties.\\

{Recall formula \eqref{Linear2} with bounded coefficients $\mathbf{A}_{\ast}, \mathbf{B}_{\ast}, \mathbf{C}_{\ast},\mathbf{D}_{\ast}$. In order to determine spectral elements $\lambda \in \sigma(\mathbf{L})$, we study the problem
\[
(\lambda I - \mathbf{L}) \xi = \psi  \quad \Leftrightarrow \quad \begin{cases}
	\qquad \qquad \,  (\lambda I -\mathbf{A}_{\ast}) \xi_1 -\mathbf{B}_{\ast} \xi_2 = \psi_1,\\
	- \mathbf{C}_{\ast} \xi_1   + (\lambda I -\mathbf{D}^v \Delta -  \mathbf{D}_{\ast}) \xi_2 = \psi_2
\end{cases}
\]
for $\lambda \in \mathbb{C}$. Let us denote by $\mathbf{A}_{\ast}$ the multiplication operator induced by $\mathbf{A}_{\ast}$ on $L^\infty(\Omega)^m$ \cite[Appendix B.2]{KMCMnonlinear}. 
The following result is inspired by the case considered in \cite[Section 4]{CMCKSregular} concerning a system of $m+1$ equations. Recall the following characterization of $\sigma(\mathbf{A}_\ast)$; there exists a null set $N \subset \Omega$ such that
\begin{align*}
	\sigma(\mathbf{A}_\ast) = \overline{ \bigcup_{x \in \Omega \setminus N} \sigma(\mathbf{A}_\ast(x)) }
\end{align*}
and the whole spectrum is essential in the sense of Schechter, i.e., for every $\lambda \in \sigma(\mathbf{A}_{\ast})$ the operator $\lambda I - \mathbf{A}_{\ast}$ is not a Fredholm operator of index zero \cite[Appendix B.2]{KMCMnonlinear}. For definition and properties of essential spectra and the discrete spectrum $\sigma_d$, we refer to \cite[Section 7.1]{Jeribi}.

\begin{proposition} \label{RDODEspec}
	Let $\mathbf{A}_{\ast}, \mathbf{B}_{\ast}, \mathbf{C}_{\ast}, \mathbf{D}_{\ast}$ be matrix-valued functions which have entries in $L^\infty(\Omega)$ according to the linear operator $\mathbf{L}$ defined by \eqref{Linear2} on $L^\infty(\Omega)^m \times L^p(\Omega)^{k}$ for some $n^\ast <  p < \infty$. Then, using the same notation for the multiplication operator induced by $\mathbf{A}_{\ast}$ on $L^\infty(\Omega)^{m}$, we have $\sigma(\mathbf{A}_\ast) \cap \sigma_d(\mathbf{L}) = \emptyset$ and
	\begin{align*}
		\sigma(\mathbf{L}) = \sigma (\mathbf{A}_{\ast}) \mathop{\dot{\cup}} \Sigma
	\end{align*}
	for some set of eigenvalues $\Sigma \subset \sigma_p(\mathbf{L})$. If $\rho(\mathbf{A}_\ast)$ is a connected set in $\mathbb{C}$, then $\Sigma$ coincides with the discrete spectrum $\sigma_d(\mathbf{L})$ of $\mathbf{L}$.
\end{proposition}

\begin{proof}
	The proof is based on the properties of essential spectra established in the book of Jeribi \cite{Jeribi}. We use the notion of essential spectra from there.
	We work here with 3 concepts of the essential spectrum, i.e. the Wolf spectrum defined by
	\begin{align*}
		\sigma_{e4}(\mathbf{L}) = \{\lambda \in \mathbb{C} \mid \lambda I - \mathbf{L} \text{ is not a Fredholm operator}\},
	\end{align*}
	the Schechter spectrum which can be characterized as follows
	\begin{align*}
		\sigma_{e5}(\mathbf{L}) = \{\lambda \in \mathbb{C} \mid \lambda I - \mathbf{L} \text{ is not a Fredholm operator with index } 0\}
	\end{align*}
	(see \cite[Proposition 7.1.1]{Jeribi}) and lastly the Browder spectrum (see \cite[Section 7.1]{Jeribi} for a definition) that provides the whole set
	\begin{align*}
		\sigma_{e6}(\mathbf{L}) = \sigma(\mathbf{L}) \setminus \sigma_d(\mathbf{L}).
	\end{align*}
	Here, $\sigma_d(\mathbf{L})$ denotes the discrete spectrum
	\begin{align*}
		\sigma_d(\mathbf{L}) = \{\lambda \in \mathbb{C} \mid \lambda \text{ is an isolated eigenvalues of } \mathbf{L} \text{ with a finite algebraic multiplicity}\}.
	\end{align*}
	It holds $\sigma_{e4}(\mathbf{L}) \subset \sigma_{e5}(\mathbf{L}) \subset \sigma_{e6}(\mathbf{L}) \subset \sigma(\mathbf{L})$.\\
	For the Wolf spectrum it was shown in \cite[Appendix B.2]{KMCMnonlinear} that $\sigma(\mathbf{A}_{\ast}) = \sigma_{e4}(\mathbf{A}_{\ast})$.
	Consequently, we have $\sigma(\mathbf{A}_{\ast}) = \sigma_{e5}(\mathbf{A}_{\ast})$, since
	\begin{align*}
		\sigma(\mathbf{A}_{\ast}) = \sigma_{e4}(\mathbf{A}_{\ast}) \subset \sigma_{e5}(\mathbf{A}_{\ast}) \subset \sigma(\mathbf{A}_{\ast}).
	\end{align*}
	In the next step, we want to show that $\sigma_{e5}(\mathbf{A}_{\ast}) = \sigma_{e5}(\mathbf{L})$.
	To this end, we apply results established in \cite[Theorem 10.1.3 (i)]{Jeribi} for an operator in a block form
	\begin{align*}
		L = \begin{pmatrix}A&B\\C&D\end{pmatrix} \quad \text{on } X_1\times X_2,
	\end{align*}
	where we need some good properties of the resolvent of $A$.\\ 
	Using the notation of the book \cite{Jeribi} for the operator $\mathbf{L}$ given by \eqref{Linear2}, we take $A:= \mathbf{D}^v \Delta + \mathbf{D}_\ast$ on $X_1:=L^p(\Omega)^{k}$.
	Since $A$ is a bounded perturbation of the generator of the heat semigroup from Proposition \ref{heathom}, $A$ is a densely defined, closed operator \cite[Ch. II, Theorem 1.4]{Engel}. Furthermore, $\rho(A) \not= \emptyset$ by \cite[Ch. II, Theorem 1.10]{Engel}. From \cite[Theorem 1.6.3]{Davies}, \cite[Ch. II, Theorem 4.29]{Engel} and perturbation theory from \cite[Ch. III, Theorem 1.12]{Engel}, we obtain compactness of the resolvent $R(\lambda, A)$.\\
	Moreover, we have the bounded multiplication operator $D := \mathbf{A}_{\ast}$ on $X_2:= L^\infty(\Omega)^{m}$.
	The remaining operators $B:= \mathbf{C}_\ast: X_2 \to X_1$ and $C:= \mathbf{B}_\ast: X_1 \to X_2$ consist of a bounded and an unbounded, closed and densely defined multiplication operator, respectively; see operators in \eqref{Bast} and \eqref{Cast}.\\
	This means the above block setting is obtained by a permutation of the operator matrix $\mathbf{L}$. Such a permutation does not change the essential spectra which can be seen as follows.
	Let us consider a permutation matrix $\mathbf{P} \in \mathbb{R}^{(m+k) \times (m+k)}$ with $\mathbf{P}^2 =I$ that interchanges the diffusive subsystem with the ODE subsystem. As an operator, $\mathbf{P}$ is an isomorphism from $L^p(\Omega)^k \times L^\infty(\Omega)^{m}$ to $L^\infty(\Omega)^{m} \times L^p(\Omega)^k$. Then $\lambda I - \mathbf{L}$ is a Fredholm operator with index zero if and only if $\lambda I - \tilde{\mathbf{L}}$ is Fredholm with index zero where $\tilde{\mathbf{L}} = \mathbf{P}^{-1}\mathbf{L} \mathbf{P}$, hence $\sigma_{e4}(\mathbf{L}) = \sigma_{e4}(\tilde{\mathbf{L}})$ and $\sigma_{e5}(\mathbf{L}) = \sigma_{e5}(\tilde{\mathbf{L}})$. The former is a consequence of the fact that the invertible operators $\mathbf{P}, \mathbf{P}^{-1}$ are (bounded) Fredholm operators and $\lambda I - \tilde{\mathbf{L}} = \mathbf{P}^{-1} (\lambda I - \mathbf{L}) \mathbf{P}$ is a reversible composition with Fredholm operators \cite[Theorem 2.2.40]{Jeribi}.\\
	These properties allow applying \cite[Theorem 10.1.3 (i)]{Jeribi}. The theorem provides the equivalence $\sigma_{ei}(\mathbf{L}) = \sigma_{ei}(S_0)$ for $i = 4,5$ for an operator $S_0$ obtained by a decomposition of the bounded operator
	\begin{align*}
		S(\mu) :=  D - C(A-\mu I)^{-1}B : X_2 \to X_2
	\end{align*}
	for $\mu\in\rho(A)$.
	Moreover, the theorem \cite[Theorem 10.1.3 (i)]{Jeribi} yields $\sigma_{e6}(\mathbf{L}) = \sigma_{e6}(S_0)$ in case of $\mathbb{C}\setminus\sigma_{e5}(S_0)$ is a connected set and
	neither $\rho(S_0)$ nor $\rho(S(\mu))$ is empty.
	In our case, $S(\mu)$ can be decomposed into the bounded operator $S_0 = \mathbf{A}_{\ast}$ and a compact operator $M(\mu) = S(\mu) - S_0$.
	Hence, $\rho(S_0)$ and $\rho(S(\mu))$ are both not empty in our case. Thus, connectedness of $\mathbb{C}\setminus\sigma_{e5}(S_0) = \rho(\mathbf{A}_{\ast})$ is sufficient for $\sigma_{e6}(\mathbf{L}) = \sigma_{e6}(\mathbf{A}_{\ast})$.\\
	Since $\mathbf{A}_{\ast}$ is bounded, it is left to show compactness of
	\begin{align*}
		M(\mu) = -C(A- \mu I)^{-1} B : X_2 \to X_2.
	\end{align*}
	By \cite[Theorem 1.34]{Adams} and Lemma \ref{domchar}, the embedding $\iota: (\mathcal{D}(A_p)^k, \| \cdot\|_{\mathbf{D}^v \Delta}) \to L^\infty(\Omega)^k$ is compact for $p>n^\ast$, using the graph norm of the operator $\mathbf{D}^v\Delta$. Recall that the graph norm of $\mu I - A$ and $\| \cdot\|_{\mathbf{D}^v \Delta}$ are equivalent for $\mu \in \rho(A)$. Since $\mu \in \rho(A)$, the operator $R_\mu \psi=-(A-\mu I)^{-1} \psi$ yields a bounded isomorphism $R_\mu: L^p(\Omega)^k \to (\mathcal{D}(A_p)^k, \| \cdot\|_{\mathbf{D}^v \Delta})$ with bounded inverse $\mu I - A$. Lastly, consider the restriction $C_|: L^\infty(\Omega)^k \to L^\infty(\Omega)^m$ of the operator $C$ to $L^\infty(\Omega)^k$. Then, since $C_|, B$ and $R_\mu$ are bounded, the operator $M(\mu) = C_| \iota R_\mu B$ is compact \cite[Appendix B]{Arendt}.
	Thus, $\sigma_{e5}(\mathbf{L}) = \sigma_{e5}(S_0) = \sigma_{e5}(\mathbf{A}_{\ast})$ is a consequence of \cite[Theorem 10.1.3 (i)]{Jeribi}.\\
	Using the relation $\sigma_{e5}(\mathbf{L}) \subset \sigma_{e6}(\mathbf{L})$ for the Browder essential spectrum $\sigma_{e6}(\mathbf{L}) = \sigma(\mathbf{L}) \setminus \sigma_d(\mathbf{L})$, we obtain
	\[
	\sigma (\mathbf{A}_{\ast}) =  \sigma_{e5} (\mathbf{A}_{\ast}) =  \sigma_{e5}(\mathbf{L}) \subset \sigma(\mathbf{L}) \setminus \sigma_d(\mathbf{L}). 
	\]
	Since $\sigma (\mathbf{A}_{\ast})= \sigma_{e5}(\mathbf{L})$, the remainder of the spectrum $\Sigma := \sigma(\mathbf{L}) \setminus \sigma (\mathbf{A}_{\ast})$ consists of those $\lambda \in \rho(\mathbf{A}_{\ast}) \cap \sigma(\mathbf{L})$ for which $\lambda I - \mathbf{L}$ is a Fredholm operator of index zero. Hence, the operator $\lambda I - \mathbf{L}$ cannot be injective and $\lambda \in \sigma_p(\mathbf{L})$, thus $\Sigma\subset\sigma_p(\mathbf{L})$.\\
	Finally, if $\rho(\mathbf{A}_\ast)$ is a connected set, \cite[Theorem 10.1.3 (i)]{Jeribi} shows $\sigma_{e6}(\mathbf{L}) = \sigma_{e6}(\mathbf{A}_\ast)$. Due to the relation $\sigma(\mathbf{A}_\ast) = \sigma_{e5}(\mathbf{A}_{\ast})  \subset \sigma_{e6}(\mathbf{A}_\ast) \subset \sigma (\mathbf{A}_\ast)$, we infer 
	\begin{align*}
		\Sigma = \sigma(\mathbf{L}) \setminus \sigma(\mathbf{A}_\ast) = \sigma(\mathbf{L}) \setminus \sigma_{e6}(\mathbf{L}) = \sigma_d(\mathbf{L}).
	\end{align*}
\end{proof}
}


This characterization of the spectrum of $\mathbf{L}$ coincides with the spectrum of its extension to $L^p(\Omega)^{m+k}$ by \cite[Proposition 5.13]{Kowall}, using elliptic regularity for $p>n^\ast$, compare also \cite{CMCKSregular} for $k=1$. It should be stressed here that the set $\Sigma$ does not have to be discrete if the resolvent set $\rho(\mathbf{A}_\ast)$ is not connected. We refer to Appendix \ref{sec:specchar} for further discussions. Similar to the $p$-independence of the spectrum of the Laplacian on $L^p(\Omega)$ from \cite[Example 1.2]{Arendtspec}, we can verify such a property for the spectrum $\sigma(\mathbf{L})$. 

\begin{corollary} \label{independence}
The spectrum $\sigma(\mathbf{L})$ characterized in Proposition \ref{RDODEspec} for the linear operator $\mathbf{L}$ defined by \eqref{Linear2} on $L^\infty(\Omega)^m \times L^p(\Omega)^{k}$ is independent of $n^\ast < p< \infty$.  
\end{corollary}

\begin{proof}
In view of Proposition \ref{RDODEspec} and the spectral characterization of $\sigma(\mathbf{A}_\ast)$ in \cite[Appendix B.2]{KMCMnonlinear}, it remains to consider the set $\Sigma$. Each $\lambda \in \Sigma$ is an eigenvalue  of $\mathbf{L}$, i.e., there exists an eigenfunction $\xi = (\xi_1, \xi_2) \in L^\infty(\Omega)^{m} \times \mathcal{D}(A_p)^{k}$ such that
$
(\lambda I- \mathbf{L}) \xi = \mathbf{0}
$
holds in the weak sense. Since $\lambda \in \Sigma \subset \rho(\mathbf{A}_\ast)$, the second equation of this eigenvalue problem leads to $\xi_1 = (\lambda I- \mathbf{A}_{\ast})^{-1} \mathbf{B}_{\ast} \xi_2$ and
\begin{align*}
	(\lambda I - \mathbf{D}^v \Delta -  \mathbf{D}_{\ast} - \mathbf{C}_{\ast}(\lambda I - \mathbf{A}_{\ast})^{-1} \mathbf{B}_{\ast}) \xi_2 =  \mathbf{0} 
\end{align*}
for $\xi_2 \in \mathcal{D}(A_p)^{k} \setminus \{\mathbf{0}\}$. This elliptic problem can be reformulated as
$
- \mathbf{D}^v \Delta \xi_2 = \mathbf{F}(x) \xi_2 
$
where $\mathbf{F} \in L^\infty(\Omega)^{k \times k}$. Using elliptic regularity from Lemma \ref{domchar}, we infer that $\mathcal{D}(A_p) \subset L^q(\Omega)$ for all $1 \le q \le \infty$. Hence, $\xi_2 \in \mathcal{D}(A_q)^{k}$ and
$\xi \in L^\infty(\Omega)^m \times \mathcal{D}(A_q)^k$ for all $2 \le q \le \infty$ by Lemma \ref{domchar}, and $(\lambda, \xi)$ is an eigenpair independent of the exponent $p$.
\end{proof}

The $p$-independence of the spectrum of the linear operator allows us to restrict the operator $\mathbf{L}$, which is defined on $L^\infty(\Omega)^m \times L^p(\Omega)^{k}$, to the subspaces $L^\infty(\Omega)^{m+k}$, $L^\infty(\Omega)^m \times C(\overline{\Omega})^{k}$ and $C(\overline{\Omega})^{m+k}$ on which the spectrum remains the same, see Proposition \ref{specrestricted} and \ref{RDODEspecC}, respectively. The advantage of these restrictions is that we are able to relate the spectral bound of the linear operator to the growth bound of the semigroup $(\mathbf{T}(t))_{t \in \mathbb{R}_{\ge 0}}$ restricted to $L^\infty(\Omega)^{m+k}$.
The main properties of the restricted operators and semigroups is depicted in Figure \ref{fig:Restrictions}. Lemmata \ref{restrclosed}--\ref{resolventestrestr} are concerned with the restriction $\mathbf{L}_\infty$, the part of $\mathbf{L}$ in $L^\infty(\Omega)^{m+k}$. Since this operator is not densely defined, it is not the generator of an analytic semigroup. Nevertheless, {the semigroup has analytic properties for $t>0$, as shown in Lemma \ref{restrsemigroup}.} A further restriction to the closed subspace $Z_c= \overline{\mathcal{D}(\mathbf{L}_\infty)}=L^\infty(\Omega)^m \times C(\overline{\Omega})^{k}$ yields again a generator, denoted by $\mathbf{L}_c$, of an analytic semigroup. 
This regularized semigroup is used to determine the growth bound of the non-continuous semigroup generated by $\mathbf{L}_\infty$, see Proposition \ref{SBeGB}.

\begin{figure}[h]
\begin{center}
	\begin{tikzcd}
		(\mathbf{L}, \mathcal{D}(\mathbf{L})) \enspace \text{on} \enspace Z_p = L^\infty(\Omega)^m \times L^p(\Omega)^k \arrow[leftrightarrow]{r}[swap]{\text{Lemma} \hspace{.2777em} \ref{Lclosed}}  
		&[4em] (\mathbf{T}(t))_{t \in \mathbb{R}_{\ge 0}} \subset \mathcal{L}(Z_p) \\[2em]
		(\mathbf{L}_{\infty}, \mathcal{D}(\mathbf{L}_{\infty}))  \enspace \text{on} \enspace Z_\infty	= L^\infty(\Omega)^m \times L^\infty(\Omega)^k \arrow[leftrightarrow]{r}[swap]{\text{Lemma} \hspace{.2777em} \ref{restrsemigroup}}  \arrow[leftarrow]{u}{\text{Lemma} \hspace{.2777em}  \ref{restrclosed}}[swap]{\sigma(\mathbf{L}_{\infty})=\sigma(\mathbf{L})}
		& (\mathbf{T}_\infty(t))_{t \in \mathbb{R}_{\ge 0}} \subset \mathcal{L}(Z_\infty)\arrow[leftarrow]{u}{w_0^{\infty}=s(\mathbf{L})}[swap]{\text{Lemma} \hspace{.2777em} \ref{restrsemigroup}} 
		\\[2em]
		(\mathbf{L}_c, \mathcal{D}(\mathbf{L}_c))  \enspace \text{on} \enspace Z_c = L^\infty(\Omega)^m \times C(\overline{\Omega})^k \arrow[leftrightarrow]{r}[swap]{\text{Lemma} \hspace{.2777em} \ref{restrsemigroup2}}  \arrow[leftarrow]{u}{\text{Lemma} \hspace{.2777em} \ref{restrclosed2}}[swap]{\sigma(\mathbf{L}_c)=\sigma(\mathbf{L}_{\infty})}
		& (\mathbf{T}_c(t))_{t \in \mathbb{R}_{\ge 0}} \subset \mathcal{L}(Z_c)\arrow[leftrightarrow]{u}{w_0^c= s(\mathbf{L})}[swap]{\text{Lemma} \hspace{.2777em} \ref{limitrep}}
	\end{tikzcd}
\end{center} 
\caption{Relation between restricted semigroups and its generators.} \label{fig:Restrictions}
\end{figure}
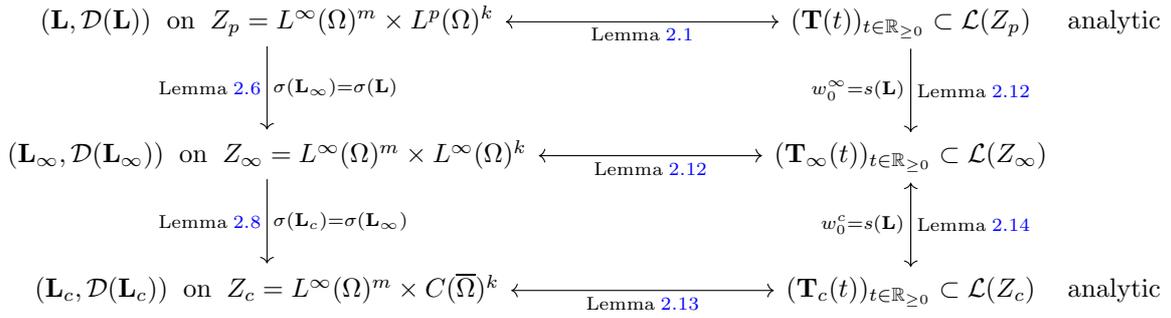

First, we look at the part of the linear operator $\mathbf{L}$ defined by \eqref{Linear2} in $Z_\infty= L^\infty(\Omega)^{m+k}$ given by the operator
\begin{align}
\mathbf{L}_{\infty} : \mathcal{D}(\mathbf{L}_{\infty}) := \{ \xi \in\mathcal{D}(\mathbf{L})\cap Z_\infty \mid \mathbf{L}\xi \in Z_\infty\}\subset Z_\infty \to Z_\infty,\ \xi \mapsto \mathbf{L}\xi. \label{definitionLinfty}
\end{align}
In view of Corollary \ref{independence}, the part of $\mathbf{L}$ in $Z_\infty$ is independent of $p>n^\ast$. 

\begin{lemma} \label{restrclosed}
The restricted operator $\mathbf{L}_{\infty}$ defined in \eqref{definitionLinfty} is closed with domain 
\[
\mathcal{D}(\mathbf{L}_{\infty})  = L^\infty(\Omega)^m \times \mathcal{D}(A_\infty)^k,
\]
where $A_\infty$ is defined in Lemma \ref{resolvestLaplace}. However, $\mathbf{L}_\infty$ is not densely defined on $L^\infty(\Omega)^{m+k}$.
\end{lemma}

\begin{proof}
Closedness of the part $\mathbf{L}_{\infty}$ follows from \cite[Section 3.10, p. 184]{Arendt} since $\mathbf{L}$ is closed and we have the continuous embedding $L^\infty(\Omega)^{m+k} \hookrightarrow L^\infty(\Omega)^m \times L^p(\Omega)^k$. {The identity $\mathcal{D}(\mathbf{L}_\infty) = L^\infty(\Omega)^m \times \mathcal{D}(A_\infty)^k$ can be deduced from the relations of $A_p$ and $A_\infty$ considered in Lemma \ref{domchar} and \ref{resolvestLaplace}}. Finally, by Lemma \ref{domchar}, we obtain $\mathcal{D}(\mathbf{L}_{\infty}) \subset L^\infty(\Omega)^m \times C(\overline{\Omega})^k$. Hence, the inclusion $\overline{\mathcal{D}(\mathbf{L}_{\infty})} \subset L^\infty(\Omega)^m \times C(\overline{\Omega})^k \subsetneq Z_\infty$ holds and the operator $\mathbf{L}_{\infty}$ cannot be densely defined on $Z_\infty$.
\end{proof}

Since the operator $\mathbf{L}_{\infty}$ is not densely defined on $L^\infty(\Omega)^{m +k}$, $\mathbf{L}_{\infty}$ cannot generate a strongly continuous semigroup on $L^\infty(\Omega)^{m +k}$ \cite[Ch. II, Theorem 1.4]{Engel}. However, it satisfies a similar resolvent estimate as in the case of a generator of an analytic semigroup. {The following result can be compared to \cite[Ch. III, Theorem 2.10]{Engel}, and to \cite[Corollary 3.7.17]{Arendt} for non-continuous holomorphic semigroups.} Such an estimate will enable us to identify the growth bound of the semigroup $(\mathbf{T}(t))_{t \in \mathbb{R}_{\ge 0}}$ restricted to $L^\infty(\Omega)^{m +k}$ with the spectral bound $s(\mathbf{L}_{\infty})$ of the operator $\mathbf{L}_{\infty}$.

\begin{lemma} \label{resolventestrestr}
There exist a closed sector $\Sigma_\omega :=  \{ \lambda \in \mathbb{C} \mid |\mathrm{arg} \, \lambda| \le \omega\}$ for some angle $\frac{\pi}{2} < \omega < \pi$ and constants $M, r >0$ such that $
\Sigma_\omega \cap \{ \lambda \in \mathbb{C} \mid |\lambda|>r\} \subset \rho(\mathbf{L}_{\infty})
$
with the resolvent estimate
\begin{align}
	\|R(\lambda, \mathbf{L}_{\infty})\| \le \frac{M}{|\lambda|} \qquad \forall \; \lambda \in \Sigma_\omega \cap \{ \lambda \in \mathbb{C} \mid |\lambda|>r\}. \label{resolventest}
\end{align} 
\end{lemma}

\begin{proof} Let us start with the operator $\mathbf{D}^v \Delta$ on $L^\infty(\Omega)^k$. Applying the estimate from Lemma \ref{resolvestLaplace} to each component $\mathbf{D}^v_j \Delta$ on $L^\infty(\Omega)$ for $j=1, \dots, k$, we obtain a resolvent estimate of the form
\begin{align}
	\|(\lambda I- \mathbf{D}^v \Delta)^{-1}\|_{L^\infty(\Omega)^k} \le \frac{\tilde{M}}{|\lambda|} \qquad \forall \; \lambda \in \Sigma_\omega \setminus \{0\}. \label{sectorial}
\end{align}
This is similar to sectorial operators on some sector $\Sigma_\omega = \{ \lambda \in \mathbb{C} \mid |\mathrm{arg} \, \lambda| \le \omega\}$ with angle $\omega > \frac{\pi}{2}$. Apart from a dense domain, the operator $\mathbf{D}^v \Delta$ restricted to $L^\infty(\Omega)^k$ satisfies all other conditions of sectoriality in \cite[Ch. II, Definition 4.1]{Engel}. Since the zero operator on $L^\infty(\Omega)^m$ is also sectorial, the operator $\mathbf{D}\Delta$ restricted to $L^\infty(\Omega)^{m +k}$ satisfies a resolvent estimate similar to \eqref{sectorial}. By \cite[Ch. III, Lemma 2.6]{Engel}, the bounded perturbation $\mathbf{J}$ in $L^\infty(\Omega)^{m +k}$ implies a similar resolvent estimate for the operator $\mathbf{L}_{\infty} = \mathbf{D} \Delta + \mathbf{J}$. More precisely, as in \cite[Ch. 3, Theorem 2.1]{Pazy}, we obtain
\begin{align*}
	\|R(\lambda, \mathbf{L}_{\infty})\| \le \frac{2\tilde{M}}{|\lambda|} \qquad \forall \; \lambda \in \Sigma_\omega \cap \{ \lambda \in \mathbb{C} \mid |\lambda|>r\}
\end{align*}  
for $r= 2\|\mathbf{J}\|_\infty \tilde{M}>0$ and $\tilde{M}$ from estimate \eqref{sectorial}.
\end{proof}

Let us now consider a further restriction of the operator $\mathbf{L}_{\infty}$ to its subspace of strong continuity. The above resolvent estimate for the operator $\mathbf{L}_{\infty}$ shows existence of a generator of an analytic semigroup on the closed subspace $Z_c:= \overline{\mathcal{D}(\mathbf{L}_{\infty})}$ of $Z_\infty=L^\infty(\Omega)^{m +k}$, see Lemma \ref{restrsemigroup2} and compare \cite[Corollary 3.1.24]{Lunardi}. In fact, the space $(Z_c, \| \cdot \|_\infty)$ is again a Banach space with continuous embedding $Z_c \hookrightarrow Z_\infty$. Let us define the part of $\mathbf{L}_{\infty}$ in $Z_c$ by
\begin{align}
\mathbf{L}_c \xi & := \mathbf{L}_{\infty}\xi \qquad \forall \; \xi \in \mathcal{D}(\mathbf{L}_c) := \{ \xi \in \mathcal{D}(\mathbf{L}_{\infty}) \cap Z_c \mid \mathbf{L}_{\infty} \xi \in Z_c \}. \label{definitionL0}
\end{align}
We obtain the following result for $\mathbf{L}_c$.

\begin{lemma} \label{restrclosed2}
The operator $\mathbf{L}_c$ defined in \eqref{definitionL0} is closed and densely defined on the space $Z_c  = L^\infty(\Omega)^m \times C(\overline{\Omega})^k$. The operator $\mathbf{L}_c$ is the part of $\mathbf{L}_{\infty}$ resp. $\mathbf{L}$ in $Z_c$.
\end{lemma}

\begin{proof}
Since $\mathbf{L}_{\infty}$ is a closed operator on $Z_\infty$ by Lemma \ref{restrclosed}, the continuous embedding $Z_c \hookrightarrow Z_\infty$ yields closedness of the part $\mathbf{L}_c$ \cite[Section 3.10, p. 184]{Arendt}. By the inclusion $\mathcal{D}(\mathbf{L}_{\infty}) \subset Z_c \subset Z_\infty$, we obtain that $\mathbf{L}_c$ is also the part of $\mathbf{L}$ in $Z_c$, i.e.,
\[
\mathcal{D}(\mathbf{L}_c) = \{ \xi \in \mathcal{D}(\mathbf{L}) \cap Z_c \mid \mathbf{L} \xi \in Z_c \}.
\]
In order to show that $\mathbf{L}_c$ is densely defined, it remains to verify $\overline{\mathcal{D}(\mathbf{L}_c)} = Z_c$. In view of estimate \eqref{resolventest} for $R(\lambda, \mathbf{L}_{\infty})$ on $Z_\infty$, this follows from \cite[Lemma 3.3.12]{Arendt} applied to the operator $\mathbf{L}_{\infty}$ and $Z_c = \overline{\mathcal{D}(\mathbf{L}_{\infty})}$. 
Finally, we show the equality $Z_c  = L^\infty(\Omega)^m \times C(\overline{\Omega})^k$. 
Using the restriction of the Laplacian $A_p$ to $C(\overline{\Omega})$ from Lemma \ref{LaplacianC}, one can verify the inclusions $L^\infty(\Omega)^m \times \mathcal{D}(A_c)^k \subset \mathcal{D}(\mathbf{L}_\infty) \subset L^\infty(\Omega)^m \times C(\overline{\Omega})^k$. Taking the closure with respect to the norm $\| \cdot \|_\infty$ yields $L^\infty(\Omega)^m \times C(\overline{\Omega})^k = \overline{\mathcal{D}(\mathbf{L}_{\infty})} = Z_c$.
\end{proof}

The next result makes use of the fact that we consider the operator $\mathbf{L}$ on the space $L^\infty(\Omega)^m \times L^p(\Omega)^k$ and not on $L^p(\Omega)^{m+k}$. Equality of the ODE domain {and elliptic regularity} allows showing equality of the spectrum of $\mathbf{L}$ and its parts $\mathbf{L}_{\infty}$ and $\mathbf{L}_c$ in $L^\infty(\Omega)^{m+k}$ and $L^\infty(\Omega)^m \times C(\overline{\Omega})^k$, respectively.

\begin{proposition} \label{specrestricted}
Let $\mathbf{L}$ be the linear operator defined in Lemma \ref{Lclosed}. Let the restrictions $\mathbf{L}_\infty$ and $\mathbf{L}_c$ be defined as in \eqref{definitionLinfty} and \eqref{definitionL0}, respectively. Then, the spectra $\sigma(\mathbf{L}),  \sigma(\mathbf{L}_{\infty})$, and $\sigma(\mathbf{L}_c)$ coincide. Furthermore, the resolvents can be determined by the restrictions
\begin{align*}
	R(\lambda, \mathbf{L}_{\infty}) = R(\lambda, \mathbf{L})_{|L^{\infty}(\Omega)^{m+k}} \quad \text{and} \quad R(\lambda, \mathbf{L}_{c}) = R(\lambda, \mathbf{L})_{|L^\infty(\Omega)^m \times C(\overline{\Omega})^k} \quad \forall \; \lambda \in \rho(\mathbf{L}). 
\end{align*}
\end{proposition}

\begin{proof}
By Corollary \ref{independence}, we know that the spectrum of $\mathbf{L}$ is independent of $p$. Hence, we fix $n^\ast < p <\infty$ and prove the equality of the spectra for this specific $p$ by applying \cite[Proposition 3.10.3]{Arendt}. 
By H\"older's inequality, we have $\|z\|_p \le \max\{|\Omega|, 1\} \|z\|_\infty$ for all $z \in L^\infty(\Omega)^{m+k}$ with a constant only depending on the Lebesgue measure of $\Omega$ but not on $p\ge 1$. This shows the continuous embeddings 
\[
L^\infty(\Omega)^m \times C(\overline{\Omega})^k \hookrightarrow  L^\infty(\Omega)^{m+k} \hookrightarrow  L^\infty(\Omega)^m \times L^p(\Omega)^k.
\]
Moreover, Lemma \ref{domchar} yields $\mathcal{D}(\mathbf{L}) \subset L^\infty(\Omega)^m \times C(\overline{\Omega})^k$ for $p>n^\ast$. This shows invariance of the sets $L^\infty(\Omega)^{m+k}$ and $L^\infty(\Omega)^m \times C(\overline{\Omega})^k$ for the resolvent $R(\lambda, \mathbf{L})$ for each $\lambda \in \rho(\mathbf{L})$. Note that the resolvent set $\rho(\mathbf{L})$ is not empty by Proposition \ref{RDODEspec}. Finally, we apply \cite[Proposition 3.10.3]{Arendt} to obtain the above equalities.
\end{proof}

In the case of a continuous stationary solution of system \eqref{fullsys}, the operator $\mathbf{L}$ can be restricted to the space $C(\overline{\Omega})^{m+k}$, too. If Assumption \ref{ass:Nc} holds, we obtain the restriction of the operator $\mathbf{L}$ by a bounded perturbation of the degenerated differential operator $\mathbf{D}\Delta$ on $C(\overline{\Omega})^{m+k}$.

\begin{lemma} \label{Lcontinuous} 
Let $(\overline{\mathbf{u}}, \overline{\mathbf{v}}) \in C(\overline{\Omega})^{m+k}$ be a stationary solution of system \eqref{fullsys} and assume that the coefficient matrices $\mathbf{A}_\ast, \mathbf{B}_\ast, \mathbf{C}_\ast, \mathbf{D}_\ast$ in formula \eqref{Jacobian} have entries in $C(\overline{\Omega})$. The part of the operator $\mathbf{L}$, defined in Lemma \ref{Lclosed}, in $C(\overline{\Omega})^{m+k}$ is given by
\begin{align*}
	\mathbf{L}^{c} : \mathcal{D}(\mathbf{L}^{c}) = C(\overline{\Omega})^{m} \times \mathcal{D}(A_c)^k \subset C(\overline{\Omega})^{m+k} \to C(\overline{\Omega})^{m+k}, \  \xi \mapsto \mathbf{L}\xi, 
\end{align*}
see Lemma \ref{LaplacianC} for definition of $\mathcal{D}(A_c)$. The operator $\mathbf{L}^c$ is closed and densely defined on $C(\overline{\Omega})^{m+k}$ and generates a strongly continuous semigroup 
on $C(\overline{\Omega})^{m+k}$ which is even analytic. This semigroup coincides with the restriction of the semigroup $(\mathbf{T}(t))_{t \in \mathbb{R}_{\ge 0}}$ defined in Lemma \ref{Lclosed} onto $C(\overline{\Omega})^{m+k}$.
\end{lemma}

\begin{proof}
Since $C(\overline{\Omega})^{m+k} \hookrightarrow L^\infty(\Omega)^{m+k}$ and $\mathbf{L}_\infty$ is closed, we obtain closedness of the part of $\mathbf{L}$ in $C(\overline{\Omega})^{m+k}$. Note that the Jacobian $\mathbf{J}$ consists of continuous entries for a continuous steady state. In this way, one verifies the identity 
\[
\mathcal{D}(\mathbf{L}^c)= \{ \xi \in\mathcal{D}(\mathbf{L})\cap C(\overline{\Omega})^{m+k} \mid \mathbf{L}\xi \in C(\overline{\Omega})^{m+k} \} =  C(\overline{\Omega})^{m} \times \mathcal{D}(A_c)^k.
\] 
By Lemma \ref{LaplacianC}, the operator $\mathbf{L}^c$ is densely defined. Recall for $\lambda \in \rho(\mathbf{L}) \subset \rho(\mathbf{A}_\ast)$ that
\begin{align} \label{resolvexplicit}
	(\lambda I - \mathbf{L}) \xi = \psi  \; \; \Leftrightarrow \; \;  \begin{cases}
		\qquad \qquad \quad \; \xi_1 = (\lambda I - \mathbf{A}_{\ast})^{-1} (\psi_1 + \mathbf{B}_{\ast} \xi_2),\\
		(\lambda I -\mathbf{D}^v \Delta) \xi_2 = \psi_2 + \mathbf{C}_{\ast} \xi_1   + \mathbf{D}_{\ast} \xi_2.
	\end{cases}
\end{align}
From this explicit form of the resolvent $R(\lambda, \mathbf{L})$, we infer that $C(\overline{\Omega})^{m+k}$ is an invariant subspace for the resolvent $R(\lambda, \mathbf{L})$. Moreover, the part of $\mathbf{L}$ on the space $C(\overline{\Omega})^{m+k}$ satisfies the resolvent estimate \eqref{resolventest}. By \cite[Theorem 1.6]{Tanabe}, this part is the generator of an analytic semigroup on $C(\overline{\Omega})^{m+k}$. Finally, we apply \cite[Ch. 4, Theorem 5.5]{Pazy} to see that the analytic semigroup on $C(\overline{\Omega})^{m+k}$ is just a restriction of the semigroup $(\mathbf{T}(t))_{t\in\mathbb{R}_{\ge0}}$.
\end{proof}

Unfortunately, methods used in the proof of Proposition \ref{specrestricted} only imply $\sigma(\mathbf{L}^c) \subset \sigma(\mathbf{L})$ due to minor regularity of the ODE component in definition of $\mathbf{L}$.
Analog to Proposition \ref{RDODEspec}, where a similar spectral decomposition is verified for the operator $\mathbf{L}$ on $L^\infty(\Omega)^m \times L^p(\Omega)^k$, we obtain the following spectral characterization. 
We are now able to conclude that the spectrum of the operator $\mathbf{L}$ and $\mathbf{L}^c$ coincide, too.

\begin{proposition} \label{RDODEspecC}
Let $m, k \in \mathbb{N}$ and $\mathbf{D}^v \in \mathbb{R}_{>0}^{k \times k}$. Let $\mathbf{A}_{\ast}, \mathbf{B}_{\ast}, \mathbf{C}_{\ast}, \mathbf{D}_{\ast}$ be matrix-valued functions with entries in $C(\overline{\Omega})$ according to the linear operator $\mathbf{L}^c$ defined in Lemma \ref{Lcontinuous} on $C(\overline{\Omega})^{m+k}$. Then, the spectrum of the operator $\mathbf{L}^c$ is given by
\[
\sigma(\mathbf{L}^c) = \sigma (\mathbf{A}_{\ast}) \mathop{\dot{\cup}} \Sigma = \sigma(\mathbf{L}),
\]
where $\sigma(\mathbf{L})$ is characterized in Proposition \ref{RDODEspec}.
\end{proposition}

\begin{proof}
Since $A_c$ has compact resolvent on $C(\overline{\Omega})$ by Lemma \ref{LaplacianC}, the proof of Proposition \ref{RDODEspec} literally applies to characterize the spectrum of $\mathbf{L}^c$. It is again sufficient to show $\sigma_{e5}(\mathbf{A}_{\ast}) = \sigma_{e5}(\mathbf{L}^c)$ for the Schechter essential spectrum, since we infer $\sigma(\mathbf{A}_{\ast}) = \sigma_{e5}(\mathbf{A}_{\ast})$ from Proposition \ref{essspectrumC}. Using the notation of Proposition \ref{RDODEspec}, we consider $X_2:= C(\overline{\Omega})^{m}$ and $X_1:=C(\overline{\Omega})^{k}$ instead and, hence, only bounded multiplication operators are used on these spaces apart from the elliptic operator. This yields $\sigma(\mathbf{L}^c) = \sigma (\mathbf{A}_{\ast}) \mathop{\dot{\cup}} \Sigma^c$ for some set of eigenvalues $\Sigma^c \subset \sigma_p(\mathbf{L}^c)$.\\ 
The inclusion $\sigma(\mathbf{L}^c) \subset \sigma(\mathbf{L})$ is clear from $\mathcal{D}(A_c) \subset \mathcal{D}(A_p)$ and the fact that the spectrum of the multiplication operator $\mathbf{A}_\ast$ is the same on $L^p(\Omega)^m$ and $C(\overline{\Omega})^m$ by Proposition \ref{specunion} and \cite[Appendix B.2]{KMCMnonlinear}. 
Hence, it remains to verify $\Sigma \subset \Sigma^{c}$. For $\lambda \in \Sigma \subset \rho(\mathbf{A}_{\ast})$, there exists an eigenfunction $\xi \in L^\infty(\Omega)^m \times \mathcal{D}(A_p)^k$ which is an element of $L^\infty(\Omega)^m \times C(\overline{\Omega})^k$ for $p>n^\ast$. Using the explicit form \eqref{resolvexplicit} of the resolvent leads to 
$\xi_1 = (\lambda I - \mathbf{A}_{\ast})^{-1} \mathbf{B}_{\ast} \xi_2 \in C(\overline{\Omega})^m$ and hence $\lambda \in \sigma_p(\mathbf{L}^c)$ by $\xi_2 \in \mathcal{D}(A_c)^k$.
\end{proof}


\subsection{Semigroup properties} \label{sec:semigroupest}

In the case of a continuous stationary solution, it is a simple consequence of the spectral mapping theorem for analytic semigroups that the growth bound of the semigroup $(\mathbf{T}(t))_{t\in\mathbb{R}_{\ge0}}$ restricted to $C(\overline{\Omega})^{m+k}$ equals the spectral bound of the linear operator $\mathbf{L}^c$ defined in Lemma \ref{Lcontinuous}.
The aim of this section is to connect the growth bound of the semigroup $(\mathbf{T}(t))_{t\in\mathbb{R}_{\ge0}}$ restricted to $L^\infty(\Omega)^{m+k}$ to the spectral bound of the linear operator $\mathbf{L}$.
This allows showing stability results in Section \ref{sec:nonlinearstab} in the case of a discontinuous stationary solution. We make use of the parts $\mathbf{L}_{\infty}$ in $Z_\infty=L^\infty(\Omega)^{m+k}$ and $\mathbf{L}_c$ in $Z_c=L^\infty(\Omega)^{m} \times C(\overline{\Omega})^k$ of the operator $\mathbf{L}$ defined in Lemma \ref{Lclosed}. 
Since $(\mathbf{T}(t))_{t \in \mathbb{R}_{\ge 0}}$ is analytic on $L^\infty(\Omega)^{m} \times L^p(\Omega)^{k}$ by Lemma \ref{Lclosed}, the spectral mapping theorem 
\[
\mathrm{e}^{\sigma(\mathbf{L})t} = \sigma(\mathbf{T}(t)) \setminus \{0\} \qquad \forall \; t \in \mathbb{R}_{\ge 0}
\] 
is valid \cite[Ch. IV, Corollary 3.12]{Engel}. From this, we infer that the spectral bound 
\[
s(\mathbf{L}) := \sup \{ \mathrm{Re} \, \lambda \mid \lambda \in \sigma(\mathbf{L})\}
\]
of the operator $\mathbf{L}$ equals the growth bound $w_0^p$ of $(\mathbf{T}(t))_{t \in \mathbb{R}_{\ge 0}}$, i.e.,
\[
s(\mathbf{L}) = w_0^p := \inf \{ w \in \mathbb{R} \mid \exists \, M_w \ge 1 \; \text{such that} \; \|\mathbf{T}(t)\|_p \le M_w \mathrm{e}^{w t} \; \text{for all} \; t \in \mathbb{R}_{\ge 0} \}.
\]
First, we show that the analytic semigroup $(\mathbf{T}(t))_{t \in \mathbb{R}_{\ge 0}}$ on $L^\infty(\Omega)^{m} \times L^p(\Omega)^{k}$ from Lemma \ref{Lclosed} is $Z_\infty$-invariant for $n^\ast = \max\{n/2,2\} <p<\infty$ and can be restricted to $Z_\infty$.

\begin{lemma} \label{restrsemigroup}
The semigroup $(\mathbf{T}_\infty(t))_{t \in \mathbb{R}_{\ge 0}}$ defined as the restriction of $(\mathbf{T}(t))_{t \in \mathbb{R}_{\ge 0}}$ to $L^\infty(\Omega)^{m+k}$ is not strongly continuous. However, the map $t \mapsto \mathbf{T}_\infty(t) \in \mathcal{L}(L^\infty(\Omega)^{m+k})$ is norm-continuous for all $t>0$ and is exponentially bounded, i.e., there exist constants $M \ge 1$, $w \in \mathbb{R}$ such that 
\begin{align}
	\| \mathbf{T}_\infty(t)\| \le M \mathrm{e}^{w t} \qquad \forall \; t \in  \mathbb{R}_{\ge 0}. \label{expgrowthTI}
\end{align}
\end{lemma}

\begin{proof}
{Since  $(\mathbf{T}(t))_{t \in \mathbb{R}_{\ge 0}}$ is an analytic semigroup on $L^\infty(\Omega)^{m} \times L^p(\Omega)^{k}$, it follows that $\mathbf{T}(t)$ maps into $\mathcal{D}(\mathbf{L})$ for all $t>0$ \cite[Corollary 3.7.21]{Arendt}. 
	Recall $\mathcal{D}(\mathbf{L}) \subset L^\infty(\Omega)^m \times C(\overline{\Omega})^k$ for $p>n^\ast$ from Proposition \ref{specrestricted}. Since $\mathbf{T}(0)= I$, we obtain $\mathbf{T}(t)L^\infty(\Omega)^{m+k} \subset L^\infty(\Omega)^{m+k}$ for all $t\ge0$ and the restricted semigroup $(\mathbf{T}_\infty(t))_{t \in \mathbb{R}_{\ge 0}}$ is well-defined.\\ 
	According to \cite[Theorem 1.6]{Tanabe}
	, the resolvent estimate in Lemma \ref{resolventestrestr} implies existence of a semigroup $(\mathbf{T}_|(t))_{t \in \mathbb{R}_{> 0}}$ of bounded operators on $L^\infty(\Omega)^{m+k}$ which is generated by the closed operator $\mathbf{L}_{\infty}$. Moreover, the semigroup can be extended to an analytic semigroup on some sector in $\mathbb{C}$ which includes the positive half-line $\{t \mid t \in \mathbb{R}_{>0}\}$, compare also \cite[Corollary 3.7.17]{Arendt}. By \cite[Definition 3.2.5]{Arendt}, this semigroup is strongly continuous on $\mathbb{R}_{>0}$ and exponentially bounded. However, the semigroup $(\mathbf{T}_|(t))_{t \in \mathbb{R}_{\ge 0}}$ with $\mathbf{T}_|(0):= I$ is not strongly continuous on $\mathbb{R}_{\ge 0}$ by Lemma \ref{restrclosed}. Analyticity of the semigroup $(\mathbf{T}(t))_{t \in \mathbb{R}_{\ge 0}}$ yields that $\mathbf{T}(t)$ is given by a complex contour integral over the resolvent $R(\lambda, \mathbf{L})$ and a similar formula is valid for $(\mathbf{T}_|(t))_{t \in \mathbb{R}_{\ge 0}}$ in terms of $R(\lambda, \mathbf{L}_\infty)$ \cite[Theorem 1.6]{Tanabe}. Since the restriction of the resolvent $R(\lambda, \mathbf{L})$ to $L^\infty (\Omega)^{m+k}$ equals $R(\lambda, \mathbf{L}_\infty)$ by Proposition \ref{specrestricted}, we obtain that $(\mathbf{T}_|(t))_{t \in \mathbb{R}_{\ge 0}}$ and $(\mathbf{T}_\infty(t))_{t \in \mathbb{R}_{\ge 0}}$ coincide on $L^\infty(\Omega)^{m+k}$. Note that we can choose the same path of integration in view of $\rho(\mathbf{L}) = \rho(\mathbf{L}_\infty)$ from Proposition \ref{specrestricted}. By analyticity of $(\mathbf{T}_\infty(t))_{t \in \mathbb{R}_{> 0}}$, it follows that $t \mapsto \mathbf{T}_\infty(t) \in \mathcal{L}(L^\infty(\Omega)^{m+k})$ is norm-continuous for $t>0$.} 
\end{proof}

To show equality of the growth bound $w_0^\infty$ of the semigroup $(\mathbf{T}_\infty(t))_{t \in \mathbb{R}_{\ge 0}}$ with the spectral bound $s(\mathbf{L})$, we use the fact that each semigroup operator $\mathbf{T}_\infty(t)$ is the limit of regularized semigroup operators \cite{Bobrowski97}. 
This regularization is obtained by a further restriction of the semigroup $(\mathbf{T}_{\infty}(t))_{t \in \mathbb{R}_{\ge 0}}$ to the closed subspace $L^\infty(\Omega)^{m} \times C(\overline{\Omega})^k$ of $L^\infty(\Omega)^{m+k}$. Although it could not be verified that the Jacobian $\mathbf{J}$ is a $\mathbf{D}\Delta$-bounded perturbation on $L^\infty(\Omega)^{m} \times C(\overline{\Omega})^k$, the operator $\mathbf{L}_c$ is the generator of an analytic semigroup.

\begin{lemma} \label{restrsemigroup2}
The semigroup $(\mathbf{T}_\infty(t))_{t \in \mathbb{R}_{\ge 0}}$ from Lemma \ref{restrsemigroup} defined on $L^\infty(\Omega)^{m+k}$ can be restricted to a strongly continuous semigroup $(\mathbf{T}_c(t))_{t \in \mathbb{R}_{\ge 0}}$ on $Z_c =  L^\infty(\Omega)^{m} \times C(\overline{\Omega})^k$. This semigroup is even analytic and is generated by the operator $\mathbf{L}_c$ defined in \eqref{definitionL0}. Moreover, the semigroup $(\mathbf{T}_c(t))_{t \in \mathbb{R}_{\ge 0}}$ coincides with the semigroup $(\mathbf{T}(t))_{t \in \mathbb{R}_{\ge 0}}$ defined in Lemma \ref{Lclosed} restricted to $L^\infty(\Omega)^{m} \times C(\overline{\Omega})^k$.
\end{lemma}

\begin{proof} 
We already considered the restriction of $\mathbf{L}_{\infty}$ to $Z_c$ in Proposition \ref{specrestricted}. Recall 
\begin{align}
	R(\lambda, \mathbf{L}_c) = R(\lambda, \mathbf{L}_{\infty})_{|Z_c} = R(\lambda, \mathbf{L})_{|Z_c}\qquad \forall \; \lambda \in \rho(\mathbf{L}). \label{resolvent1}
\end{align}
Since $\|R(\lambda, \mathbf{L}_c)\| \le \|R(\lambda, \mathbf{L}_{\infty})\|$ due to $Z_c \subset L^\infty(\Omega)^{m+k}$, the resolvent estimate derived in Lemma \ref{resolventestrestr} holds also for $R(\lambda, \mathbf{L}_c)$. The results from Lemma \ref{resolventestrestr} for $\mathbf{L}_c$ and density of the domain $\mathcal{D}(\mathbf{L}_c)$ in $Z_c$ from Lemma \ref{restrclosed2} imply existence of an analytic semigroup $(\mathbf{T}_c(t))_{t \in \mathbb{R}_{\ge 0}}$ generated by $\mathbf{L}_c$ \cite[Theorem 1.6]{Tanabe}.\\ 
In order to deduce that $(\mathbf{T}_c(t))_{t \in \mathbb{R}_{\ge 0}}$ is the restriction of the semigroup $(\mathbf{T}(t))_{t \in \mathbb{R}_{\ge 0}}$ to $Z_c$, we apply \cite[Ch. 4, Theorem 5.5]{Pazy}. In view of the generator property of the part $\mathbf{L}_c$ of $\mathbf{L}$ in $Z_c$, it remains to verify that the space $Z_c$ is invariant for the resolvent $R(\lambda, \mathbf{L})$. However, this follows from \eqref{resolvent1}, 
and we obtain $\mathbf{T}(t)_{|Z_c} = \mathbf{T}_c(t)$ for all $t \in \mathbb{R}_{\ge 0}$. As the semigroup $(\mathbf{T}_\infty(t))_{t \in \mathbb{R}_{\ge 0}}$ defined in Lemma \ref{restrsemigroup} is a restriction of the semigroup $(\mathbf{T}(t))_{t \in \mathbb{R}_{\ge 0}}$ to $L^\infty(\Omega)^{m+k}$, we also infer $\mathbf{T}_\infty(t)_{|Z_c}=\mathbf{T}_c(t)$ for all $t \in \mathbb{R}_{\ge 0}$.
\end{proof}

\begin{lemma} \label{limitrep}
The semigroup $(\mathbf{T}_\infty(t))_{t \in \mathbb{R}_{\ge 0}}$ is determined by its restriction on the space $L^\infty(\Omega)^{m} \times C(\overline{\Omega})^k$ due to the pointwise limit
\begin{align}
	\mathbf{T}_\infty(t) \xi = \lim_{\lambda \to \infty} \mathbf{T}_c(t) \lambda R(\lambda, \mathbf{L}_{\infty}) \xi \qquad \forall \; t \in \mathbb{R}_{>0}, \xi \in L^\infty(\Omega)^{m+k}. \label{semigrouplimit}
\end{align} 
\end{lemma}

\begin{proof}
{According to \cite[p. 125]{Arendt} and properties derived for the semigroup $(\mathbf{T}_\infty(t))_{t \in \mathbb{R}_{\ge 0}}$ in Lemma \ref{restrsemigroup}, the family $(\mathbf{U}(t))_{t \in \mathbb{R}_{\ge 0}}$ of linear, bounded operators on $Z_\infty$ given by
	\[
	\mathbf{U}(t):= \int_0^t \mathbf{T}_\infty(\tau)\; \mathrm{d}\tau \qquad \forall \; t \in \mathbb{R}_{\ge 0}
	\] 
	is a once integrated semigroup \cite[Definition 3.2.1]{Arendt}. Moreover, following \cite[Lemma 2.1.6]{Lunardi}, the semigroup $(\mathbf{U}(t))_{t \in \mathbb{R}_{\ge 0}}$ is generated by $\mathbf{L}_{\infty}$ and the resolvent is given by
	\[
	R(\lambda, \mathbf{L}_{\infty}) = \lambda \int_0^\infty \mathrm{e}^{-\lambda \tau} \mathbf{U}(\tau) \; \mathrm{d}\tau =   \int_0^\infty \mathrm{e}^{-\lambda \tau} \mathbf{T}_|(\tau) \; \mathrm{d}\tau \qquad \forall \; \lambda > w.
	\]
	By growth estimate \eqref{expgrowthTI}, the once integrated semigroup $(\mathbf{U}(t))_{t \in \mathbb{R}_{\ge 0}}$ satisfies estimates of the resolvent derivatives in \cite[Lemma 1]{Bobrowski97}. By choosing $\omega=w+\varepsilon$ and $\varphi \in L^1(\mathbb{R}_{\ge 0}) \cap L^\infty(\mathbb{R}_{\ge 0})$ with $\varphi(t)=M \mathrm{e}^{-\varepsilon t}$ for a small parameter $\varepsilon>0$ and recalling estimate \eqref{resolventest}, we have verified all assumptions of \cite[Theorem 3]{Bobrowski97}.} By analyticity of the semigroup $(\mathbf{T}_c(t))_{t \in \mathbb{R}_{\ge 0}}$, there holds $D=Z_c$ in \cite[Lemma 2]{Bobrowski97} and condition e) in \cite[Theorem 3]{Bobrowski97} applies for $\mathcal{D}(\mathbf{L}_{\infty}) \subset Z_c$. Then the limit \eqref{semigrouplimit} is a consequence of the latter theorem, see property (3.5) in \cite[Theorem 3]{Bobrowski97}.
\end{proof}

From this proof it is clear that the semigroup $(\mathbf{T}_\infty(t))_{t \in \mathbb{R}_{\ge 0}}$ is the non-continuous semigroup constructed in \cite[Theorem 3]{Bobrowski97}, which is uniquely determined by the operator $\mathbf{L}_{\infty}$. Finally, the latter result can be used to relate the spectral bound of the operator $\mathbf{L}_c$ to the growth bound of the semigroup $(\mathbf{T}_\infty(t))_{t \in \mathbb{R}_{\ge 0}}$.

\begin{proposition} \label{SBeGB}
The spectral bound $s(\mathbf{L})$ of the operator $\mathbf{L}$ equals the growth bound $w_0^\infty$ of the restricted semigroup $(\mathbf{T}_\infty(t))_{t \in \mathbb{R}_{\ge 0}}$ on $L^\infty(\Omega)^{m+k}$. 
\end{proposition}

\begin{proof}
Let us first show $w_0^\infty \le s(\mathbf{L})$. By \cite[Ch. IV, Corollary 3.12]{Engel}, the spectral mapping theorem holds for the analytic semigroup $(\mathbf{T}_c(t))_{t \in \mathbb{R}_{\ge 0}}$. Proposition \ref{specrestricted} implies $s(\mathbf{L})= s(\mathbf{L}_c)=w_0^c$ for the growth bound $w_0^c$ of the semigroup $(\mathbf{T}_c(t))_{t \in \mathbb{R}_{\ge 0}}$. Consequently, by definition of the growth bound
, for each $\varepsilon>0$ there exists a constant $C_\varepsilon \ge 1$ such that
\begin{align*}
	\|\mathbf{T}_c(t) \xi \|_\infty \le C_\varepsilon \mathrm{e}^{(s(\mathbf{L}) +\varepsilon)t} \|\xi\|_\infty \qquad \forall \; \xi \in L^\infty(\Omega)^{m} \times C(\overline{\Omega})^k, t \in \mathbb{R}_{\ge 0}. 
\end{align*} 
Using the limit \eqref{semigrouplimit} as well as the resolvent estimate \eqref{resolventest}, we deduce for $t \in \mathbb{R}_{> 0}$ that
\[
\|\mathbf{T}_\infty(t) \xi \|_\infty = \lim_{\lambda \to \infty} \| \mathbf{T}_c(t) \lambda R(\lambda, \mathbf{L}_{\infty}) \xi \|_\infty \le M C_\varepsilon \mathrm{e}^{(s(\mathbf{L}) +\varepsilon)t} \|\xi\|_\infty \quad \forall \; \xi \in L^\infty(\Omega)^{m+k}.
\]
This estimate shows that the growth bound $w_0^\infty$ of the semigroup $(\mathbf{T}_\infty(t))_{t \in\mathbb{R}_{\ge 0}}$ satisfies the inequality $w_0^\infty \le s(\mathbf{L}) +\varepsilon$. Letting $\varepsilon \to 0$ yields the first claim $w_0^\infty \le s(\mathbf{L})$.\\
It remains to prove the converse inequality $s(\mathbf{L}_{\infty}) \le w_0^\infty$. Let us assume that $s(\mathbf{L}_{\infty})> w_0^\infty$. Consider an arbitrary $\lambda \in \mathbb{C}$ with $\mathrm{Re}\, \lambda>w$, where $w \in \mathbb{R}$ with $w_0^\infty < w < s(\mathbf{L}_{\infty})$. From the definition of the growth bound $w_0^\infty$, we know that there is a constant $M \ge 1$ such that
$
\|\mathbf{T}_\infty(t)\|_\infty \le M \mathrm{e}^{w t}$ for all $t \in \mathbb{R}_{\ge 0}.
$
This estimate and continuity of the map $t \mapsto \mathbf{T}_\infty(t) \in \mathcal{L}(L^\infty(\Omega)^{m+k})$ for $t>0$ implies that the improper Riemann integral
\[
R(\lambda) f := \int_0^\infty \mathrm{e}^{-\lambda \tau} \mathbf{T}_\infty(\tau) f \; \mathrm{d}\tau
\]
exists for each $f \in Z_\infty=L^\infty(\Omega)^{m+k}$ with the estimate $\|R(\lambda)\|_\infty \le M/(\mathrm{Re}\, \lambda -w)$. We follow the proof of \cite[Ch. II, Theorem 1.10]{Engel} to obtain a contradiction to $w<s(\mathbf{L}_{\infty})$. We will verify that $\lambda \in \rho(\mathbf{L}_{\infty})$ and $R(\lambda)$ coincides with the resolvent of $\mathbf{L}_{\infty}$ at $\lambda$.\\
Since $Z_\infty \hookrightarrow L^\infty(\Omega)^m \times L^p(\Omega)^k$, we can consider the integral $R(\lambda)f$ as an element in $L^\infty(\Omega)^m \times L^p(\Omega)^k$ for each $f \in Z_\infty$. Using the semigroup $(\mathbf{T}(t))_{t \in \mathbb{R}_{\ge 0}}$, one shows that $R(\lambda)f \in \mathcal{D}(\mathbf{L})$ with
$
(\lambda I-\mathbf{L})R(\lambda)f =f, 
$
see the proof of \cite[Ch. II, Theorem 1.10]{Engel}. This implies that $R(\lambda)$ is the right inverse of $\lambda I - \mathbf{L}_{\infty}$. For the left inverse, let us consider $f \in \mathcal{D}(\mathbf{L}_{\infty})$ and the sequence $(f_j)_{j \in \mathbb{N}} \subset Z_\infty$ given by the proper integrals
\[
f_j:= \int_0^j  \mathrm{e}^{-\lambda \tau} \mathbf{T}_\infty(\tau) f \; \mathrm{d}\tau.
\] 
This sequence converges to $R(\lambda)f$ in $Z_\infty$ by definition of $R(\lambda)$.
Considering again the generator $(\mathbf{L}, \mathcal{D}(\mathbf{L}))$ of $(\mathbf{T}(t))_{t \in \mathbb{R}_{\ge 0}}$, \cite[Ch. II, Lemma 1.9]{Engel} implies $f_j \in \mathcal{D}(\mathbf{L}_{\infty})$ with
\[
(\lambda I - \mathbf{L}_{\infty})f_j =(\lambda I - \mathbf{L})f_j =  \int_0^j  \mathrm{e}^{-\lambda \tau} \mathbf{T}(\tau) (\lambda I - \mathbf{L})f \; \mathrm{d}\tau = \int_0^j  \mathrm{e}^{-\lambda \tau} \mathbf{T}_\infty(\tau) (\lambda I - \mathbf{L}_{\infty})f \; \mathrm{d}\tau. 
\]
By definition of $R(\lambda)$, this shows that $(\lambda I - \mathbf{L}_{\infty}) f_j \to R(\lambda)(\lambda I-\mathbf{L}_{\infty})f$ in $Z_\infty$ as $j \to \infty$. Closedness of $\mathbf{L}_{\infty}$ as shown in Lemma \ref{restrclosed} implies $R(\lambda)f= \lim_{j \to \infty}f_j \in \mathcal{D}(\mathbf{L}_{\infty})$ with 
\[
(\lambda I - \mathbf{L}_{\infty}) R(\lambda) f =(\lambda I -\mathbf{L}_{\infty}) \lim_{j \to \infty} f_j =\lim_{j \to \infty} (\lambda I -\mathbf{L}_{\infty}) f_j  =R(\lambda) (\lambda I - \mathbf{L}_{\infty}) f.
\]
In summary, $\lambda \in \rho(\mathbf{L}_{\infty})$ with resolvent $R(\lambda, \mathbf{L}_{\infty})= R(\lambda)$. The arbitrariness of $\lambda \in \mathbb{C}$ with $\mathrm{Re}\, \lambda>w$ and $w_0^\infty < w < s(\mathbf{L}_{\infty})$ finally yields $s(\mathbf{L}_{\infty}) \le w_0^\infty$.
\end{proof}


\section{Nonlinear stability analysis} \label{sec:nonlinearstab}

{After a linearization at a stationary solution $(\overline{\mathbf{u}},\overline{\mathbf{v}}) \in L^\infty(\Omega)^{m} \times C(\overline{\Omega})^k$ of the reaction-diffusion-ODE system \eqref{fullsys}, we reduced our considerations to equation \eqref{LNequation} with the shift $\xi = (\mathbf{u},\mathbf{v}) - (\overline{\mathbf{u}},\overline{\mathbf{v}})$ and the corresponding trivial steady state $\xi \equiv \mathbf{0}$, i.e.,
\begin{align*}
	\frac{\partial \xi}{\partial t} = \mathbf{L} \xi + \mathbf{N}(\xi), \qquad \xi(0) = \xi^0 \in L^\infty(\Omega)^{m+k}. 
\end{align*} 
Recall from Appendix \ref{sec:exmildsol} that a mild solution of this equation satisfies the implicit integral equation 
\begin{align*}
	\xi(\cdot,t) = \mathbf{T}(t)\xi^0 + \int_0^t \mathbf{T}(t-\tau) \mathbf{N}(\xi(\cdot,\tau)) \; \mathrm{d}\tau. 
\end{align*}
It is well-known that the control of the growth of the semigroup is fundamental for estimating the solution of the nonlinear problem. However, even for the linear case asymptotic stability and exponential asymptotic stability are non-equivalent concepts \cite{Xu}. 
The detailed examination of the operator $\mathbf{L}$ on $L^\infty(\Omega)^{m+k}$ and its corresponding semigroup $(\mathbf{T}_\infty(t))_{t \in \mathbb{R}_{\ge 0}}$ allows considering the implicit integral equation on this space, too. Recall that the nonlinear term $\mathbf{N}$ can be controlled on $L^\infty(\Omega)^{m+k}$ by estimates given in Lemma \ref{nonlinearity}. The nonlinear operator $\mathbf{N}$ is} locally Lipschitz continuous and satisfies $\mathbf{N}(\xi) =o(\|\xi\|_\infty)$ as $\|\xi\|_\infty \to 0$, i.e., there exist a constant $\rho>0$ and a continuous increasing function $b: \mathbb{R}_{\ge 0} \to \mathbb{R}_{\ge 0}$ with $b(0)=0$ such that 
\begin{align}
\|\mathbf{N}(\xi)\|_\infty \le b(\rho) \|\xi\|_\infty \qquad \forall \; \xi \in L^\infty(\Omega)^{m+k}, \| \xi\|_\infty < \rho. \label{estimateN2} 
\end{align}
This indicates that the behavior of a solution to equation \eqref{LNequation} is mainly governed by the growth of the semigroup $(\mathbf{T}(t))_{t \in \mathbb{R}_{\ge 0}}$ on $L^\infty(\Omega)^{m+k}$. Recall that the semigroup $(\mathbf{T}(t))_{t \in \mathbb{R}_{\ge 0}}$ restricted to $L^\infty(\Omega)^{m+k}$ is not strongly continuous. However, this semigroup still has a smoothing property for $t>0$ and the growth bound is determined by the spectral bound of $\mathbf{L}$, i.e., 
the equality $s(\mathbf{L})= w_0^\infty$ holds by Proposition \ref{SBeGB}. This leads us to the stability result in Theorem \ref{stab} and the instability result in Theorem \ref{instabtheorem} in the case of a discontinuous stationary solution. Note that both proofs rely on temporal continuity of the solution in $L^\infty(\Omega)^{m+k}$ for $t>0$. For this reason, we use the following notion of stability in $L^\infty(\Omega)^{m+k}$, similar to \cite[Ch. 9, p. 337]{Lunardi}.

\begin{definition} \label{stability}
The trivial solution $\xi \equiv \mathbf{0}$ of equation \eqref{LNequation} is called \emph{nonlinearly stable} in $L^\infty(\Omega)^{m+k}$ if for each $\varepsilon>0$ there exists a constant $\delta > 0$ such that an upper bound $\|\xi^0 \|_\infty < \delta$ for initial values $\xi^0 \in L^\infty(\Omega)^{m+k}$ implies existence and uniqueness of a mild solution $\xi$ in the sense of Definition \ref{mildsol} which satisfies
\begin{itemize}
	\item[(i)] the solution $\xi$ is defined globally in time, i.e., $T_{\max}=\infty$;
	
	\item[(ii)] $\|\xi(\cdot,t)\|_\infty < \varepsilon$ for all $t \in \mathbb{R}_{> 0}$.
\end{itemize}
The trivial solution $\xi \equiv \mathbf{0}$ is called \emph{nonlinearly unstable} if
it is not nonlinearly stable. Furthermore, a stationary solution $(\overline{\mathbf{u}},\overline{\mathbf{v}}) \in L^\infty(\Omega)^{m+k}$ of system \eqref{fullsys} is called nonlinearly stable (nonlinearly unstable) if the corresponding trivial solution of equation \eqref{LNequation} is nonlinearly stable (nonlinearly unstable).
\end{definition}

If the steady state $(\overline{\mathbf{u}},\overline{\mathbf{v}}) \in L^\infty(\Omega)^{m} \times C(\overline{\Omega})^k$ has a discontinuity, then a sufficiently small $\delta$-neighborhood in Definition \ref{stability} consists only of discontinuous initial conditions $\xi^0$ as well. 
Note that this notion of stability is stronger than in \cite[Section 2]{Friedlander}, in which condition (ii) holds only for almost every $t \in \mathbb{R}_{\ge 0}$. Further note that possible jump-discontinuities of the ODE component $\overline{\mathbf{u}}$ are not cut out in the used topology in Definition \ref{stability}. Hence, the notion of this work differs from $(\varepsilon, A)$ stability considered in \cite{HMCT, Takagi21, Weinberger}. That approach is used for a partly continuous approximation of the steady state but excludes a neighborhood of the discontinuity of the steady state. 

\begin{remark} \label{rem:regmildsol}
The regularity of the mild solution strongly depends on the regularity of the initial condition. Depending on the regularity of the initial perturbation, we consider three different notions of stability -- stability in $L^\infty(\Omega)^{m+k}$ in the sense of Definition \ref{stability}, Lyapunov stability in $L^\infty(\Omega)^m \times C(\overline{\Omega})^k$ or in $C(\overline{\Omega})^{m+k}$.

Recall that a solution of equation \eqref{fullsys} with initial values $\xi^0 \in L^\infty(\Omega)^{m+k}$ may be discontinuous at time $t=0$ with respect to $L^\infty(\Omega)^{m+k}$, see Proposition \ref{Rothesol}. Nevertheless, 
\[
\xi \in C(\mathbb{R}_{\ge 0}; L^\infty(\Omega)^m \times L^p(\Omega)^k) \cap C(\mathbb{R}_{> 0}; L^\infty(\Omega)^m \times C(\overline{\Omega})^k)
\]
by analyticity of the semigroup $(\mathbf{T}(t))_{t \in \mathbb{R}_{\ge 0}}$ from Lemma \ref{restrsemigroup} and \cite[Proposition 1.3.4]{Arendt}. 

In order to determine Lyapunov stability in $L^\infty(\Omega)^m \times C(\overline{\Omega})^k$, problem \eqref{LNequation} is endowed with initial values $\xi^0 \in L^\infty(\Omega)^m \times C(\overline{\Omega})^k$. Since the 
semigroup $(\mathbf{T}(t))_{t \in \mathbb{R}_{\ge 0}}$ restricted to the space $L^\infty(\Omega)^m \times C(\overline{\Omega})^k$ is strongly continuous by Lemma \ref{restrsemigroup2}, the solution $\xi$ of system \eqref{LNequation} satisfies $\xi \in C([0, T_{\max}); L^\infty(\Omega)^m \times C(\overline{\Omega})^k)$ by Proposition \ref{Rothesol} and \cite[Proposition 1.3.4]{Arendt}. Hence, the inequality in condition (ii) of Definition \ref{stability} holds for all $t \in \mathbb{R}_{\ge 0}$ in the classical sense of Lyapunov.  

In the case of a continuous steady state $(\overline{\mathbf{u}},\overline{\mathbf{v}}) \in C(\overline{\Omega})^{m+k}$, one may consider stability as in Definition \ref{stability} or Lyapunov stability in $L^\infty(\Omega)^m \times C(\overline{\Omega})^k$. Under Assumption \ref{ass:Nc}, we may also consider Lyapunov stability in $C(\overline{\Omega})^{m+k}$, see Corollary \ref{stabC} and Corollary \ref{instabC}. In this case, initial conditions $\xi^0 \in C(\overline{\Omega})^{m+k}$ are prescribed. By Lemma \ref{Lcontinuous} and validity of estimate \eqref{estimateN} for the nonlinear remainder in $C(\overline{\Omega})^{m+k}$, the solution $\xi$ of problem \eqref{LNequation} satisfies $\xi \in C([0, T_{\max}); C(\overline{\Omega})^{m+k})$. Here, we use a Picard iteration in $C(\overline{\Omega})^{m+k}$ in the proof of Proposition \ref{Rothesol}, compare also \cite[Ch. 6, Theorem 1.4]{Pazy}. Note that, compared to stability with respect to $W^{1,p}(\Omega)^{m+1}$ considered in \cite{CMCKSregular, MKS17}, we allow for a larger class of perturbations in Definition \ref{stability}. 
\end{remark}

{The works \cite{CMCKSirregular, KMCT20, Takagi21} consider Lyapunov stability in the space $L^\infty(\Omega)^m \times C(\overline{\Omega})$ as perturbations are smoother for the diffusing component. We briefly depict the possibility of perturbing a stationary solution with bounded initial conditions that are small with respect to the norm on $L^\infty(\Omega)^m \times L^p(\Omega)^k$ or $L^p(\Omega)^{m+k}$ for a certain model in Remark \ref{rem:Lpstability}. Since the nonlinear remainder even does not need to be well-defined on these spaces, this kind of perturbations is far from being understood. 
It remains an open question whether the transfer from linear to nonlinear stability is possible on these spaces. In that respect, we also refer to numerical simulations illustrated in \cite[Figure 4]{KMCT20}.}

\subsection{Stability of steady states} \label{sec:stab}

This section is devoted to stability of a bounded stationary solution $(\overline{\mathbf{u}}, \overline{\mathbf{v}}) \in L^\infty(\Omega)^m \times C(\overline{\Omega})^k$ of the reaction-diffusion-ODE system \eqref{fullsys}. We will show that a negative spectral bound of the linearized operator $\mathbf{L}$ in \eqref{LNequation} implies a {local exponential stability result. Consequently, unlike the specific study \cite{He},} we do not treat the degenerated case in which $0 \in \sigma(\mathbf{L})$. The section is based on the stability result \cite[Proposition 4.17]{Webb}. Since stability of a continuous steady state can be deduced from stability in $L^\infty(\Omega)^{m+k}$ or $L^\infty(\Omega)^m \times C(\overline{\Omega})^k$, we focus on the case of a discontinuous steady state for which the choice of the function space is challenging. The nonlinear operator $\mathbf{N}$ defined by equation \eqref{LNequation} is not well-defined on $L^\infty(\Omega)^m \times C(\overline{\Omega})^k$ for discontinuous steady states, 
while the nonlinear operator is continuous on $L^\infty(\Omega)^{m+k}$. On this space, however, the semigroup $(\mathbf{T}_\infty(t))_{t \in \mathbb{R}_{\ge 0}}$ is not strongly continuous and \cite{Webb} cannot directly be applied. Nevertheless, we can adapt the proof by applying semigroup estimates derived in the last section.

\begin{theorem} \label{stab}
Consider a steady state $(\overline{\mathbf{u}}, \overline{\mathbf{v}}) \in L^\infty(\Omega)^m \times C(\overline{\Omega})^k$ of system \eqref{fullsys}. Let Assumption \ref{ass:N} be satisfied. If the linearized operator $\mathbf{L}$ from Lemma \ref{Lclosed} has a negative spectral bound $s(\mathbf{L})<0$, then, the steady state is nonlinearly stable in the sense of Definition \ref{stability}. More precisely, we obtain local exponential stability in the sense that there exist some constants $\delta, \beta>0$ and $M \ge 1$ such that for initial conditions $(\mathbf{u}^0, \mathbf{v}^0) \in L^\infty(\Omega)^{m+k}$ with $\|(\mathbf{u}^0, \mathbf{v}^0)-(\overline{\mathbf{u}}, \overline{\mathbf{v}}) \|_\infty \le \delta$ the mild solution $\xi =(u,v)$ of system \eqref{fullsys} is globally defined and satisfies the estimate
\begin{align}
	\|(\mathbf{u}(\cdot,t), \mathbf{v}(\cdot,t)) - (\overline{\mathbf{u}}, \overline{\mathbf{v}})\|_\infty \le M \mathrm{e}^{-\beta t} \|(\mathbf{u}^0, \mathbf{v}^0)-(\overline{\mathbf{u}}, \overline{\mathbf{v}}) \|_\infty \quad \forall \; t \in \mathbb{R}_{\ge 0}. \label{expstab}
\end{align}
\end{theorem}

\begin{proof}
We adapt the proof of \cite[Proposition 4.17]{Webb} to the linearized equation \eqref{LNequation} at $\xi=\mathbf{0}$. The proof is based on the choice of some constants $0<w<-w_0^\infty=-s(\mathbf{L})$ and $M=M_w \ge 1$ such that 
\[
\|\mathbf{T}(t) \xi\|_\infty \le M \mathrm{e}^{-w t} \|\xi\|_\infty \qquad \forall \; \xi \in L^\infty(\Omega)^{m+k}, t \in \mathbb{R}_{\ge 0}.
\]
In view of estimate \eqref{estimateN2}, we can choose $r>0$ such that $\|\mathbf{N}(\xi)\|_\infty \le \frac{w}{2M} \|\xi\|_\infty$ for all $\xi\in  L^\infty(\Omega)^{m+k}$ with $\|\xi\|_\infty < r$. Let $\delta = \frac{r}{M}$ and fix $\xi^0 \in  L^\infty(\Omega)^{m+k}$ with $\|\xi^0\|_\infty <\delta$. By Remark \ref{rem:regmildsol}, there exists a unique mild solution 
\[
\xi \in C([0,T_{\max});L^\infty(\Omega)^m \times L^p(\Omega)^k) \cap C((0,T_{\max});L^\infty(\Omega)^m \times C(\overline{\Omega})^k)
\]
of problem \eqref{LNequation} on a maximal existence interval $[0,T_{\max})$, where $T_{\max}$ depends on the initial conditions $\xi^0$. Let us define
\begin{align*}
	T:= \sup\{t>0 \mid  \|\xi(\cdot,\tau)\|_\infty \le r \ \text{for all}\  0<\tau\le t\}.
\end{align*}
Since $(\mathbf{T}_\infty(t))_{t\ge 0}$ is not strongly continuous, it is not obvious that $T< T_{\max}$ is a positive real number. The mild solution $\xi$ satisfies the implicit integral formulation \eqref{implicitT}. Estimating each term yields 
\begin{align*}
	\|\xi(t) \|_\infty & \le \|\mathbf{T}_{\infty}(t) \xi^0\|_\infty + \int_0^t \|\mathbf{T}_{\infty}(t-\tau) \mathbf{N}(\xi(\tau))\|_\infty \; \mathrm{d}\tau \\
	& \le M \mathrm{e}^{-w t} \|\xi^0 \|_\infty + M\int_0^t \mathrm{e}^{-w (t-\tau)} \|\mathbf{N}(\xi(\tau))\|_\infty \; \mathrm{d}\tau,
\end{align*}
where the last terms are continuous in $t \in \mathbb{R}_{\ge 0}$ and smaller than $r$ for $t=0$. This shows $T>0$ and we can further estimate the solution $\xi$ for $0\le t <T$. Continuity of $\mathbf{N}$ shows that the estimate $\|\mathbf{N}(\xi)\|_\infty \le \frac{w}{2M}\|\xi\|_\infty$ holds as long as $\|\xi\|_\infty \le r$, hence
\begin{align*}
	\mathrm{e}^{wt} \|\xi(\cdot, t)\|_\infty & \le M \|\xi^0 \|_\infty + \int_0^t M\mathrm{e}^{w\tau} \|\mathbf{N}(\xi(\cdot, \tau))\|_\infty \; \mathrm{d}\tau \\
	& \le M\|\xi^0\|_\infty + \frac{w}{2} \int_0^t \mathrm{e}^{w\tau}\|\xi(\cdot, \tau)\|_\infty \; \mathrm{d}\tau.
\end{align*}
Since $\tau \mapsto \xi(\cdot, \tau)$ is Bochner integrable on $L^\infty(\Omega)^{m+k}$, the function $\tau \mapsto \| \xi(\cdot, \tau) \|_\infty$ is a bounded, Lebesgue integrable function on $[0,T)$. By Gronwall's lemma, it follows
\begin{align*}
	\|\xi(\cdot, t) \|_\infty \le M \|\xi^0\|_\infty\mathrm{e}^{-\frac{w}{2}t} \le r \mathrm{e}^{-\frac{w}{2}t} \le r \qquad \forall \; 0\le t < T.
\end{align*}
If $T<T_{\max}$, we would have $\|\xi(\cdot, t)\|_\infty \to r$ for $t\to T$ by continuity of the solution $\xi \in C((0,T_{\max});  L^\infty(\Omega)^{m+k})$ and definition of $T$. However, this contradicts the estimate $\|\xi(\cdot, t)\|_\infty \le r\mathrm{e}^{-\frac{w}{2}t} \to r\mathrm{e}^{-\frac{w}{2}T} < r$ for $t\to T$. This shows $T \ge T_{\max}$ and no blow-up is possible as $t$ reaches $T_{\max}$ from below. By Proposition \ref{Rothesol} (iii), it follows $T = T_{\max} = \infty$. Hence, the exponential decay rate in estimate \eqref{expstab} can be chosen to be $\beta=w/2>0$. 
\end{proof}

{Using the notion of Lyapunov stability on certain subspaces of $L^\infty(\Omega)^{m+k}$, we gain the following result.}

\begin{corollary} \label{stabC}
Stability with respect to $L^\infty(\Omega)^{m+k}$ in the sense of Definition \ref{stability} implies Lyapunov stability in the subspace $L^\infty(\Omega)^m \times C(\overline{\Omega})^k$. Furthermore, concerning a continuous steady state in $C(\overline{\Omega})^{m+k}$ under Assumption \ref{ass:Nc}, Lyapunov stability in $L^\infty(\Omega)^m \times C(\overline{\Omega})^k$ implies Lyapunov stability in $C(\overline{\Omega})^{m+k}$.
\end{corollary}

\begin{proof}
This is due to regularity of the solution discussed in Remark \ref{rem:regmildsol} and the fact that the same norm is used for each of these spaces.
\end{proof}

\subsection{Instability of steady states} \label{sec:instab}

In the nonlinear case, there are mainly two ways showing instability via linearization.\\ Superlinear decay of the nonlinear reaction term in combination with a non-empty intersection of the spectrum with the complex right half-plane is sufficient for instability, see \cite[Theorem 1]{Shatah} or \cite[Ch. VII, Theorem 2.3]{Daleckii} for bounded operators. Unfortunately, both the nonlinear term and the semigroup generated by the linearized operator need to be continuous on the same function space $X$. This seems to be the reason why \cite[Theorem 2.1]{MKS17} uses the space $X=W^{1,p}(\Omega)^{m+1}$, which requires a regular steady state $(\overline{u}, \overline{v}) \in X$ for the nonlinear term to be well-defined. In order to reduce regularity of the steady state, the authors of \cite{MKS17} choose a spectral gap condition described next.\\
A growth estimate of the nonlinear reaction term with respect to two different norms and a spectral gap next to the imaginary axis may lead to instability, see \cite[Theorem 2.1]{Friedlander} or \cite[Ch. VII, Theorem 2.2]{Daleckii} for bounded operators. The former method is used to provide the instability result \cite[Theorem 2.11]{MKS17} which can be generalized to systems and discontinuous, bounded steady states in view of Proposition \ref{RDODEspec}.\\ 

In this section, we aim for a generalization of the results from \cite{MKS17} to bounded, discontinuous and continuous, steady states of the reaction-diffusion-ODE system \eqref{fullsys} without requiring a spectral gap. 
We use the semigroup approach of \cite{Shatah} combined with the notion of a mild solution from \cite{Rothe} to show instability. First, we aim to verify semigroup estimates from \cite[Lemma 2, 3]{Shatah} {on $L^\infty(\Omega)^{m+k}$ in the case of a discontinuous stationary solution. Besides continuity of the solution for positive times, such estimates are essential for showing instability, see Theorem \ref{instabtheorem}.} \\

To obtain instability, let the spectrum of the linear operator $\mathbf{L}$ on $ L^\infty(\Omega)^{m} \times L^p(\Omega)^{k}$ defined in Lemma \ref{Lclosed} meet the right complex half-plane $\{\lambda \in \mathbb{C} \mid \mathrm{Re} \, \lambda >0\}$, i.e., we assume $s(\mathbf{L})>0$. In view of Proposition \ref{specrestricted}, $s(\mathbf{L}_c)=s(\mathbf{L})>0$ and the spectral radius of the restricted semigroup $(\mathbf{T}_c(t))_{t \in \mathbb{R}_{\ge 0}}$ satisfies $r(\mathbf{B})=\mathrm{e}^{w_0^c} >1$ for $\mathbf{B}:= \mathbf{T}_c(1)$ \cite[Ch. IV, Proposition 2.2]{Engel}. We choose some $\mu \in \sigma(\mathbf{B})$ on the outer boundary of the spectrum $\sigma(\mathbf{B})$ and obtain the following result similar to \cite[Lemma 2]{Shatah}. 

\begin{lemma} \label{Shatah2}
Let $\mu=\mathrm{e}^{\lambda} \in \sigma(\mathbf{T}_c(1))$ for $\lambda \in \mathbb{C}$ satisfy $|\mu|=r(\mathbf{T}_c(1))=\mathrm{e}^{s(\mathbf{L})}$ and $\mathrm{Re} \, \lambda = s(\mathbf{L}) >0$. Then for each $\gamma \in \mathbb{R}_{>0}$ and $m\in\mathbb{N}$ there exists a function $\mathbf{v} \in L^\infty(\Omega)^m \times C(\overline{\Omega})^k$ such that 
\begin{align*}
	\|\mathbf{T}_c(m) \mathbf{v} -\mathrm{e}^{\lambda m}  \mathbf{v} \|_\infty < \gamma \|\mathbf{v}\|_\infty  
\end{align*}
holds. Moreover, the same function $\mathbf{v}$ fulfills
\begin{align*}
	\|\mathbf{T}_c(t)\mathbf{v}\|_\infty  \le 2 K \mathrm{e}^{\mathrm{Re}  \lambda \,  t} \|\mathbf{v}\|_\infty \qquad \forall \; 0 \le t \le m, 
\end{align*}
where $K := \sup \{\|\mathbf{T}_c(s)\|_\infty \mid 0\le s \le1\}< \infty$.
\end{lemma}

\begin{proof}
Apply \cite[Lemma 2]{Shatah} to the bounded operator $\mathbf{B}=\mathbf{T}_c(1)$ on $L^\infty(\Omega)^m \times C(\overline{\Omega})^k$, see Lemma \ref{restrsemigroup2}.
\end{proof}

While Lemma \ref{Shatah2} provides initial data which lead us to instability, the following inequalities are used for an estimation of the implicit integral equation of the mild solution in $L^\infty(\Omega)^{m+k}$.

\begin{lemma} \label{Shatah3}
Let $\lambda \in \mathbb{C}$ with $\mathrm{Re} \, \lambda = s(\mathbf{L})>0$
. Then for each $\varepsilon>0$ there is a constant $C_{\varepsilon}>0$ such that 
\begin{align*}
	\mathrm{e}^{\mathrm{Re}  \lambda \, t} \le \|\mathbf{T}_\infty(t)\|_\infty  \le C_{\varepsilon} \mathrm{e}^{(\mathrm{Re} \lambda + \varepsilon) t} \qquad \; \forall \; t \in \mathbb{R}_{\ge 0}. 
\end{align*}
\end{lemma}

\begin{proof} 
Proposition \ref{SBeGB} shows $w_0^\infty=s(\mathbf{L})$ for the growth bound $w_0^\infty$ of $(\mathbf{T}_\infty(t))_{t \in \mathbb{R}_{\ge 0}}$. Applying \cite[Ch. IV, Lemma 2.3]{Engel} to the exponentially bounded semigroup $(\mathbf{T}_\infty(t))_{t \in \mathbb{R}_{\ge 0}}$ on $L^\infty(\Omega)^{m+k}$ yields $w_0^\infty = \log r(\mathbf{T}_\infty(1))$, similar to the proof of \cite[Ch. IV, Proposition 2.2]{Engel}. Hence, the definition of the spectral radius implies	
\[
\mathrm{e}^{\mathrm{Re} \lambda} = |\mathrm{e}^{\lambda}| = r(\mathbf{T}_\infty(1)) = \lim_{m\to\infty} \|\mathbf{T}_\infty(m)\|^{1/m}.
\]
The rest is analog to the proof of \cite[Lemma 3]{Shatah}.
\end{proof}

Proposition \ref{SBeGB} shows the relation $w_0^\infty = s(\mathbf{L})$ between the growth bound $w_0^\infty$ of the semigroup  $(\mathbf{T}(t))_{t \in \mathbb{R}_{\ge 0}}$ restricted to $L^\infty(\Omega)^{m+k}$ and the spectral bound $s(\mathbf{L})$ of the linear operator $\mathbf{L}$. As a consequence, a positive spectral bound $s(\mathbf{L})$ implies a positive growth bound $w_0^\infty$, and hence instability.

\begin{theorem} \label{instabtheorem}
Consider a steady state $(\overline{\mathbf{u}}, \overline{\mathbf{v}}) \in L^\infty(\Omega)^m \times C(\overline{\Omega})^k$ of system \eqref{fullsys}. Let Assumption \ref{ass:N} be satisfied. If the linear operator $\mathbf{L}$ from Lemma \ref{Lclosed} has a positive spectral bound $s(\mathbf{L})>0$, then the steady state is nonlinearly unstable in the sense of Definition \ref{stability} and in the sense of Lyapunov in $L^\infty(\Omega)^m \times C(\overline{\Omega})^k$.
\end{theorem}

\begin{proof}
The proof is basically analog to the proof of \cite[Theorem 1]{Shatah}, where a linearized equation of the form \eqref{LNequation} is considered. {In the following we depict only main adaptions of the proof.} We consider the restricted semigroup $(\mathbf{T}_c(t))_{t \in \mathbb{R}_{\ge 0}}$ in Lemma \ref{Shatah2} to provide bounded initial data $\mathbf{v} \in L^\infty(\Omega)^{m} \times C(\overline{\Omega})^{k}$ that lead us to instability. Assuming stability to the contrary, a global mild solution in the sense of Definition \ref{mildsol} satisfies 
\[
\xi \in C(\mathbb{R}_{\ge 0}; L^\infty(\Omega)^m \times L^p(\Omega)^k) \cap C(\mathbb{R}_{>0}; L^\infty(\Omega)^{m} \times C(\overline{\Omega})^{k}).
\]
However, the restricted semigroup $(\mathbf{T}_\infty(t))_{t\in \mathbb{R}_{\ge0}}$ is not strongly continuous on $L^\infty(\Omega)^{m+k}$. Nevertheless, from boundedness of the solution $\xi$, we infer that
\begin{align*}
	t \mapsto	 \xi(\cdot,t)-\mathbf{T}(t)\mathbf{v} = \int_{0}^{t} \mathbf{T}(t-\tau) \mathbf{N}(\xi(\cdot,\tau)) \; \mathrm{d}\tau \in L^\infty(\Omega)^{m+k}
\end{align*}
is continuous in $\mathbb{R}_{\ge 0}$ and equals $\mathbf{0}$ in $t=0$ \cite[Proposition 1.3.4]{Arendt}. Using $\sigma=s(\mathbf{L}) >0$, we define 
\begin{align*}
	T:= \sup\{t>0 \mid \|\xi(\cdot,\tau)-\mathbf{T}(\tau)\mathbf{v}\|_\infty < \frac{1}{2|\mu|}\delta \mathrm{e}^{\sigma \tau},  \|\xi(\cdot,\tau)\|_\infty < \frac{\rho}{2} \ \text{for all}\  0<\tau\le t\}.
\end{align*}
Whereas the first condition can be satisfied by continuity for all sufficiently small $t$, the second condition needs more arguments. Using Lemma \ref{Shatah2} and \ref{Shatah3}, we obtain
\begin{align*}
	\|\xi(\cdot,t)\|_\infty &\le \|\mathbf{T}(t)\mathbf{v}\|_\infty + \int_{0}^{t} \|\mathbf{T}(t-\tau) \mathbf{N}(\xi(\cdot,\tau))\|_\infty \; \mathrm{d}\tau \\
	& \le 2K\mathrm{e}^{\sigma t}\|\mathbf{v}\|_\infty + \int_{0}^{t} C_{\varepsilon}\mathrm{e}^{(\sigma + \varepsilon)(t-\tau)} \|\mathbf{N}(\xi(\cdot,\tau))\|_\infty \; \mathrm{d}\tau.
\end{align*}
The right-hand side is continuous in $t \in \mathbb{R}_{\ge 0}$ and equals $2K\|\mathbf{v}\|_\infty = 2K \delta$ in $t=0$. In order to ensure $T>0$ in above definition, we choose in case of $2K >1$
\[
0<\delta<\min\left\{\delta_0,k^{-1},\frac{\rho}{4K},1\right\}
\]
rather than $0<\delta<\min\{\delta_0,k^{-1},\frac{\rho}{2},1\}$ for initial conditions $\mathbf{v} \in L^\infty(\Omega)^{m} \times C(\overline{\Omega})^{k}$ with $\|\mathbf{v}\|_\infty = \delta$. Apart from these modifications, the remainder of the proof is identical to \cite{Shatah}. It is shown that there is some time point at which the solution leaves the presumed invariant region. This leads to a contradiction to stability in the sense of Definition \ref{stability}, and hence instability of the zero solution.\\
Concerning instability in the sense of Lyapunov in $L^\infty(\Omega)^m \times C(\overline{\Omega})^k$, we allow only perturbations in $L^\infty(\Omega)^m \times C(\overline{\Omega})^k$. Thus, the proof remains the same. Remark \ref{rem:regmildsol} yields continuity of the mild solution.
\end{proof}

The characterization of the spectrum of the linear operator $\mathbf{L}$ in Proposition \ref{RDODEspec} can be used to determine the sign of the spectral bound $s(\mathbf{L})$. Note that this spectrum is independent of the parameter $p$ by Corollary \ref{independence}. Recall that a spectral value $\lambda \in \sigma(\mathbf{A}_{\ast})$ of the ODE subsystem with $\mathrm{Re} \, \lambda >0$ already implies the instability condition $s(\mathbf{L})>0$.\\ 

Theorem \ref{instabtheorem} provides an instability result in the sense of Definition \ref{stability} and in the sense of Lyapunov in $L^\infty(\Omega)^m \times C(\overline{\Omega})^k$. This is a consequence of the choice of initial values from Lemma \ref{Shatah2}. In the case of a continuous stationary solution, one may also consider instability in the sense of Lyapunov in $C(\overline{\Omega})^{m+k}$. For this, however, we also need continuous initial values that lead us to instability. 

\begin{corollary} \label{instabC}
Consider a steady state $(\overline{\mathbf{u}}, \overline{\mathbf{v}}) \in C(\overline{\Omega})^{m+k}$ of system \eqref{fullsys} and let Assumption \ref{ass:Nc} hold. If the linear operator $\mathbf{L}$ from Lemma \ref{Lclosed} has a positive spectral bound $s(\mathbf{L})>0$, then the steady state is nonlinearly unstable in the sense of Lyapunov in $C(\overline{\Omega})^{m+k}$.
\end{corollary}

\begin{proof}
The modified Assumption \ref{ass:N} implies that the operator $\mathbf{L}^c$ is defined as in Lemma \ref{Lcontinuous} with spectral characterization $\sigma(\mathbf{L}^c)= \sigma(\mathbf{L})$ by Proposition \ref{RDODEspecC}. Moreover, the assumptions imply validity of the estimates of the nonlinear remainder from Lemma \ref{nonlinearity} on $C(\overline{\Omega})^{m+k}$. Since the spectral mapping theorem holds in case of the analytic semigroup generated by $\mathbf{L}^c$, \cite[Theorem 1]{Shatah} can be applied on $C(\overline{\Omega})^{m+k}$.
\end{proof}

\begin{remark} \label{rem:general}
Due to flexibility of semigroup theory, other differential operators than the Laplace operator may be considered, also endowed with different boundary conditions. The spectral characterization given in Proposition \ref{RDODEspec} and Corollary \ref{independence} is mainly due to the compact resolvent and the regularizing effect of the Laplacian on a bounded domain. Further calculations in Section \ref{sec:specanalysis} are based on analyticity of the semigroup $(\mathbf{T}(t))_{t \in \mathbb{R}_{\ge 0}}$ on $L^\infty(\Omega)^m \times L^p(\Omega)^k$ with a continuous embedding $(\mathcal{D}(\mathbf{L}), \| \cdot\|_{\mathbf{L}}) \hookrightarrow L^\infty(\Omega)^{m+k}$ for the generator $\mathbf{L}$. We refer to \cite{Lunardi, Nittka, Ouhabaz} for differential operators which satisfy similar properties. 
The notion of stability and the corresponding stability results are based on existence and uniqueness of a mild solution $t \mapsto \xi(\cdot, t)$ of problem \eqref{LNequation} which is norm-continuous in $L^\infty(\Omega)^{m+k}$ for $t>0$. 
Moreover, the stability result is applicable to classical reaction-diffusion systems $(m=0)$, for which the spectrum $\sigma(\mathbf{L})$ is purely discrete by \cite[Ch. IV, Corollary 1.19]{Engel} and the compact resolvent of $\mathbf{L}$, compare \cite{Henry, Smoller}.
\end{remark}

\section{Applications} \label{sec:appl}

In this section, we show applicability of the developed theory to the analysis of pattern formation on two basic examples of reaction-diffusion-ODE problems exhibiting unstable continuous and stable discontinuous stationary solutions. We start with some general properties of the linearized model. To this end, we focus on the spectrum of the linearized operator $\mathbf{L}$ defined in Lemma \ref{Lclosed}, i.e., $\mathbf{L}: \mathcal{D}(\mathbf{L}) \subset L^\infty(\Omega)^{m} \times L^p(\Omega)^{k} \to L^\infty(\Omega)^{m} \times L^p(\Omega)^{k}$ given by
\begin{align*}
(\mathbf{L} 
\xi) (x) & =  \begin{pmatrix}
	\mathbf{0} \\
	\mathbf{D}^v \Delta \xi_2(x)
\end{pmatrix}
+  \begin{pmatrix}
	\mathbf{A}_{\ast}(x) \xi_1(x) +  \mathbf{B}_{\ast}(x) \xi_2(x)\\
	\mathbf{C}_{\ast}(x) \xi_1(x) + \mathbf{D}_{\ast}(x) \xi_2(x)
\end{pmatrix} 
\end{align*}
where $\xi= (\xi_1, \xi_2) \in \mathcal{D}(\mathbf{L})= L^\infty(\Omega)^{m} \times \mathcal{D}(A_p)^{k}$ for $p>n^\ast=\max\{n/2,2\}$. Recall that the coefficients $\mathbf{A}_{\ast}, \mathbf{B}_{\ast}, \mathbf{C}_{\ast},\mathbf{D}_{\ast}$ are matrices with entries in $L^\infty(\Omega)$. If $s(\mathbf{A}_{\ast}) >0$, then it follows $s(\mathbf{L})>0$ by Proposition \ref{RDODEspec}. 
In general, $s(\mathbf{A}_{\ast}) < 0$ does not imply a negative spectral bound of $\mathbf{L}$, see 
Theorem \ref{bistable} below. In such a case, the next lemma is helpful for a reaction-diffusion-ODE system consisting of one ordinary and one partial differential equation. This result {generalizes \cite[Corollary 3.10]{Steffenthesis}, \cite[Proposition 4.3]{CMCKSirregular}} and is adapted from \cite[Section 4]{Weinberger}. 

\begin{lemma}\label{Cor3.10Steffen}
Assume that there exists a constant $c>0$ such that 
\begin{align}
	A_{\ast}(x) \le -c, \quad D_{\ast}(x) \le 0 , \quad A_{\ast}(x)D_{\ast}(x) - B_{\ast}(x) C_{\ast}(x) > 0 \label{stabassumption} 
\end{align}
holds for almost every $x\in \Omega$. Then, the linear operator $\mathbf{L}$ defined in Lemma \ref{Lclosed} satisfies $s(\mathbf{L}) < 0$.
\end{lemma}

\begin{proof}
By the characterization of $\sigma(A_{\ast})$ in \cite[Appendix B.2]{KMCMnonlinear}, it follows $s(A_{\ast}) < 0$.
We have $\sigma(\mathbf{L}) = \sigma(A_{\ast}) \cup \Sigma$ by Proposition \ref{RDODEspec}. Moreover, the resolvent $\rho(A_\ast)$ is connected in $\mathbb{C}$ since $A_\ast(x)\subset\mathbb{R}$. Thus, $\Sigma \subset \sigma_p(\mathbf{L})$ is discrete by Proposition \ref{RDODEspec}.
Hence, it is left to prove $\sup_{\lambda\in\Sigma} \mathrm{Re} \, \lambda < 0$.
Although $\Sigma$ is discrete, we have to find a constant $c_0>0$ such that $\mathrm{Re} \,  \lambda \le - c_0<0$ for all $\lambda \in \Sigma = \sigma_p(\mathbf{L})\cap \rho(A_{\ast})$. For each $\lambda \in \sigma_p(\mathbf{L}) \cap \rho(A_{\ast})$ we find an eigenfunction $\xi= (\xi_1, \xi_2) \in\mathcal{D}(\mathbf{L})\setminus \{(0,0)\}$ such that
\begin{align*}
	\mathbf{0} = (\mathbf{L}- \lambda I) \xi =   \begin{pmatrix}0\\D^v \Delta \xi_2\end{pmatrix}
	+ \begin{pmatrix}(A_{\ast}- \lambda I) \xi_1 +  B_{\ast} \xi_2\\C_{\ast}\xi_1 + (D_{\ast}- \lambda I) \xi_2\end{pmatrix}.
\end{align*}
Since $\lambda \in \rho(A_{\ast})$, we can solve the first equation to obtain $\xi_1 = (\lambda I - A_{\ast})^{-1} B_{\ast} \xi_2$.
This relation shows $\xi_2 \not\equiv 0$, since $\xi_2 \equiv 0$ would imply $\xi_1 \equiv 0$. However, this contradicts $\lambda \in \sigma_p(\mathbf{L})$ with eigenfunction $\xi \not \equiv \mathbf{0}$.\\
Similar to Corollary \ref{independence}, a substitution of $\xi_1$ in the second equation leads us to the following elliptic boundary value problem
\begin{align*}
	- D^v \Delta \xi_2 = 	\left( \frac{B_{\ast}C_{\ast}}{\lambda- A_{\ast}} +  D_{\ast} - \lambda \right) \xi_2 
\end{align*}
for $\xi_2 \in \mathcal{D}(A_p) \setminus \{0\}, p > n^\ast$. According to Lemma \ref{domchar}, this problem can be seen as a weak formulation for $\xi_2 \in W^{1,2}(\Omega)$. Multiplying this equation with the complex conjugate of $\xi_2$ and integrating yields
\begin{align}\label{Cor3.10equation}
	D^v \int_\Omega |\nabla \xi_2|^2 \; \mathrm{d}x = \int_\Omega  \left(\frac{B_{\ast}C_{\ast}}{\lambda- A_{\ast}} +  D_{\ast} - \lambda \right) |\xi_2|^2 \; \mathrm{d}x.
\end{align}
We write $\lambda = \lambda_1 + i \lambda_2$ for $\lambda_1, \lambda_2 \in \mathbb{R}$ and consider the imaginary and real part of equation \eqref{Cor3.10equation}, i.e.,
\begin{align}
	- \lambda_2 \int_\Omega \left( \frac{B_{\ast}C_{\ast}}{|\lambda - A_{\ast}|^2} +1 \right)|\xi_2|^2  \; \mathrm{d}x = 0, \label{Impart}\\
	\int_\Omega A(x,\lambda) |\xi_2|^2  \; \mathrm{d}x= D^v \int_\Omega |\nabla \xi_2|^2 \; \mathrm{d}x \ge 0 \label{Repart}
\end{align}
with
\begin{align*}
	A(x, \lambda) =  \frac{B_{\ast}(x)C_{\ast}(x)(\lambda_1- A_{\ast}(x))}{| \lambda-A_{\ast}(x)|^2} +D_{\ast}(x) - \lambda_1.
\end{align*}
Equation \eqref{Impart} implies that $\lambda_2$ is uniformly bounded. In fact, $\lambda_2 \not=0$ and 
\begin{align*}
	\frac{B_{\ast}C_{\ast}}{|\lambda - A_{\ast}|^2} +1 > 0 \qquad \text{a.e. in } \; \Omega
\end{align*}
leads to a contradiction since $\xi_2 \not\equiv 0$. So, in case $\lambda_2\not=0$, we may assume that the last inequality does not hold almost everywhere in $\Omega$, i.e., there is a measurable set $\Omega_1$ with $|\Omega_1|>0$ such that 
\[
\frac{B_{\ast}C_{\ast}}{|\lambda - A_{\ast}|^2} +1 \le  0 \qquad \text{a.e. in } \; \Omega_1.
\]	
Hence, boundedness of the coefficients $B_{\ast}, C_{\ast}$ implies uniform boundedness of the imaginary part $\lambda_2$ as $\lambda_2^2 \le  \|B_{\ast} C_{\ast}\|_\infty$. Since $\Sigma$ is a discrete set and the imaginary parts of $\lambda \in \Sigma$ are uniformly bounded, it remains to show $\mathrm{Re} \,  \lambda <0$ for all $\lambda \in \Sigma$.\\
For this reason, let us assume to the contrary that there exists $\lambda \in\Sigma$ with $\lambda_1 = \mathrm{Re} \, \lambda \ge 0$. We will show that such a spectral value cannot exist under above assumption \eqref{stabassumption}. Let us consider the real part of equation \eqref{Cor3.10equation}. We aim to show $A(\cdot, \lambda) <0$ a.e. in $\Omega$ to obtain a contradiction to relation \eqref{Repart}, since $\xi_2 \not\equiv 0$.
With $| \lambda-A_{\ast}|^2= (\lambda_1-A_{\ast})^2  + \lambda_2^2$, we obtain
\begin{align*}
	A(\cdot, \lambda) &= \frac{B_{\ast}C_{\ast}(\lambda_1- A_{\ast})}{| \lambda-A_{\ast}|^2} +D_{\ast} - \lambda_1 \\
	& = \frac{(A_{\ast}-\lambda_1) [ (A_{\ast}-\lambda_1)(D_{\ast}-\lambda_1) - B_{\ast} C_{\ast} ] + (D_{\ast}-\lambda_1) \lambda_2^2}{| \lambda-A_{\ast}|^2}\\
	& = \frac{(A_{\ast}-\lambda_1) [ \lambda_1^2 - (A_{\ast}+D_{\ast}) \lambda_1 + A_{\ast}D_{\ast} - B_{\ast} C_{\ast} ] + (D_{\ast}-\lambda_1) \lambda_2^2}{| \lambda-A_{\ast}|^2}.
\end{align*}
Recall that $A_{\ast}- \lambda_1 \le -c <0, D_{\ast}- \lambda_1 \le 0$ for $\lambda_1 \ge 0$. Moreover, assumption \eqref{stabassumption} guarantees that 
\[
\lambda_1^2 - (A_{\ast}+D_{\ast}) \lambda_1 + A_{\ast}D_{\ast} - B_{\ast} C_{\ast} \ge A_{\ast}D_{\ast} - B_{\ast} C_{\ast} >0.
\]
All in all, $A(x,\lambda) < 0$ for almost every $x\in \Omega$ and, using monotony of the Lebesgue integral, relation \eqref{Repart} yields a contradiction since $\xi_2\not\equiv 0$. Hence, there exists no $\lambda \in\Sigma$ with $\mathrm{Re} \, \lambda \ge 0$. 
\end{proof}

\subsection{Hysteresis model} 
\label{sec:KMCT}

We start examples with a basic reaction-diffusion-ODE model exhibiting discontinuous patterns. Introduced in \cite{KMCT20} as a prototype of such phenomena, it is called the {\it hysteresis model}, because the jump-discontinuous solutions emerge due to co-existence of multiple constant stationary solutions for the ODE subsystem. In particular, it requires a hysteresis effect in the nullclines of the model nonlinearities. A minimal model with such properties takes a two-component form of \eqref{fullsys} with the following choice of functions $f$ and $g$,
\begin{align}
f(u,v) = v - p(u), \qquad g(u,v) = \alpha u - \beta v  \label{KMCTnonlinearities}
\end{align}
for $\alpha,\beta> 0$, and a polynomial $p$ of degree three with only one real root at $u=0$. Furthermore, we assume that there are three distinct intersection points of $f=0$ and $g=0$ with non-negative coordinates
\begin{align*}
S_0 = (0,0), \quad S_1 = (u_1,v_1) \quad \text{and} \quad S_2 = (u_2,v_2). 
\end{align*}
Additionally, we prescribe the slope $p'(0)>\alpha/\beta$ of the polynomial $p$ to obtain an S-hysteresis effect which can be observed in the nullclines of the model nonlinearities regarding the $(v,u)$-plane. We choose $\Omega=(0,1) \subset \mathbb{R}$. A generalization to heterogeneous environments can be found in \cite{Takagi21}. In the remainder of this subsection, we show application of our nonlinear stability and instability results, which provides insights into the model dynamics, significantly simplifies the proofs originally provided in \cite{KMCT20} and allows extending them from the linear to nonlinear analysis.

Let $(\overline{u},\overline{v})$ be a bounded stationary solution. Then, the linear operator $\mathbf{L}$ defined in \eqref{Linear2} is given by
\begin{align*}
(\mathbf{L}\xi)(x) = \begin{pmatrix} 0 \\ D^v \Delta \xi_2(x)\end{pmatrix} + \begin{pmatrix}	-p'(\overline{u}(x)) & 1\\ \alpha & -\beta\end{pmatrix} \xi.
\end{align*}
We consider two cases.

\begin{description}
\item[\bf Bistable case] The polynomial $p$ is strictly increasing.

\item[\bf Hysteresis case] The model exhibits S-hysteresis, i.e., the polynomial $p$ is non-monotone with $\lim_{u \to \infty} p(u) = \infty$. Let $H=(u_H,v_H)$ and $T=(u_T,v_T)$ denote the local maximum and minimum of $p$, respectively. We assume that the coordinates of $H$ and $T$ are positive with $u_T>u_H$. 
\end{description}

Using \cite[Lemma 2.1]{KMCT20} and Lemma \ref{Cor3.10Steffen}, one verifies in both cases that the constant steady states $S_0, S_2$ are nonlinearly stable in the sense of Definition \ref{stability} and $S_1$ is nonlinearly unstable for system \eqref{fullsys} endowed with nonlinearities \eqref{KMCTnonlinearities}. Hence, we focus on non-homogeneous stationary solutions. Recall that in the bistable case, the equation $f(\overline{u},\overline{v})=0$ can be uniquely solved with respect to $\overline{u}$, i.e., $\overline{u}=h(\overline{v})$ for some smooth function $h$. By elliptic regularity from Lemma \ref{domchar}, the steady state $(\overline{u},\overline{v})$ is at least continuous in $\overline{\Omega}$.

\begin{theorem}[Instability of all continuous patterns]\label{bistable}
In the bistable case, every non-homogeneous stationary solution $(\overline{u}, \overline{v}) \in C(\overline{\Omega})^2$ of system \eqref{fullsys} with model nonlinearities \eqref{KMCTnonlinearities} is nonlinearly unstable in the sense of Lyapunov in $L^\infty(\Omega)^m \times C(\overline{\Omega})^k$, $C(\overline{\Omega})^{m+k}$, and in the sense of Definition \ref{stability}.
\end{theorem}

\begin{proof}
Since $p$ is strictly increasing, $K = \min_{x\in \overline{\Omega}} p'(\overline{u}(x)) > 0$ by continuity of the component $\overline{u}$. Hence, we obtain $A_{\ast}(x) \le -c < 0$ for all $x\in \overline{\Omega}$ for some $c>0$. In order to apply Corollary \ref{instabC}, one has to prove existence of $\lambda\in\Sigma = \sigma(\mathbf{L})\cap \rho(A_{\ast})\subset \sigma_p(\mathbf{L})$ with $\mathrm{Re} \, \lambda > 0$. However, this is shown in the proof of \cite[Theorem 2.2]{KMCT20} for $\Omega=(0,1) \subset \mathbb{R}$. 
\end{proof}

We note that \cite[Theorem 2.8]{CMCKSregular} applies in this context to the convex domain $\Omega=(0,1)$. Hence, the single reaction-diffusion equation destabilizes the continuous steady state. However, in the hysteresis case, there exist various stable non-homogeneous patterns which are discontinuous in space.\\

In the hysteresis case, it is shown in \cite[Section 3]{KMCT20} for the one-dimensional case that all stationary, spatially heterogeneous solutions have a jump-discontinuity in the ODE component $\overline{u}$. For existence of such solutions, we refer to \cite[Section 3]{KMCT20}. The following result originates from \cite[Theorem 3.2]{KMCT20}. Denote by $h_H,h_0$ and $h_T$ the corresponding branches of $f(\overline{u},\overline{v})=0$ solved with respect to $\overline{u}$.
For $v_j\in (v_T,\min\{v_H,v_2\})$ there exists a unique monotone increasing solution $(\overline{u},\overline{v})$ with jump at $v_j$. Hence, $\overline{u}$ is of the form
\begin{align*}
\overline{u}(x) = \begin{cases}
	h_H(\overline{v}(x)) & \text{if} \quad  \overline{v}(x)<v_j,\\
	h_T(\overline{v}(x)) & \text{if} \quad \overline{v}(x)>v_j.
\end{cases}
\end{align*}
For such a steady state one can prove $s(A_{\ast}) \le 0$. However, $s(A_{\ast}) \le 0$ is not sufficient for stability of the steady state $(\overline{u},\overline{v})$. The next theorem states a sufficient condition for stability, compare its analog \cite[Theorem 3.7]{KMCT20}.

\begin{theorem}[Nonlinear stability of a discontinuous pattern]\label{hysteresis}
Consider a stationary solution $(\overline{u},\overline{v}) \in L^\infty(\Omega) \times C(\overline{\Omega})$ of system \eqref{fullsys} with model nonlinearities \eqref{KMCTnonlinearities} in the hysteresis case. Assume that $(\overline{u},\overline{v})$ satisfies the stability condition $K := \operatornamewithlimits{essinf}_{x\in \Omega} p'(\overline{u}(x)) > \frac{\alpha}{\beta}$. Then, $(\overline{u},\overline{v})$ is nonlinearly stable in the sense of Definition \ref{stability}.
\end{theorem}

\begin{proof}
By assumption, there exists $c>0$ such that there holds
\begin{align*}
	A_{\ast} &= -p'(\overline{u}) \le -c, \quad 	D_{\ast}= -\beta \le 0, \quad A_{\ast}D_{\ast}-B_{\ast}C_{\ast}  = p'(\overline{u})\beta - \alpha >0 \quad \text{a.e. in} \quad \Omega.	 
\end{align*}
Hence, we can apply Lemma \ref{Cor3.10Steffen} which yields $s(\mathbf{L}) < 0$.
Theorem \ref{stab} implies that $(\overline{u},\overline{v})$ is stable.
\end{proof}

\begin{remark} \label{rem:Lpstability}
For the hysteresis model determined by \eqref{KMCTnonlinearities}, the nonlinear remainder $\mathbf{N}$ is also locally Lipschitz continuous as an operator on $L^\infty(\Omega) \times L^p(\Omega)$ since it is given by $\mathbf{N}(\xi)= (N_1(\xi_1), 0)$. The operator $\mathbf{N}$ even satisfies the decay estimate \eqref{estimateN} in this space. Since $\mathbf{L}$ generates an analytic semigroup for $p \ge 2, p >n/2$, we may apply the results \cite[Theorem 1]{Shatah} and \cite[Proposition 4.17]{Webb} for Lyapunov instability and stability of a stationary solution in $L^\infty(\Omega) \times L^p(\Omega)$.\\
However, if we look at small perturbations of a stationary solution in the space $L^p(\Omega) \times L^p(\Omega)$, the situation is different. Although linear stability analysis is the same as for $\mathbf{L}$ on the space $L^\infty(\Omega) \times  L^p(\Omega)$, the nonlinearities are not Fr\'{e}chet differentiable, even not well-defined as an operator on $L^p(\Omega) \times L^p(\Omega)$. Using \cite[Theorem 2.6]{CMCKSirregular}, one can derive that each small neighborhood (with respect to the $L^p$ norm) of the constant, linearly stable steady state $(0,0)$ contains infinitely many discontinuous stationary solutions in the hysteresis case. Hence, the constant steady state cannot be locally asymptotically stable in $L^p(\Omega) \times L^p(\Omega)$. It remains open whether the  steady state $(0,0)$ is nonlinearly stable in $L^p(\Omega) \times L^p(\Omega)$. 
\end{remark}

\subsection{DDI-hysteresis model}
\label{sec:Steffen}

In this section, we consider the {\it DDI-hysteresis model} that, additionally to the hysteresis-based mechanism for patterns with jump-discontinuities, exhibits diffusion-driven instability (Turing instability). As the latter is a diffusion-dependent bifurcation leading to the loss of stability of a spatially constant steady state, the system describes  {\it de novo} formation of the discontinuous patterns and can be considered as a prototype of the model of Turing mechanism leading to {\it far-from-equilibrium} patterns due to instability of all regular Turing patterns.
The model is given by system \eqref{fullsys} consisting of one ordinary and one partial differential equation with nonlinearities
\begin{align}\label{RecepBased}
f(u,v) = -u-uv + m_1 \frac{u^2}{1+k u^2} \quad \text{and} \quad g(u,v) = -\mu v - uv + m_2 \frac{u^2}{1+k u^2}
\end{align}
for some constants $k,\mu, m_1,m_2 >0$ with $m_1 > 2\sqrt{k}$. The {\it DDI-hysteresis model} originates from \cite{HMCT}, for details see also \cite[Section 3.6.1]{Steffenthesis}. In the remainder of this subsection, we explore the underlying pattern formation phenomena applying the developed theory. In particular, we show nonlinear stability of the jump-discontinuous solutions, which originally was shown only for the linearized model in a restrictive setting of $(\varepsilon, A)$ stability.

For a bounded steady state $(\overline{u},\overline{v})$, the linear operator defined in Lemma \ref{Lclosed} is given by
\begin{align*}
\mathbf{L} \begin{pmatrix}\xi_1 \\ \xi_2\end{pmatrix} = \begin{pmatrix}0 \\ D^v \Delta\xi_2\end{pmatrix} + J(\overline{u},\overline{v}) \begin{pmatrix}\xi_1 \\ \xi_2\end{pmatrix},
\end{align*}
where the Jacobian reads
\begin{align*}
J(\overline{u},\overline{v}) = \begin{pmatrix}-(1+\overline{v}) + m_1 \frac{2\overline{u}}{(1+k\overline{u}^2)^2} & -\overline{u}\\ -\overline{v} +m_2\frac{2\overline{u}}{(1+k\overline{u}^2)^2} & -(\mu+\overline{v})\end{pmatrix}.
\end{align*}
Following \cite[Section 3.2]{HMCT}, we compute non-negative steady states $(\overline{u},\overline{v})$. Since a stationary solution $(\overline{u},\overline{v})$ satisfies in particular $f(\overline{u},\overline{v}) \equiv 0$, it holds $\overline{u}(x) = 0$ or 
\begin{align}
-(1+\overline{v}(x))(1+k\overline{u}(x)^2)+m_1 \overline{u}(x) = 0 \label{f=0}
\end{align}
for $x\in \Omega$. Hence, there are three branches of solutions for $\overline{u}$, namely $u_0(v) := 0$ defined for all $v\ge0$, and
\begin{align*}
u_{\pm}(v) :=  \frac{1}{2k(1+v)} \left(m_1 \pm  \sqrt{m_1^2 - 4k(1+v)^2}\right)
\end{align*}
defined for $0 \le v \le v_r := \frac{m_1}{2\sqrt{k}} - 1$. We have $v_r>0$ since $m_1 > 2\sqrt{k}$.
It holds $u_-(v) u_+(v) = \frac{1}{k}$,
\begin{align*}
u_+(v) \ge \frac{m_1}{2k(1+v)} \ge \frac{m_1}{2k(1+v_r)} = \frac{1}{\sqrt{k}},
\end{align*}
and $u_+(v) >\frac{1}{\sqrt{k}}$ for $v<v_r$. This yields $0 < u_-(v) \le \frac{1}{\sqrt{k}}$ for $0\le v\le v_r$ and $u_-(v) < \frac{1}{\sqrt{k}}$ for $0\le v < v_r$. We obtain the following instability result, compare \cite[Section 3.3]{HMCT}.

\begin{theorem}[Instability of patterns]
Let $(\overline{u},\overline{v})\in L^\infty(\Omega) \times C(\overline{\Omega})$, $\overline{u}, \overline{v}\ge0$, be a stationary solution of problem \eqref{fullsys} with the nonlinearities $f$ and $g$ given by equation \eqref{RecepBased} and $m_1>2\sqrt{k}$.
If there exists a measurable subset $M\subset \Omega$, $|M|> 0$, such that for almost every $x\in M$ the stationary solution $(\overline{u},\overline{v})$ is of class $u_-$ and $\overline{v}(x) < v_r$ resp. $\overline{u}(x) < \frac{1}{\sqrt{k}}$, then the steady state $(\overline{u},\overline{v})$ is nonlinearly unstable in the sense of Definition \ref{stability} and in the sense of Lyapunov in $L^\infty(\Omega)^m \times C(\overline{\Omega})^k$.
\end{theorem}

\begin{proof}
To investigate stability of the steady state $(\overline{u},\overline{v})$, we first analyze the spectrum of the multiplication operator $A_\ast \in \mathcal{L}(L^\infty(\Omega))$ induced by
\begin{align*}
	A_{\ast}(x) = -(1+\overline{v}(x)) + m_1 \frac{2\overline{u}(x)}{(1+k\overline{u}(x)^2)^2}, \qquad x\in\Omega.
\end{align*}
By \cite[Appendix B.2]{KMCMnonlinear}, there exists a null set $N\subset \Omega$ such that 
$
\sigma(A_{\ast}) 
= \overline{\{A_{\ast}(x) \mid x\in \Omega\setminus N\}}.
$
For $x\in \Omega$, the steady state $(\overline{u}, \overline{v})$ is of the form $\overline{u}(x) = u_i(\overline{v}(x))$ for $i \in \{ 0, -, +\}$. If $\overline{u}(x) = u_0(\overline{v}(x)) = 0$, then $A_{\ast}(x) = -(1+\overline{v}(x)) \le -1 < 0$ and $s(A_\ast)<0$. For $\overline{u}(x) \ne 0$, we may use relation \eqref{f=0}.
In the case of $\overline{u}(x) = u_-(\overline{v}(x))$, this yields
\begin{align*}
	A_{\ast}(x) &= (1+\overline{v}(x)) \left( -1 + \frac{2}{1+k \overline{u}(x)^2} \right) \ge -1 + \frac{2}{1+k(\frac{1}{\sqrt{k}})^2} \ge 0
\end{align*}
and $A_{\ast}(x) > 0$ if $\overline{v}(x) \ne v_r$, i.e., $\overline{u}(x) \ne \frac{1}{\sqrt{k}}$.
In the remaining case $\overline{u}(x) = u_+(\overline{v}(x))$, we obtain $A_{\ast}(x) \le 0$ and $A_{\ast}(x) < 0$ if $\overline{v}(x) \ne v_r$, i.e., $\overline{u}(x) \ne \frac{1}{\sqrt{k}}$. In summary, it follows $s(A_{\ast}) > 0$ from above calculations and, by Proposition \ref{RDODEspec}, $s(\mathbf{L}) > 0$. Finally, apply Theorem \ref{instabtheorem}. 
\end{proof}

To get a statement about stability, we further have to study the remainder of the spectrum, i.e., $\Sigma = \sigma(\mathbf{L}) \cap \rho(A_{\ast}) \subset \sigma_p(\mathbf{L})$, which is discrete since $\rho(A_{\ast})$ connected. We aim to apply Lemma \ref{Cor3.10Steffen} in this case. The following theorem can be compared to the stability results obtained in \cite[Theorem 3.21]{Steffenthesis}.

\begin{theorem}[Nonlinear stability of patterns]
Let $(\overline{u},\overline{v}) \in L^\infty(\Omega) \times C(\overline{\Omega})$, $\overline{u}, \overline{v}\ge0$, be a stationary solution of system \eqref{fullsys} with functions $f$ and $g$ given by \eqref{RecepBased}. We assume that $\overline{u}$ is only of class $u_0$ and $u_+$ and, moreover, $u_s:= \mathrm{essinf}_{x \in \Omega, \; \overline{u}(x) \text{ of class }u_+} \overline{u}(x) > \frac{1}{\sqrt{k}}$ for $2 \sqrt{k} < m_1<m_2$. 
Then, the stationary solution $(\overline{u},\overline{v})$ is nonlinearly stable in the sense of Definition \ref{stability}.
\end{theorem}
\begin{proof}
In order to apply Theorem \ref{stab}, we prove $s(\mathbf{L}) < 0$ with the help of Lemma \ref{Cor3.10Steffen}. For $x\in \Omega$ with $\overline{u}(x) = u_0(x)$, we already know $A_{\ast}(x) \le -1 < 0$. Furthermore, it holds $D_{\ast}(x) \le 0$.
For the determinant of $J(\overline{u}, \overline{v})$, we obtain
\begin{align*}
	A_{\ast}(x)D_{\ast}(x) - B_{\ast}(x)C_{\ast}(x) &= (1+\overline{v}(x))(\mu+\overline{v}(x)) = \mu >0,
\end{align*}
since $\mu, D^v>0$ imply $\overline{v}=0$ by uniqueness of solutions of the stationary elliptic problem.\\
Next, let us consider $x\in \Omega$ with $\overline{u}(x) = u_+(x)$. 
Since
$
\overline{u}(x) = u_+(\overline{v}(x)) \ge u_s > 1/\sqrt{k}
$
holds for a.e. $x$, we estimate
\begin{align}\label{ADnegative}
	A_{\ast}(x) &\le -1 + \frac{2}{1+k u_s^2} =: -\tilde{c} < 0 \qquad\text{and} \qquad
	D_{\ast}(x) = -(\mu +\overline{v}(x)) \le 0.
\end{align}
It is left to investigate the determinant of $J(\overline{u},\overline{v})(x)$.
Using chain rule, we derive
\begin{align*}
	\det J(u_+(v),v) &= \partial_u f(u_+(v),v) \partial_v g(u_+(v),v) - \partial_v f(u_+(v),v) \partial_u g(u_+(v),v)\\
	&= \partial_u f(u_+(v),v) \left( \frac{\mathrm{d}}{\mathrm{d}v} g(u_+(v),v) - \frac{\mathrm{d}}{\mathrm{d}v} u_+(v) \partial_u g(u_+(v),v)\right) \\
	& \quad - \partial_v f(u_+(v),v) \partial_u g(u_+(v),v)\\
	&= -\partial_u g(u_+(v),v) \left( \partial_u f(u_+(v),v) \frac{\mathrm{d}}{\mathrm{d}v} u_+(v) + \partial_v f(u_+(v),v) \right) \\
	& \quad +\partial_u f(u_+(v),v) \frac{\mathrm{d}}{\mathrm{d}v} g(u_+(v),v)\\
	&= -\partial_u g(u_+(v),v) \frac{\mathrm{d}}{\mathrm{d}v} f(u_+(v),v) + \partial_u f(u_+(v),v) \frac{\mathrm{d}}{\mathrm{d}v} g(u_+(v),v)\\
	&= \partial_u f(u_+(v),v) \frac{\mathrm{d}}{\mathrm{d}v} g(u_+(v),v),
\end{align*}
where we used $f(u_+(v),v) = 0$ in the last step. First, we obtain
\begin{align*}
	\frac{\mathrm{d}}{\mathrm{d}v} u_+(v) &= \frac{-u_+(v)}{1+v}  - \frac{2}{\sqrt{m_1^2-4k(1+v)^2}} \le \frac{-u_+(v)}{1+v}.
\end{align*} 
Since $m_2>m_1>0$ and $v\ge0$, we have
\begin{align*}
	\frac{\mathrm{d}}{\mathrm{d}v} g(u_+(v),v) &= \frac{\mathrm{d}}{\mathrm{d}v} \left[-\mu v + u_+(v) \left( \frac{m_2}{m_1}v-v+\frac{m_2}{m_1} \right)\right]\\
	&= -\mu + \frac{\mathrm{d}}{\mathrm{d}v} u_+(v) \left(\frac{m_2}{m_1}(1+v)-v\right) + u_+(v) \left(\frac{m_2}{m_1}-1\right)\\
	&\le -\mu -  \frac{1}{1+v} u_+(v) \le - \mu <0.
\end{align*}
In view of \eqref{ADnegative}, we find a constant $c>0$ such that
\begin{align*}
	A_{\ast}(x) \le -c, \quad D_{\ast}(x) \le 0\quad\text{and}\quad A_{\ast}(x)D_{\ast}(x) - B_{\ast}(x) C_{\ast}(x) \ge c
\end{align*}
holds for a.e. $x\in \Omega$.
An application of Lemma \ref{Cor3.10Steffen} yields $s(\mathbf{L})<0$. Finally, Theorem \ref{stab} implies stability of the steady state $(\overline{u},\overline{v})$.
\end{proof}

\section*{Acknowledgments}

This work is supported by the German Research Foundation (DFG) under Germany’s Excellence Strategy EXC 2181/1 - 390900948 (the Heidelberg STRUCTURES Excellence Cluster) and through the Collaborative Research Center 1324 (SFB1324, project B6).




\begin{appendices}

\section{Heat semigroup} \label{sec:heatsemigroup} 

The preceding sections are based on properties of the heat semigroup. We define the heat semigroup $(S_\Delta(t))_{t \in \mathbb{R}_{\ge 0}}$ according to \cite{Davies, Ouhabaz} and study basic properties on the spaces $L^p(\Omega)$ for $1 \le p \le \infty$ and $C(\overline{\Omega})$. 

\begin{proposition} \label{heathom}
	Let $\Omega \subset \mathbb{R}^n$ be a bounded domain with $\partial \Omega \in C^{0,1}$, $z^0 \in L^2(\Omega)$. Then, there exists a unique variational solution $z \in C(\mathbb{R}_{\ge 0}; L^2 (\Omega ))$ of the homogeneous heat equation 
	\begin{align*}
		\frac{\partial z}{\partial t} - \Delta z & = 0 \quad \mbox{in} \quad   \Omega \times \mathbb{R}_{>0}, \qquad z(\cdot,0) = z^0   \quad \mbox{in} \quad   \Omega, \qquad
		\frac{\partial z}{\partial \mathbf{n}}  = 0 \quad \mbox{on} \quad  \partial \Omega \times \mathbb{R}_{>0}. 
	\end{align*}
	The solution is given by the heat semigroup $(S_\Delta(t))_{t \in \mathbb{R}_{\ge 0}} \subset \mathcal{L}(L^2(\Omega))$,
	\begin{align}
		z(x,t) = (S_\Delta(t) z^0)(x), \qquad x \in \Omega.  \label{heatsemigroup}
	\end{align}
	Moreover, the heat semigroup defined by \eqref{heatsemigroup} can be extended to a contraction semigroup $(S_{\Delta,p}(t))_{t \in \mathbb{R}_{\ge 0}}$ on $L^p(\Omega)$ for each $1 \le p \le \infty$, which is strongly continuous for $1 \le p < \infty$. For $1 < p <\infty$, the semigroup is analytic and $S_{\Delta, p}(t)$ is compact for $t>0$. Consequently, the generator $A_p: \mathcal{D}(A_p) \subset L^p(\Omega) \to L^p(\Omega)$ of $(S_{\Delta,p}(t))_{t \in \mathbb{R}_{\ge 0}}$ has compact resolvent for $1 < p < \infty$.
\end{proposition}

\begin{proof}
	Contractivity of the semigroup is shown in \cite[Theorem 1.3.9]{Davies}. By \cite[Theorems 1.4.1, 1.4.2]{Davies}, $(S_\Delta(t))_{t \in \mathbb{R}_{\ge 0}}$ is a strongly continuous semigroup for each $1 \le p < \infty$, which is even analytic for $p>1$. Compactness of the semigroup operators follows from \cite[Theorem 1.6.3]{Davies}. Finally, compactness of the resolvent $R(\lambda, A_p)$ for each $\lambda \in \rho(A_p)$ is a consequence of \cite[Ch. II, Theorem 4.29]{Engel}.
\end{proof}

For $A_2$, the domain $\mathcal{D}(A_2)$ is given in \cite{Davies} by the following characterization of a weak solution, i.e.,
\begin{align*}
	w \in \mathcal{D}(A_2) &\quad \Longleftrightarrow \quad w \in W^{1,2}(\Omega) \; \text{and there exists a function} \; f \in L^2(\Omega) \; \text{such that} \\[1ex]
	& \qquad \qquad \quad ( \nabla w, \nabla \varphi )_{L^2(\Omega)} = (f, \varphi )_{L^2(\Omega)} \qquad \forall \; \varphi \in W^{1,2}(\Omega).
\end{align*}
As usual in the weak sense, we identify $A_2w = f$ and obtain the graph norm estimate 
\begin{align*}
	\|w\|^2_{W^{1,2}(\Omega)} = \|w\|^2_{L^2(\Omega)} + (f,w)_{L^2(\Omega)} \le C\left(\|w\|^2_{L^2(\Omega)} + \|A_2w\|^2_{L^2(\Omega)} \right) \quad \forall \; w \in \mathcal{D}(A_2).
\end{align*}
From this characterization it is clear that the set 
\[
H_N^2(\Omega) := \{ w \in W^{2,2}(\Omega) \mid \partial_\mathbf{n} w = 0 \; \text{a.e. in} \; \Omega \}
\]
is included in $\mathcal{D}(A_2)$. In general, elliptic regularity is restricted for low-regular boundaries such as $\partial \Omega \in C^{0,1}$ \cite[Section 4.4.3]{Grisvard}. By the characterization of contraction semigroups via maximal dissipative generators in \cite[Theorem 3.4.5]{Arendt}
, we obtain $H_N^2(\Omega) \subsetneq \mathcal{D}(A_2)$ for certain Lipschitz boundaries. As shown in \cite[Theorem 2.15]{Yagi}, there holds $\mathcal{D}(A_2)=H_N^2(\Omega)$ in case of $\partial \Omega \in C^{1,1}$ since \cite[Theorem 2.4.1.3]{Grisvard} is used.\\

Let us next characterize the domain $\mathcal{D}(A_p)$ of the generator $A_p$ of the heat semigroup on $L^p(\Omega)$ for $p > 2$. We mention that, for Lipschitz domains, we might have elements in $\mathcal{D}(A_p)$ which are not in $W^{1,p}(\Omega)$, although these functions are H\"older continuous \cite{Nittka, Wood}.

\begin{lemma} \label{domchar}
	Let $2 \le p< \infty$ and the semigroup $(S_{\Delta,p}(t))_{t \in \mathbb{R}_{\ge 0}}$ be defined as in Proposition \ref{heathom} with generator $A_p: \mathcal{D}(A_p) \subset L^p(\Omega) \to L^p(\Omega)$. Then, $A_p$ is the part of the operator $A_2$ in $L^p(\Omega)$, i.e.,
	\begin{align}
		\mathcal{D}(A_p) = \{ w \in \mathcal{D}(A_2) \cap L^p(\Omega) \mid A_2w \in L^p(\Omega) \}. \label{DHppart}
	\end{align}
	Moreover, for $p>n/2$, the embedding $\mathcal{D}(A_p) \hookrightarrow C^{0, \gamma}(\overline{\Omega})$ is continuous for some $\gamma>0$ with 
	\begin{align}
		\|w\|_{C^{0,\gamma}(\overline{\Omega})} \le  C\left(\| w\|_{L^p(\Omega)} + \|A_pw\|_{L^p(\Omega)} \right) \qquad \forall \; w \in \mathcal{D}(A_p). \label{DHpembedding}
	\end{align}
\end{lemma}

\begin{proof} Since the semigroups on $L^p(\Omega)$ are consistent due to \cite[Theorem 1.4.1]{Davies}, we obtain that the heat semigroup on $L^p(\Omega)$ is the restriction of the semigroup $(S_{\Delta}(t))_{t \in \mathbb{R}_{\ge 0}}$ defined on $L^2(\Omega)$ onto $L^p(\Omega)$ for $p \ge 2$. Hence, \cite[Ch. 4, Theorem 5.5]{Pazy} shows equality \eqref{DHppart}.\\
	The domain $\mathcal{D}(A_2)$ is implicitly defined by an elliptic problem. Hence, we focus on the problem $A_2 w = f$ for $f \in L^p(\Omega)$ in order to prove an embedding for $\mathcal{D}(A_p)$ endowed with the graph norm. We apply \cite[Proposition 3.6]{Nittka} for $p>n/2$ to obtain an embedding into the H\"older space $C^{0, \gamma}(\overline{\Omega})$ for a certain $\gamma>0$ which depends on $\Omega$. Estimate \eqref{DHpembedding} follows from the continuous embedding $L^p(\Omega) \hookrightarrow L^2(\Omega)$ for $p \ge 2$.
\end{proof}

Next we consider a restriction of the Laplace operator to the space $L^\infty(\Omega)$, i.e., we define
\begin{align*}
	A_{\infty}: \mathcal{D}(A_{\infty}) \subset L^\infty(\Omega) \to L^\infty(\Omega)
\end{align*}
as the part of $A_p$ defined in Lemma \ref{domchar} in $L^\infty(\Omega)$. From
\[
\mathcal{D}(A_{\infty}) = \{ u \in \mathcal{D}(A_p) \cap L^\infty(\Omega) \mid A_p u \in L^\infty(\Omega)\}
\]
we infer that this operator is again closed as the generator $A_p$ of the heat semigroup is closed on $L^p(\Omega)$ \cite[Section 3.10, p. 184]{Arendt}. However, $A_{\infty}$ is not densely defined due to the inclusion $\mathcal{D}(A_p) \subset C(\overline{\Omega})$ from Lemma \ref{domchar}. This means that $(A_{\infty}, \mathcal{D}(A_{\infty}))$ cannot be a generator of a strongly continuous semigroup by \cite[Ch. II, Theorem 1.4]{Engel}. Note that continuity of the heat semigroup restricted to $L^\infty(\Omega)$ only fails in $t=0$ and also analyticity is preserved if we exclude $t=0$. In fact, following \cite[Corollary 3.1.24]{Lunardi}, there exists a corresponding resolvent estimate for $A_{\infty}$ as for sectorial operators. Unfortunately, this estimate is shown under higher regularity assumptions on the boundary. But using heat kernel estimates from \cite{Ouhabaz}, we derive a corresponding resolvent estimate next. 

\begin{lemma} \label{resolvestLaplace}
	The operator $(A_{\infty}, \mathcal{D}(A_{\infty}))$ is closed on $L^\infty(\Omega)$. Moreover, there exist an angle $\frac{\pi}{2} < \omega < \pi$, a closed sector $\Sigma_\omega :=  \{ \lambda \in \mathbb{C} \mid |\mathrm{arg} \, \lambda| \le \omega\}$ and a constant $M >0$ such that $\Sigma_\omega \setminus \{0\} \subset \rho(A_{\infty})$ holds with the resolvent estimate
	\begin{align*}
		\|R(\lambda, A_{\infty})\| \le \frac{M}{|\lambda|} \qquad \forall \; \lambda \in \Sigma_\omega \setminus \{0\}. 
	\end{align*}  
\end{lemma}

\begin{proof}
	It is well-known that there exists such a resolvent estimate for the spaces $L^p(\Omega)$ for $p< \infty$ \cite[Theorem 2.12]{Yagi}. Nevertheless, constants in these estimates deteriorate as $p \to \infty$. Hence, we choose another strategy which is used for continuous functions in \cite[Section 5.5]{Tanabe}. Therefore, let us start with the analytic heat semigroup $(S_\Delta(t))_{t \in \mathbb{R}_{\ge 0}}$ generated by the Laplace operator defined on $L^2(\Omega)$, see Proposition \ref{heathom}. \cite[Theorem 2.4.4]{Davies} shows that $(S_\Delta(t))_{t \in \mathbb{R}_{\ge 0}}$ is hypercontractive and maps $L^2(\Omega)$ to $L^\infty(\Omega)$ for each $t >0$. Thus for each $t>0$, $x \in \Omega$, the map $z \mapsto S_\Delta(t)z(x) \in L^2(\Omega)^\ast$ can be represented by a function $p(t,x, \cdot) \in L^2(\Omega)$, the so called heat kernel, such that
	\[
	(S_\Delta(t) z)(x) =  \int_\Omega p(t,x,y) z(y) \; \mathrm{d}y \qquad \forall \; z \in L^2(\Omega), x \in \Omega.
	\]
	We use an estimation of the heat kernel from \cite[Theorem 3.2.9]{Davies}, more precisely, there exist constants $C_1$, $C_2>0$ such that
	\[
	0 \le p(t,x,y) \le C_1 m(t)^{-n/2} \exp\left( \frac{-|x-y|^2}{C_2 t}\right) \qquad \forall \; x,y \in \Omega, t \in \mathbb{R}_{> 0},
	\]
	where $m(t) = \min\{1,t\}$. By \cite[Theorem 7.2]{Ouhabaz}, this estimate can be extended to some sector $\Sigma_{\theta_1} \setminus \{0\} \subset \mathbb{C}$ for $0 < \theta_1 < \frac{\pi}{2}$, where $\Sigma_{\theta_1} =  \{ \lambda \in \mathbb{C} \mid |\mathrm{arg} \, \lambda| \le \theta_1 \}$. Indeed, we have $\cos(\theta_1) \le \cos(\theta) \le 1$ for each argument $\theta=\mathrm{arg} \, z$ of $z \in \Sigma_{\theta_1}$ with $\mathrm{Re}\, z >0$. Hence, according to \cite[Theorem 7.2]{Ouhabaz}, for each $0 < \theta_1 < \frac{\pi}{2}$ there exist constants $C_1, C_2>0$ such that
	\begin{align}
		0 \le p(z,x,y) \le C_1 |z|^{-n/2} \exp\left( \frac{-|x-y|^2}{C_2 |z|}\right) \quad \forall \; x,y \in \Omega, z \in \Sigma_{\theta_1} \setminus \{0\}. \label{heatkernelest}
	\end{align}
	Such an estimate for the heat kernel can be used to obtain resolvent estimates, since
	\begin{align}
		R(\lambda, A_2) = \int_0^\infty \mathrm{e}^{-\lambda \tau} S_\Delta(\tau) \; \mathrm{d}\tau \quad \forall \; \lambda \in \mathbb{C}, \mathrm{Re} \, \lambda>0 \label{Laplacetransform}
	\end{align}
	holds for the generator $A_2$ of the heat semigroup $(S_\Delta(t))_{t \in \mathbb{R}_{\ge 0}}$ from Proposition \ref{heathom} \cite[Ch. II, Theorem 1.10]{Engel}. Such estimates are also derived in \cite[Theorem 5.7]{Tanabe} under higher regularity assumptions. Note that the spectrum of the Laplacian $A_p$ considered on $L^p(\Omega)$ is independent of $2 \le p < \infty$ by \cite[Example 1.2]{Arendtspec}, and there holds $R(\lambda, A_2)_{|L^p(\Omega)} = R(\lambda, A_p)$ for each $\lambda \in \rho(A_2)=\rho(A_p)$. Furthermore, as in the proof of Proposition \ref{specrestricted}, one can show $\rho(A_p)=\rho(A_{\infty})$ and $R(\lambda, A_{\infty})= R(\lambda, A_p)_{|L^\infty(\Omega)}$ for $p>n^\ast$. In this way, we use the Laplace transform \eqref{Laplacetransform} of the heat semigroup to deduce estimates of $R(\lambda, A_p)$ for each $2 \le p \le \infty$. \\
	Recall from \cite[Ch. II, Theorem 3.5]{Engel} that the contraction property of the heat semigroup $(S_\Delta(t))_{t \in \mathbb{R}_{\ge 0}}$ in $L^p(\Omega)$ with $p<\infty$ implies 
	\begin{align}
		\| R(\lambda, A_2)\|_{\mathcal{L}(L^p(\Omega))} \le \frac{1}{\mathrm{Re} \, \lambda} \quad \forall \; \lambda \in \mathbb{C}, \mathrm{Re} \, \lambda >0. \label{resolventest3}
	\end{align}
	Taking the limit $p \to \infty$ in this resolvent estimate \eqref{resolventest3}, we obtain the same estimate for the resolvent restricted to $L^\infty(\Omega)$, compare also \cite[Ch. 2, Section 2.2]{Yagi}. Hence, as in the proof (b) $\Rightarrow$ (c) of \cite[Ch. 2, Theorem 5.2]{Pazy}, it remains to verify a corresponding estimate for the imaginary part of $\lambda \in \mathbb{C}$ with $\mathrm{Re} \, \lambda >0.$ Such an estimate can be deduced as in the proof (a) $\Rightarrow$ (b) of \cite[Ch. 2, Theorem 5.2]{Pazy} by shifting the path of integration in \eqref{Laplacetransform} to a ray in the complex plane. Recall that $A_p$ is sectorial in the sense of \cite[Ch. II, Definition 4.1]{Engel} since the heat semigroup is an analytic contraction semigroup \cite[Ch. II, Theorem 4.6]{Engel}. Uniform boundedness of the semigroup with respect to $L^\infty(\Omega)$ in some closed sector in $\mathbb{C}$ can be deduced from \cite[Theorem 7.4]{Ouhabaz}, using heat kernel estimate \eqref{heatkernelest}. Following the proof of \cite[Ch. 2, Theorem 5.2]{Pazy}, we find a certain angle $0 <\delta < \pi/2$ and for each small $\varepsilon>0$ there is a constant $M_\varepsilon>0$ such that 
	\begin{align}
		\|R(\lambda, A_2)\|_{\mathcal{L}(L^\infty(\Omega))} \le \frac{M_\varepsilon}{|\lambda|} \quad \forall \; \lambda \in \Sigma_{\pi/2 + \delta-\varepsilon} \setminus \{0\}. \label{resolventest2}
	\end{align}
	This is analog to \cite[pp. 207-210]{Tanabe} and \cite[p. 216]{Tanabe} for continuous functions and $\theta_0 = \frac{\pi}{2}-\delta +\varepsilon$. Finally, recall that $\rho(A_2) = \rho(A_{\infty})$ and $R(\lambda, A_{\infty})= R(\lambda, A_2)_{|L^\infty(\Omega)}$ for $\lambda \in \Sigma_{\pi/2 + \delta-\varepsilon} \setminus \{0\}$. This shows that estimate \eqref{resolventest2} holds for $R(\lambda, A_{\infty})$. 
\end{proof}

Finally, the Laplace operator can be restricted to the space of uniformly continuous functions, see \cite{Arendtsurvey, Biegert} and \cite[Corollary 3.1.24]{Lunardi}. 

\begin{lemma} \label{LaplacianC}
	Let $A_2$ be the generator of the analytic heat semigroup $(S_\Delta(t))_{t \in \mathbb{R}_{\ge 0}}$ defined in Proposition \ref{heathom} on $L^2(\Omega)$ with domain $\mathcal{D}(A_2) \subset W^{1,2}(\Omega)$. Then the part of the Laplacian $A_2$ in $C(\overline{\Omega})$, denoted by $A_c: \mathcal{D}(A_c) \subset C(\overline{\Omega}) \to C(\overline{\Omega})$, generates a strongly continuous semigroup $(S_{\Delta,c}(t))_{t \in \mathbb{R}_{\ge 0}}$ of contractions which is even analytic on $C(\overline{\Omega})$. The domain of $A_c$ is defined by the restriction
	\begin{align}
		\mathcal{D}(A_c) &= \{ w \in \mathcal{D}(A_2) \cap C(\overline{\Omega}) \mid A_2 w \in C(\overline{\Omega})\}. \label{domainAc}
	\end{align}
	The operator $A_c$ is also the part of the operator $A_p$ defined in Lemma \ref{domchar} for $p \ge 2$. Moreover, the semigroup $(S_{\Delta,c}(t))_{t \in \mathbb{R}_{\ge 0}}$ is compact, i.e., $S_{\Delta,c}(t)$ is compact for each $t \in \mathbb{R}_{> 0}$. Consequently, the operator $A_c$ has compact resolvent on $C(\overline{\Omega})$.
\end{lemma}

\begin{proof}
	References and proofs can be found in \cite[Proposition 3, Theorem 4]{Biegert} and \cite[Ch. I, Section 2.6]{Arendtsurvey}. That $A_c$ is a restriction of the Laplace operator $A_p$ defined on $L^p(\Omega)$ is a consequence of elliptic regularity and the fact that $C(\overline{\Omega}) \subset L^p(\Omega)$ for all $p \ge 2$, compare Lemma \ref{domchar} and Lemma \ref{resolvestLaplace}. Finally, compactness of the resolvent $R(\lambda, A_c)$ for each $\lambda \in \rho(A_c)$ is a consequence of \cite[Ch. II, Theorem 4.29]{Engel}.
\end{proof}

\section{Spectral characterization of the reaction-diffusion-ODE operator}
\label{sec:specchar}
We already characterized the spectrum of the operator $\mathbf{L}$ in Proposition \ref{RDODEspec} by splitting the spectrum into the two components
\begin{align*}
	\sigma(\mathbf{L}) = \sigma (\mathbf{A}_{\ast}) \mathop{\dot{\cup}} \Sigma
\end{align*}
where $\Sigma$ is a set of eigenvalues and $\sigma (\mathbf{A}_{\ast})$ is the spectrum of the corresponding multiplication operator $\mathbf{A}_{\ast}$ on $L^\infty(\Omega)^m$. In the case where the resolvent set $\rho(\mathbf{A}_\ast)$ is connected, then \cite[Theorem 10.1.3 (i)]{Jeribi} additionally yields that $\Sigma$ is a discrete set of eigenvalues. However, when $\rho(\mathbf{A}_{\ast})$ is not a connected set, $\Sigma$ might include a non-discrete set of the point spectrum of $\mathbf{L}$ \cite[Example 4.3]{Hardt}. We aim to further characterize the set $\Sigma$ of eigenvalues of  $\mathbf{L}$ and to show that the spectral bound of $\mathbf{L}$ is determined by the spectral bound $s(\mathbf{A}_\ast)$ and the spectral bound of a discrete set of eigenvalues within the unbounded connected component $\Lambda_\infty$ of $\rho(\mathbf{A}_\ast)$.

\begin{proposition} \label{Sigma}
	Let $m, k \in \mathbb{N}$ and $\mathbf{D}^v \in \mathbb{R}_{>0}^{k \times k}$. Let $\mathbf{L}$ be the partly diffusive operator defined in \eqref{Linear2} on 
	$L^\infty(\Omega)^m \times L^p(\Omega)^{k} $ for bounded coefficient matrices $\mathbf{A}_{\ast}, \mathbf{B}_{\ast}, \mathbf{C}_{\ast}, \mathbf{D}_{\ast}$ and a finite $n^\ast < p<\infty$. Then 
	\[
	\Sigma:= \sigma(\mathbf{L}) \cap \rho(\mathbf{A}_{\ast}) = \sigma_d(\mathbf{L}) \; \dot{\cup} \;  \Sigma_0
	\]
	where $\sigma_d(\mathbf{L})$ is the discrete spectrum of $\mathbf{L}$ and $\Sigma_0$ is a union of bounded (open) connected components of $\rho(\mathbf{A}_{\ast})$ consisting of non-isolated eigenvalues of $\mathbf{L}$. In case that $\rho(\mathbf{A}_\ast)$ is a connected set in $\mathbb{C}$, then $\Sigma_0 = \emptyset$.
\end{proposition}

\begin{proof}
	By calculations in Corollary \ref{independence} we infer that, for $\lambda \in \rho(\mathbf{A}_{\ast})$, invertibility of $\lambda I - \mathbf{L}$ is equivalent to invertibility of the elliptic operator
	\begin{align*}
		\lambda I - \mathbf{D}^v \Delta -  \mathbf{D}_{\ast} - \mathbf{C}_{\ast}(\lambda I - \mathbf{A}_{\ast})^{-1} \mathbf{B}_{\ast}.
	\end{align*}
	To distinguish discrete parts within the spectral set $\Sigma =  \sigma(\mathbf{L}) \cap \rho(\mathbf{A}_\ast)$ we apply an analytic Fredholm theorem. For this reason we use that invertibility of this elliptic operator is equivalent to invertibility of  $I - \mathbf{C}(\lambda)$ where
	\begin{align}
		\mathbf{C}(\lambda)  =  R(\lambda_0, L_1)  \left( (\lambda_0- \lambda) I +  \mathbf{C}_{\ast}(\lambda I - \mathbf{A}_{\ast})^{-1} \mathbf{B}_{\ast} \right) \label{Fredholm2}
	\end{align}
	for some $\lambda_0 \in \rho(L_1)$, $L_1:= \mathbf{D}^v \Delta + \mathbf{D}_\ast$. From \cite[Theorem 1.6.3]{Davies} and \cite[Ch. II, Theorem 4.29]{Engel}, we obtain compactness of the resolvent $R(\lambda_0, L_1)$ for $L_1$ defined on $L^p(\Omega)^k$. Hence, $\mathbf{C}(\lambda)$ is compact for each $\lambda \in \rho(\mathbf{A}_{\ast})$ and $\mathbf{C}$ is analytic since the resolvent mapping of $\mathbf{A}_{\ast}$ is. 
	Note that for $\lambda \in \rho(\mathbf{A}_{\ast}) \cap \rho(L_1)$ the operator $\mathbf{C}(\lambda)$ can be chosen similar to \cite[Theorem 4.2]{Hardt}, i.e., 
	\[
	\mathbf{C}(\lambda)  =  R(\lambda, L_1)  \mathbf{C}_{\ast}(\lambda I - \mathbf{A}_{\ast})^{-1} \mathbf{B}_{\ast}.
	\]
	In the case where $\rho(\mathbf{A}_\ast)$ is connected, we have seen in Proposition \ref{RDODEspec} that $\Sigma_0 = \emptyset$. In the case where the open set $\rho(\mathbf{A}_\ast)$ is not connected, there exists an at most countable number of open connected components $\Lambda_i$, $i\in I$, which are pairwise disjoint and form a partition of $\rho(\mathbf{A}_{\ast})$. The components are chosen to be maximal, i.e. for all $i,j\in I$ with $i\ne j$ the set $\Lambda_i\cup \Lambda_j$ is not connected. Moreover, there exists exactly one unbounded connected component denoted by $\Lambda_\infty$ since $\sigma(\mathbf{A}_{\ast})$ is bounded. The remaining $\Lambda_i$, $i\in I\setminus\{\infty\}$, are bounded components. Now, we are able to apply the analytic Fredholm theorem on each of the components $\Lambda_i$. Either there holds $1 \in \sigma(\mathbf{C}(\lambda))$ for all $\lambda \in \Lambda_i$ or there exists  $\lambda \in \Lambda_i$ such that $1 \in \rho(\mathbf{C}(\lambda))$. \\
	In the first case, we recall from Proposition \ref{RDODEspec} that $\Lambda_i \subset \Sigma$ is a set of eigenvalues of $\mathbf{L}$. As $\Lambda_i$ is an open subset of $\mathbb{C}$, $\Lambda_i  \subset \sigma_{e6}(\mathcal{L})$ is a set of non-isolated eigenvalues, hence $\Lambda_i \cap \sigma_d(\mathbf{L}) = \emptyset$. We define the open set $\Sigma_0$ as the disjoint union of all such $\Lambda_i$. \\
	In the second case where $I - \mathbf{C}(\lambda)$ is invertible for some $\lambda \in \Lambda_i$, the Gohberg-Shmul'yan theorem \cite[Theorem 2.5.13]{Jeribi} applies and shows that $\Lambda_i \cap \Sigma$ is a discrete set of eigenvalues of $\mathbf{L}$.	Due to the separation of the connected components by the closed set $\sigma(\mathbf{A}_\ast)$, eigenvalues $\lambda \in \Lambda_i \cap \Sigma$ are also isolated values of $\sigma(\mathbf{L})$. Since $\lambda I - \mathbf{L}$ is a Fredholm operator of index zero, we find $\Lambda_i \cap \Sigma \subset \sigma_d(\mathbf{L})$ for each such component $\Lambda_i$. We recall from Proposition \ref{RDODEspec} that $\sigma(\mathbf{A}_\ast) \cap \sigma_d(\mathbf{L}) = \emptyset$ and, hence, we characterized all parts of $\sigma_d(\mathbf{L})$.\\
	Note that in any case, the unbounded connected component $\Lambda_\infty$ satisfies $\Lambda_\infty \cap \sigma_p(\mathbf{L}) \subset \sigma_d(\mathbf{L})$. In fact, since the multiplication operator $\mathbf{A}_{\ast}$ is bounded, the operator $I- \mathbf{C}(\lambda)$ resp. $\lambda I - \mathbf{L}$ is invertible for some large $\lambda \in \rho(\mathbf{L}) \cap \rho(\mathbf{A}_\ast)$ by \cite[Ch. II, Theorem 1.10]{Engel}.  
\end{proof}

When $\rho(\mathbf{A}_\ast)$ is connected, then $\rho(\mathbf{A}_\ast)$ coincides with the only unbounded connected component $\Lambda_\infty$. Then it is well-known that the spectral bound of the reaction-diffusion-ODE operator $\mathbf{L}$ is determined by the spectral bound of $\mathbf{A}_{\ast}$ and the spectral bound of the discrete part $\Sigma = \Lambda_\infty \cap \sigma_p(\mathbf{L})$ \cite{MKS17}. In general, the spectrum $\sigma(\mathbf{L})$ might consist of further discrete or non-discrete bounded components of $\rho(\mathbf{A}_\ast)$. However, their spectral bound can be controlled in terms of the spectral bound of the surrounding spectrum $\sigma(\mathbf{A}_\ast)$.

\begin{corollary}
	Let the assumptions of Proposition \ref{Sigma} hold true and denote by $\Lambda_\infty$ the unbounded open connected component of $\rho(\mathbf{A}_{\ast})$. Further, let us denote by 
	\[
	s_\infty := \sup\{\mathrm{Re} \, \lambda \mid \lambda \in \Lambda_\infty\cap\sigma_p(\mathbf{L})\}
	\]
	the spectral bound of the discrete set of eigenvalues within $\Lambda_\infty$.
	Then the spectral bound of $\mathbf{L}$ is determined by
	\begin{align}
		s(\mathbf{L}) = \max\{s(\mathbf{A}_{\ast}), s_\infty\}.
	\end{align}
\end{corollary}
\begin{proof}
	The spectral bound of $\mathbf{L}$ is given by $s(\mathbf{L}) = \sup\{\mathrm{Re}\, \lambda \mid \lambda\in \sigma(\mathbf{L})\}.$ Let us decompose $\rho(\mathbf{A}_\ast)$ again into open connected components $\Lambda \subset \mathbb{C}$, where exactly one of these is unbounded, denoted by $\Lambda_\infty$. We define 
	\[
	s_b= \sup\{\mathrm{Re} \, \lambda \mid \lambda \in \Lambda \cap\sigma_p(\mathbf{L}) \; \text{for some bounded connected component} \; \Lambda \; \text{of} \; \rho(\mathbf{A}_{\ast}) \}.
	\]
	By Proposition \ref{RDODEspec} and Proposition \ref{Sigma} it is clear that
	$
	s(\mathbf{L}) =	\max\{s(\mathbf{A}_{\ast}), s_\infty, s_b\}.
	$
	This proves $\max\{s(\mathbf{A}_{\ast}), s_\infty\} \le s(\mathbf{L})$. For the reverse inequality we show $s_b \le s(\mathbf{A}_{\ast})$.
	Therefore, let us take a bounded open connected component $\Lambda$ of the resolvent set $\rho(\mathbf{A}_{\ast})$ which is completely surrounded by $\sigma(\mathbf{A}_{\ast})$. Then $s(\Lambda) := \sup\{\mathrm{Re} \, \lambda \mid \lambda \in \Lambda\}$ cannot be attained within the open set $\Lambda$, however, the supremum is attained on the boundary $\partial \Lambda$ which is a subset of $\sigma(\mathbf{A}_{\ast})$. Thus, we have $s_b \le s(\mathbf{A}_{\ast})$ which proves the claim.
\end{proof}

Let us note that the discrete sets $\sigma_d(\mathbf{L})$ or $\Lambda_\infty \cap \sigma_p(\mathbf{L})$ are not necessarily closed. However, all {accumulation points in $\mathbb{C}$ lie on the boundary of the connected components in which the discrete sets are located. Since this boundary is a subset of $\partial \rho(\mathbf{A}_{\ast})$, we infer that these accumulation points of $\sigma_d(\mathbf{L})$ are} included in $\sigma(\mathbf{A}_{\ast})$. Compare also \cite[Section 2.1.2]{Klika} in the case $\Sigma_0 = \emptyset$.

\section{Existence of mild solutions} \label{sec:exmildsol}

In the case of continuous functions, existence and uniqueness of mild solutions for the corresponding equation \eqref{LNequation} is well established \cite[Ch. 6, Theorem 1.4]{Pazy}. In this section, we present a proper notion and an existence and uniqueness result of a mild solution to system \eqref{LNequation} with discontinuous components, i.e.,
\begin{align*}
	\frac{\partial \xi}{\partial t} = \mathbf{L} \xi + \mathbf{N}(\xi), \qquad \xi(0) = \xi^0 \in L^\infty(\Omega)^{m+k}. 
\end{align*}
This is in accordance with the theory developed in \cite{Rothe}. The analytic semigroup $(\mathbf{T}(t))_{t \in \mathbb{R}_{\ge 0}}$ from Lemma \ref{Lclosed} can be restricted to $L^\infty(\Omega)^{m+k}$ by Lemma \ref{restrsemigroup}. However, strong continuity of the semigroup is lost in the limit $p \to \infty$ \cite[Part I, Lemma 2]{Rothe}. Nevertheless, due to regularizing effects of the heat semigroup, a solution to equation \eqref{LNequation} can be obtained by a Picard iteration in $L^\infty(\Omega)^{m+k}$ \cite[Part II, Theorem 1]{Rothe}. The result in \cite{Rothe} is based on an integral representation 
\begin{align}
	\xi(\cdot,t) = \mathbf{S}(t)\xi^0 + \int_0^t \mathbf{S}(t-\tau) \big( \mathbf{N}(\xi(\cdot,\tau)) + \mathbf{J}\xi(\cdot, \tau) \big) \; \mathrm{d}\tau, \label{implicitS}
\end{align}
with the analytic contraction semigroup $(\mathbf{S}(t))_{t \in \mathbb{R}_{\ge 0}}$ generated by $\mathbf{D}\Delta$ on $L^\infty(\Omega)^{m} \times L^p(\Omega)^{k}$, see \eqref{semigroup}. In order to relate stability properties of the steady state to its linearization $\mathbf{L}$, it is necessary to obtain an integral representation with the analytic semigroup $(\mathbf{T}(t))_{t \in \mathbb{R}_{\ge 0}}$ generated by $\mathbf{L}$, i.e.,
\begin{align}
	\xi(\cdot,t) = \mathbf{T}(t)\xi^0 + \int_0^t \mathbf{T}(t-\tau) \mathbf{N}(\xi(\cdot,\tau)) \; \mathrm{d}\tau. \label{implicitT}
\end{align}
To show existence and uniqueness of a mild solution, we use estimate \eqref{estimateN} for the nonlinear term $\mathbf{N}$ as well as growth estimate \eqref{expgrowthTI} for the semigroup $(\mathbf{T}(t))_{t \in \mathbb{R}_{\ge 0}}$ on $L^\infty(\Omega)^{m+k}$. A solution of equation \eqref{LNequation} is understood analog to the mild sense of Rothe, compare \cite[Part II, Definition 2]{Rothe}.

\begin{definition} \label{mildsol}
	Let $0 < T \le \infty$. A mild solution of problem \eqref{LNequation} on the interval $[0, T)$ for initial datum $\xi^0 \in L^\infty(\Omega)^{m+k}$ is a measurable function
	$
	\xi: \Omega \times [0,T) \to \mathbb{R}^{m+k}
	$
	satisfying for all $t \in (0,T)$
	\begin{itemize}
		\item[(i)] $\xi(\cdot,t) \in L^{\infty}(\Omega)^{m+k}$ and \, $\sup_{s \in (0,t)} \|\xi(\cdot,s)\|_\infty < \infty$;
		
		\item[(ii)] the integral representation
		\[
		\xi(\cdot,t) = \mathbf{T}(t)\xi^0 + \int_0^t \mathbf{T}(t-\tau) \mathbf{N}(\xi(\cdot,\tau)) \; \mathrm{d}\tau,
		\]
	\end{itemize}
	where the integral is an absolutely converging Bochner integral in $L^{\infty}(\Omega)^{m+k}$.
\end{definition}

The local-in-time mild solution fulfills $\xi \in L^\infty(\Omega_T)^{m+k}$ for each finite $T < T_{\max}$, where $[0,T_{\max})$ is the maximal time interval of existence derived in the next Proposition \ref{Rothesol}. This is achieved via a Picard iteration in the Banach space $(L^\infty(\Omega)^{m+k},\|\cdot\|_\infty)$, similar to \cite[Part II, Theorem 1]{Rothe}. Boundedness of the solution and norm-continuity of the analytic semigroup $(\mathbf{T}(t))_{t \in \mathbb{R}_{\ge 0}}$ for $t>0$ yields a mild solution $\xi \in C([0, T_{\max}); L^\infty(\Omega)^m \times L^p(\Omega)^{k})$. The regularity of the semigroup for $p>n^\ast$ even implies $\xi \in C((0, T_{\max}); L^\infty(\Omega)^m \times C(\overline{\Omega})^{k})$, as a consequence of Lemma \ref{restrsemigroup}. We refer to Remark \ref{rem:regmildsol} for more comments on the regularity of solutions depending on the regularity of the initial condition.

\begin{proposition} \label{Rothesol}
	Let the nonlinear operator $\mathbf{N}$ defined in \eqref{LNequation} satisfy the local Lipschitz condition \ref{ass:Exist} and let the semigroup $(\mathbf{T}(t))_{t \in \mathbb{R}_{\ge 0}}$ defined in Lemma \ref{Lclosed} satisfy the growth estimate $\|\mathbf{T}(t)\|_\infty \le M \mathrm{e}^{w t}$ for some $M \ge 1, w \in \mathbb{R}$. Then the following assertions hold.
	\begin{itemize}
		
		\item[(i)] For all $\xi^0 \in L^{\infty}(\Omega)^{m+k}$ there exists a $0 < T \le \infty$, such that the initial value problem \eqref{LNequation} has a unique mild solution on the time interval $[0,T)$.
		
		\item[(ii)] If we consider $T=T(\xi^0)$ as a function of $\xi^0$, then it satisfies
		\[
		0 < \inf \{ T(\xi^0) \mid \xi^0 \in L^{\infty}(\Omega)^{m+k}, \|\xi^0\|_{\infty} \le U_0\}
		\]
		for all $U_0 \in [0,\infty)$.
		
		\item[(iii)] The existence interval can be chosen to be maximal, i.e., $\mathrm{(i)}$ holds for $T \le T_{\max}$, $T_{\max}  \le \infty$, and not for $T>T_{\max}$. If $T_{\max} < \infty$, the solution $\xi$ satisfies
		\[
		\|\xi(\cdot,t)\|_{\infty}  \to \infty  \qquad (t \nearrow T_{\max}).
		\]
	\end{itemize}
\end{proposition}

\begin{proof}
	Let $\xi^0 \in L^\infty(\Omega)^{m+k}$ satisfy $\|\xi^0\|_{\infty} \le U_0$ for some $U_0\ge0$.
	Choose an arbitrary real number $U >M U_0$ and consider the bounded set $B= \overline{\Omega} \times \overline{B_U(\mathbf{0})} \subset \mathbb{R}^n \times \mathbb{R}^{m+k}$ on which the nonlinearity is bounded and Lipschitz continuous with some constant $L(B)>0$ by Assumption \ref{ass:Exist}. Herein, we denote $B_U(\mathbf{0})$ the open ball with radius $U$ around $\mathbf{0} \in \mathbb{R}^{m+k}$. Choose $T>0$ such that 
	\begin{equation}
		M U_0 \mathrm{e}^{w T} + \mathrm{e}^{(L(B)M + w)T} -1 \le U, \label{choiceT}
	\end{equation}
	where this particular choice is due to the following proof. As the case $w \le 0$ is already considered in \cite[Part II, Theorem 1]{Rothe}, we assume $w \ge 0$ in the following.\\
	We use a Picard iteration to obtain the mild solution $\xi$. Define 
	$(\xi_i)_{i \in \mathbb{N}} \subset L^{\infty}(\Omega_T)^{m+k}$ recursively by
	\begin{align}
		\xi_1(\cdot,t)  = \mathbf{T}(t)\xi^0, \qquad
		\xi_{i+1}(\cdot,t)  = \mathbf{T}(t)\xi^0 + \int_0^t \mathbf{T}(t-s) \mathbf{N}(\xi_i(\cdot,s)) \; \mathrm{d}s \label{Rotheiteration}
	\end{align}
	for $i \in \mathbb{N}, t \in [0,T]$. Recall that we use in fact the restricted semigroup $(\mathbf{T}_\infty(t))_{t \in \mathbb{R}_{\ge 0}}$ on $L^{\infty}(\Omega)^{m+k}$ from Lemma \ref{restrsemigroup}. In order to show well-definedness of the approximating sequence $(\xi_i)_{i \in \mathbb{N}}$, define the differences 
	\[
	\eta_i: [0,T] \to \mathbb{R}_{\ge 0},\quad t \mapsto \|(\xi_{i+1}-\xi_i)(\cdot,t)\|_{\infty}
	\]
	for each $i \in \mathbb{N}$. Our aim is to show the following inequalities
	\begin{align} \label{R4}
		\begin{split}
			\|\xi_i(\cdot,t) \|_{\infty} & \le U, \\
			\eta_i(t) & \le \alpha \int_0^t \mathrm{e}^{w(t-s)} \eta_{i-1}(s) \; \mathrm{d}s \qquad \text{where} \quad \eta_0:= 1,\\
			\sum_{j=1}^i \eta_j(t)  & \le \mathrm{e}^{(\alpha + w)t} -1,\\
			\eta_i(t)  & \le \alpha \int_0^t \frac{(\alpha s)^{i-1}}{(i-1)!} \mathrm{e}^{w s} \; \mathrm{d}s
		\end{split}
	\end{align}
	for all $t \in [0,T], i \in \mathbb{N}$ and $\alpha:= M L(B)$ via induction. Norm-continuity of the semigroup $(\mathbf{T}_\infty(t))_{t \in \mathbb{R}_{> 0}}$ implies (Bochner) measurability of $\xi_1$ by \cite[Corollary 1.1.2]{Arendt}. The growth estimate \eqref{expgrowthTI} of the semigroup $(\mathbf{T}_\infty(t))_{t \in \mathbb{R}_{\ge 0}}$ yields 
	\[
	\|\xi_1(\cdot,t)\|_{\infty} = \|\mathbf{T}(t)\xi^0\|_{\infty} \le M\mathrm{e}^{w t} U_0 <U \qquad \text{for} \quad t\in [0,T],
	\]
	where we used the choice of $T$ in \eqref{choiceT}. Continuity of $\mathbf{N}$ implies that $\mathbf{N} \circ \xi_1$ is also (Bochner) measurable. Estimates for the nonlinearity from Assumption \ref{ass:Exist} lead us to integrability, i.e., $s \mapsto \mathbf{N}(\xi_1(\cdot,s)) \in L^1([0,T]; L^{\infty}(\Omega)^{m+k})$. By \cite[Proposition 1.3.4]{Arendt} and continuity of the semigroup from Lemma \ref{restrsemigroup}, the convolution in definition of $\xi_2$ is well-defined as a Bochner integral in $L^{\infty}(\Omega)^{m+k}$ or $L^\infty(\Omega)^m \times C(\overline{\Omega})^k$ with the estimate
	\begin{align*}
		\eta_1(t) &\le \int_0^t M \mathrm{e}^{w(t-s)} \|\mathbf{N}(\xi_1(\cdot,s))\|_{\infty} \; \mathrm{d}s \le L(B) M \int_0^t \mathrm{e}^{w(t-s)} \; \mathrm{d}s = \alpha \int_{0}^{t} \mathrm{e}^{ws} \; \mathrm{d}s\\
		&\le \alpha \int_0^t \mathrm{e}^{(\alpha+w)s} \; \mathrm{d}s  \le \mathrm{e}^{(\alpha + w)t}-1
	\end{align*}
	since $\alpha =M L(B) \ge 0$ and $ w \ge 0$.
	This proves the base case for estimates \eqref{R4}. If the estimates hold true for a particular $i\in \mathbb{N}$, this implies
	\begin{align*}
		\eta_{i+1}(t) & \le \int_0^t M \mathrm{e}^{w(t-s)} \|\mathbf{N}(\xi_{i+1}(\cdot,s))-\mathbf{N}(\xi_i(\cdot,s)) \|_{\infty} \; \mathrm{d}s\\
		& \le \alpha \int_0^t \mathrm{e}^{w(t-s)} \eta_i(s) \; \mathrm{d}s \le\alpha^2 \int_0^t \mathrm{e}^{w(t-s)} \int_0^s \frac{(\alpha r)^{i-1}}{(i-1)!} \mathrm{e}^{w r} \; \mathrm{d}r \; \mathrm{d}s\\
		& = \frac{\alpha^2}{w} \mathrm{e}^{w t} \int_0^t  \frac{(\alpha r)^{i-1}}{(i-1)!}  \; \mathrm{d}r - \frac{\alpha^2}{w} \int_0^t \frac{(\alpha r)^{i-1}}{(i-1)!} \mathrm{e}^{w r} \; \mathrm{d}r = \alpha \int_0^t \frac{(\alpha s)^{i}}{i!} \mathrm{e}^{w s} \; \mathrm{d}s,
	\end{align*}
	where we used partial integration twice. We demonstrated only the case $w>0$, $w=0$ is even simpler.
	By definition of $\xi_{i+1}$ in \eqref{Rotheiteration} and $T$ in \eqref{choiceT}, we obtain
	\begin{align*}
		\|\xi_{i+1}(\cdot,t)\|_{\infty} &\le M\mathrm{e}^{wt}U_0 + ML(B) \int_{0}^{t} \mathrm{e}^{w(t-s)}\; \mathrm{d}s \\
		& \le M\mathrm{e}^{wT}U_0 + ML(B) \int_{0}^{t} \mathrm{e}^{(ML(B)+w)s}\; \mathrm{d}s\\
		&\le M\mathrm{e}^{wT}U_0 + \mathrm{e}^{(L(B)M+w)T} - 1 \le U
	\end{align*}
	for all $t\in[0,T]$, since $\alpha=ML(B) \ge 0$ and $w \ge 0$. Finally, it is left to prove the estimate for the sum of $\eta_j$ in \eqref{R4}. The induction step follows from
	\begin{align*}
		\sum_{j=1}^{i+1} \eta_j(t) \le  \alpha \int_0^t \sum_{j=1}^{i+1}\frac{(\alpha s)^{j-1}}{(j-1)!} \mathrm{e}^{w s} \; \mathrm{d}s \le \alpha \int_0^t \mathrm{e}^{(\alpha+ w) s} \; \mathrm{d}s  \le \mathrm{e}^{(\alpha + w)t}-1.
	\end{align*}
	For fixed $t \in [0,T]$, these results imply convergence of 
	\[
	\xi_{i+1}(\cdot,t) = \sum_{j=1}^i (\xi_{j+1}-\xi_j)(\cdot,t) + \xi_1(\cdot,t) \to \xi(\cdot,t):= \sum_{j \in \mathbb{N}} (\xi_{j+1}-\xi_j)(\cdot,t) + \xi_1(\cdot,t)
	\]
	as $i \to \infty$. The latter series converges absolutely in $L^{\infty}(\Omega)^{m+k}$, i.e.,
	\[
	\sum_{j \in \mathbb{N}} \|(\xi_{j+1}-\xi_j)(\cdot,t)\|_{\infty} = \sum_{j \in \mathbb{N}} \eta_j(t) \le \mathrm{e}^{(\alpha + w)t}-1,
	\]
	by convergence of the partial sums in \eqref{R4}. We even have convergence $\xi_i \to \xi$ in the Banach space $L^{\infty}(\Omega_T)^{m+k}$ and $L^\infty([0,T]; L^\infty(\Omega)^{m+k})$. This is a consequence of the estimates
	\begin{align*}
		\sup_{t \in [0,T]} \|(\xi_i-\xi)(\cdot,t)\|_{\infty} & \le \sup_{t \in [0,T]} \sum_{j=i}^{\infty} \eta_j(t) \le \sup_{t \in [0,T]} \sum_{j=i}^{\infty} \alpha \int_{0}^{t} \frac{(\alpha s)^{j-1}}{(j-1)!} \mathrm{e}^{ws} \; \mathrm{d}s\\
		&  \le \alpha \int_{0}^{T} \left(\mathrm{e}^{\alpha s} - \sum_{j=0}^{i-2} \frac{(\alpha s)^{j}}{(j)!}\right) \mathrm{e}^{ws} \; \mathrm{d}s \to 0 \quad (i \to \infty).
	\end{align*}
	where the exponential series converges uniformly on the bounded interval $[0,T]$.
	Using the above estimate of $\xi_1$ and inequalities \eqref{R4}, we obtain an estimate for $\xi$,
	\[
	\sup_{t \in [0,T]} \|\xi(\cdot,t)\|_{\infty} \le \sup_{t \in [0,T]} \left( \mathrm{e}^{(\alpha+ w)t} -1 + M\mathrm{e}^{w t} U_0 \right) \le U.
	\]
	Due to uniform convergence, the limit $\xi$ is (Bochner) measurable and a mild solution of the integral equation
	\[
	\xi(\cdot,t) = \mathbf{T}(t)\xi^0 + \int_0^t \mathbf{T}(t-s) \mathbf{N}(\xi(\cdot,s)) \; \mathrm{d}s \qquad \forall \; t \in [0,T].
	\]
	Uniqueness of mild solutions follows from Gronwall's inequality and local Lipschitz continuity from Assumption \ref{ass:Exist}. For the proof of (ii) and (iii), we refer to \cite[p. 114]{Rothe}. 
\end{proof}

Equivalence of the implicit integral formulations \eqref{implicitS} and \eqref{implicitT} can be shown using \cite[Equation (4.100)]{Webb} for $X=L^p(\Omega)^{m+k}$. Indeed, by \cite[Section 5.2.2]{Kowall}, $\mathbf{D} \Delta$ and $\mathbf{D} \Delta + \mathbf{J}$ generate analytic semigroups on $L^p(\Omega)^{m+k}$ and it can be verified that its restrictions to the subspace $L^\infty(\Omega)^m \times L^p(\Omega)^k$ coincide with $(\mathbf{S}(t))_{t \in \mathbb{R}_{\ge 0}}$ defined in \eqref{semigroup} and $(\mathbf{T}(t))_{t \in \mathbb{R}_{\ge 0}}$, respectively \cite[Ch. 4, Theorem 5.5]{Pazy}. Continuity of $\mathbf{N}$ from Lemma \ref{nonlinearity} implies that \[
h:= \mathbf{N} \circ \xi \in C([0,T]; L^p(\Omega)^{m+k})
\]
for a solution $\xi$ of problem \eqref{LNequation} and  \cite[Equation (4.100)]{Webb} applies to the bounded perturbation $\mathbf{J}$ on $L^p(\Omega)^{m+k}$. 

\section{Multiplication operator on uniformly continuous functions} 

Each matrix-valued function $\mathbf{A} \in C(\overline{\Omega})^{m \times m}$ induces a corresponding multiplication operator
\[
\mathbf{M}_{\mathbf{A}}: C(\overline{\Omega})^{m} \to C(\overline{\Omega})^{m}, \qquad \mathbf{z} \mapsto \mathbf{A} \mathbf{z}
\]
where $(\mathbf{A}\mathbf{z})(x) :=\mathbf{A}(x)\mathbf{z}(x)$ for each $\mathbf{z} \in C(\overline{\Omega})^{m}$. Since $\|\mathbf{M}_{\mathbf{A}}\| \le \|\mathbf{A}\|_\infty$, this is a bounded, linear operator. Let us simply write $\mathbf{A}$ instead of $\mathbf{M}_{\mathbf{A}}$ in the following. The knowledge of the spectrum of a multiplication operator $\mathbf{A}$ allows us to characterize the spectrum of the partly diffusive operator $\mathbf{L}^c$ in Proposition \ref{RDODEspecC}.

\begin{lemma} \label{essspecC}
	Let $\mathbf{A} \in C(\overline{\Omega})^{m \times m}$ for $m \in \mathbb{N},$ and let $\mathbf{A}$ denote its corresponding bounded multiplication operator on $C(\overline{\Omega})^{m}$. Then the spectrum $\sigma(\mathbf{A})$ is essential in the sense that each $\lambda \in \sigma(\mathbf{A})$ is characterized by the property 
	\begin{align}
		\forall \, \varepsilon>0 \; \exists \, \Omega_\varepsilon \subset \Omega \; \text{open set} \; \exists \,  \mathbf{v}_\varepsilon  \in \mathbb{R}^m \setminus \{\mathbf{0}\} :  \|(\lambda I - \mathbf{A}) \mathbf{v}_{\varepsilon}\|_{C(\overline{\Omega_\varepsilon})^m} \le \varepsilon |\mathbf{v}_{\varepsilon}|. \label{essspectrumC2}
	\end{align}
	This means $\sigma(\mathbf{A}) =   \{ \lambda \in \mathbb{C} \mid \text{condition} \;  \eqref{essspectrumC2} \; \text{is satisfied} \} =: A_{\mathrm{ess}}$ and, furthermore, $\sigma(\mathbf{A})$ consists of approximate eigenvalues only.
\end{lemma}

\begin{proof}
	First, let us verify $A_{\mathrm{ess}} \subset \sigma_{ap}(\mathbf{A})$. We assume $\lambda \in A_{\mathrm{ess}}$ and consider the sequence $(\xi_\varepsilon)_{\varepsilon >0} \subset C(\overline{\Omega})^m$ defined by
	$
	\xi_\varepsilon =  \eta_\varepsilon \mathbf{v}_\varepsilon/|\mathbf{v}_\varepsilon|
	$
	for a $\eta_\varepsilon \in C_c^\infty(\Omega_\varepsilon)$ with $0 \le \eta_\varepsilon \le 1$ and $\eta_\varepsilon \equiv 1$ on some small ball $\overline{B_\delta(y)} \subset \Omega_\varepsilon$ \cite[Lemma 9.3]{Brezis}. This construction implies $\|\xi_\varepsilon\|_{C(\overline{\Omega})^m} =1$ and 
	$
	\| (\lambda I - \mathbf{A}) \xi_\varepsilon \|_{C(\overline{\Omega})^m} =  \| (\lambda I - \mathbf{A}) \xi_\varepsilon \|_{C(\overline{\Omega_\varepsilon})^m} \le \varepsilon. 
	$
	Thus, $\lambda \in A_{\mathrm{ess}}$ is also an approximate eigenvalue of $\mathbf{A}$.\\
	In order to show $\sigma(\mathbf{A}) \subset  A_{\mathrm{ess}}$, we verify the converse inclusion $\mathbb{C} \setminus A_{\mathrm{ess}} \subset \rho(\mathbf{A})$. For $\lambda \notin A_{\mathrm{ess}}$, we find a constant $\varepsilon>0$ such that for all open sets $\Omega_0 \subset \Omega$ and vectors $\mathbf{v} \in \mathbb{R}^m \setminus \{\mathbf{0}\}$ there holds
	\[
	\|(\lambda I - \mathbf{A})(\cdot) \mathbf{v}\|_{C(\overline{\Omega_0})^m} \ge \varepsilon |\mathbf{v}|.
	\]
	By contradiction, this implies the pointwise estimate
	\[
	|(\lambda I - \mathbf{A})(x) \mathbf{v} | \ge \varepsilon |\mathbf{v}| \qquad \forall \; x \in \overline{\Omega}, \mathbf{v} \in \mathbb{R}^m. 
	\] 
	Consequently, the matrices $(\lambda I - \mathbf{A}(x))^{-1}$ imply a bounded multiplication operator on $C(\overline{\Omega})^m$ which turns out to be the inverse of $\lambda I - \mathbf{A}$, hence $\lambda \in \rho(\mathbf{A})$. 
\end{proof}

\begin{proposition} \label{specunion}
	Let $\mathbf{A} \in C(\overline{\Omega})^{m \times m}$ for $m \in \mathbb{N}$. Then the spectrum of the bounded multiplication operator $\mathbf{A}$ on $C(\overline{\Omega})^{m}$ is given by  
	\begin{align}
		\sigma(\mathbf{A}) =  \overline{ \bigcup_{x \in \Omega} \sigma(\mathbf{A}(x))}. 
		\label{specmultC}
	\end{align}
\end{proposition}

\begin{proof}
	In order to verify equality \eqref{specmultC}, we use a continuity argument. Recall that $\sigma(\mathbf{A})$ is a closed set in $\mathbb{C}$ by \cite[Corollary B.3]{Arendt}. In view of the closedness of $A_{\mathrm{ess}}=\sigma(\mathbf{A})$, it suffices to show
	\[
	\bigcup_{x \in \Omega} \sigma(\mathbf{A}(x)) \subset A_{\mathrm{ess}} 
	\]
	to obtain one inclusion in \eqref{specmultC}. For this reason consider $\lambda \in \sigma(\mathbf{A}(x_0))$ for some $x_0 \in \Omega$. Then there exists a normalized eigenvector $\mathbf{v}_0 \in \mathbb{R}^m$ with $|\mathbf{v}_0|=1$ such that $(\lambda I - \mathbf{A}(x_0))\mathbf{v}_0 = \mathbf{0}$. By continuity, we obtain a small neighborhood $\Omega_\varepsilon \subset \Omega$ of $x_0$ such that
	\[
	\sup_{x \in \Omega_\varepsilon} | (\lambda I - \mathbf{A}(x)) \mathbf{v}_{0}|  \le \varepsilon = \varepsilon |\mathbf{v}_0|,
	\]
	hence $\lambda \in A_{\mathrm{ess}}$.\\
	For the reverse inclusion in \eqref{specmultC}, we show that $\lambda \notin  \overline{ \bigcup_{x \in \Omega} \sigma(\mathbf{A}(x)) }$ implies $\lambda\notin A_{\mathrm{ess}}$
	. Indeed, if $\lambda$ is not in the closed set $\overline{ \bigcup_{x \in \Omega} \sigma(\mathbf{A}(x)) }$ and $x \in \Omega$, every matrix $\lambda I - \mathbf{A}(x)$ is invertible in $\mathbb{R}^{m \times m}$, i.e., we find a constant $C_x>0$ 
	such that
	\[
	|( \lambda I - \mathbf{A}(x)) \mathbf{v}| \ge C_x |\mathbf{v}| \qquad \forall \; \mathbf{v} \in \mathbb{R}^m.
	\]
	Let us verify that there exists a lower bound $C_x  \ge C > 0$ which is independent of $x \in \Omega$. We assume to the contrary that there exists a sequence $(x_{\ell})_{\ell \in \mathbb{N}} \subset \Omega$ and a sequence of normalized eigenvectors $(\mathbf{v}_{\ell})_{\ell \in \mathbb{N}} \subset \mathbb{R}^m$ with $|\mathbf{v}_\ell|=1$ satisfying $|( \lambda I - \mathbf{A}(x_\ell)) \mathbf{v}_\ell| \to 0$ as $\ell \to \infty$. However, this yields a contradiction, since $\lambda$ is arbitrarily close to the spectrum of the matrices $\mathbf{A}(x_\ell)$ as $\ell \to \infty$. Uniformity of the constant $C>0$ implies that $\lambda I - \mathbf{A}$ is bounded from below on $C(\overline{\Omega})^m$ and, thus, $\lambda\notin A_{\mathrm{ess}}$ by definition \eqref{essspectrumC2}. 
\end{proof}

In order to characterize the spectrum of the part of the operator $\mathbf{L}$ on $C(\overline{\Omega})^{m+k}$ in Proposition \ref{RDODEspecC}, we use again that the spectrum of the multiplication operator $\mathbf{A}_\ast$ is essential in the sense of Wolf.
This is a consequence of the following refined characterization of the spectrum of the multiplication operator $\mathbf{A}$.

\begin{proposition}
\label{essspectrumC}
Let $\mathbf{A} \in C(\overline{\Omega})^{m \times m}$ for $m \in \mathbb{N},$ and let $\mathbf{A}$ denote its corresponding bounded multiplication operator on $C(\overline{\Omega})^{m}$. Then the point spectrum of $\mathbf{A}$ is given by\begin{equation}
	\sigma_p(\mathbf{A}) = \{  \lambda \in \mathbb{C} \mid \Gamma_\lambda \; \text{has non-empty interior} \; \Gamma_\lambda^\circ \}, \label{pointspecC}
\end{equation}
where the compact set $\Gamma_\lambda$ is defined by
$
\Gamma_\lambda = \{x \in \overline{\Omega} \mid \det(\lambda I -\mathbf{A}(x))=0\}. 
$
Moreover, the spectrum $\sigma(\mathbf{A})$ is essential in the sense of Wolf, i.e., \[
\sigma(\mathbf{A})= \sigma_{\mathrm{ess}} (\mathbf{A}) := \{\lambda \in \mathbb{C} \mid \lambda I - \mathbf{A} \; \text{is not a  Fredholm operator} \}.
\]
\end{proposition}

\begin{proof} 
Note that $\Gamma_\lambda$ is compact by continuity of the determinant and entries of the matrix $\mathbf{A}(x)$. First, let us verify characterization \eqref{pointspecC} for $\sigma_p(\mathbf{A})$. On the one hand, let $\lambda \in \sigma_p(\mathbf{A})$ with a corresponding eigenfunction $\mathbf{f} \in C(\overline{\Omega})^m \setminus \{0\}$. Then $\mathbf{f}(x) \not=\mathbf{0}$ for some $x \in \overline{\Omega}$. By continuity of $\mathbf{f}$ and (path-)connectedness of $\Omega \subset \mathbb{R}^n$, we may choose a non-empty open set $\tilde{\Omega} \subset \Omega$ for which $\mathbf{f}(x) \not= \mathbf{0}$ holds for all $x \in \tilde{\Omega}$. Then, for each $x \in \tilde{\Omega}$, the vector $\mathbf{f}(x) \in \mathbb{C}^m \setminus \{\mathbf{0}\}$ is an eigenvector of the matrix $\lambda I - \mathbf{A}(x)$. Hence, $\tilde{\Omega} \subset \Gamma_\lambda$ and $\Gamma_\lambda$ has non-empty interior with $\tilde{\Omega} = \tilde{\Omega}^\circ \subset \Gamma_\lambda^\circ$.\\
On the other hand, if $\lambda \in \mathbb{C}$ satisfies $\Gamma_\lambda^\circ \not= \emptyset$, we can construct eigenvectors $\mathbf{v}(x)$ corresponding to $\mathbf{A}(x)$ and the constant eigenvalue $\lambda$ such that the function $\mathbf{v}$ is 
continuous on some closed ball $\overline{B_\delta(y)} \subset \Gamma_\lambda^\circ$. This follows from Gaussian elimination using elementary row operations which consists of multiplying, adding or swapping continuous entries. Once we obtained such eigenvectors, we use the idea of proof of \cite[Theorem 2.5]{Hardt}. For arbitrary $\phi \in C(\overline{B_\delta(y)})$ with $\phi \not\equiv 0$ on the open ball $B_{\delta/2}(y)$, we define the function
\begin{align*}
	\mathbf{f}: \overline{\Omega} \to \mathbb{R}^m \setminus \{\mathbf{0}\}, \quad \mathbf{f}(x) = \begin{cases}
		\mathbf{v}(x) \phi(x) \eta(x) & \text{for} \quad x \in B_\delta(y),\\
		\mathbf{0} & \text{else}
	\end{cases}
\end{align*}
for a cut-off function $\eta \in C_c^\infty(B_\delta(y))$ with $\eta_{|B_{\delta/2}(y)} \equiv 1$. Then, $\mathbf{f} \in C(\overline{\Omega})^m \setminus \{\mathbf{0}\}$ satisfies $(\lambda I - \mathbf{A})\mathbf{f} = \mathbf{0}$, i.e., $\lambda \in \sigma_p(\mathbf{A})$.

It remains to show that $\lambda I - \mathbf{A}$ is not Fredholm for all $\lambda \in \sigma(\mathbf{A})$.\\
If $\lambda \in \sigma_p(\mathbf{A})$, we infer $\sigma_p(\mathbf{A}) \subset \sigma_{\mathrm{ess}}(\mathbf{A})$ from an infinite-dimensional kernel of $\lambda I - \mathbf{A}$, analog to \cite[Proposition 3.2]{Hardt}. Indeed, from above calculations we see that this kernel contains a subspace which is isomorphic to 
$C(\overline{B_{\delta/2}(y)})$ and, hence, is infinite-dimensional.\\
If $\lambda \in \sigma(\mathbf{A}) \setminus \sigma_p(\mathbf{A})$, the injective operator $\lambda I -\mathbf{A}$ cannot be bounded from below since $\lambda \in \sigma_{ap}(\mathbf{A})$ by Lemma \ref{essspecC}. Thus, $\lambda I - \mathbf{A}$ cannot have closed range by \cite[Theorem 2.19, Remark 18]{Brezis} and $\lambda \in \sigma_{\mathrm{ess}}(\mathbf{A})$. 
\end{proof}

\end{appendices}

\end{document}